\documentclass[11pt]{amsart}

\usepackage[all]{xy}
\usepackage{color, pstricks}
\usepackage{mathrsfs}
\usepackage{amsmath,amsthm, amssymb, amsgen,amsxtra,amsfonts, amsbsy} 
\usepackage[all]{xy}
\usepackage{verbatim}
\usepackage{url}
\usepackage{hyperref}

\usepackage{array,multirow,graphicx}

\DeclareMathOperator{\Ker}{Ker}
\DeclareMathOperator{\Coker}{Coker}
\DeclareMathOperator{\Image}{Im}
\DeclareMathOperator{\Frac}{Frac}
\DeclareMathOperator{\Stab}{Stab}
\DeclareMathOperator{\Ann}{Ann}
\DeclareMathOperator{\length}{length}

\DeclareMathOperator{\Fix}{Fix}
\newcommand{\floor}[1]{\lfloor #1 \rfloor}
\newcommand{\ceil}[1]{\lceil #1 \rceil}

\begin{document}
%
%
%
\theoremstyle{definition}
\newtheorem{Definition}{Definition}[section]
\newtheorem*{Definitionx}{Definition}
\newtheorem{Convention}[Definition]{Convention}
\newtheorem{Construction}{Construction}[section]
\newtheorem{Example}[Definition]{Example}
\newtheorem{Examples}[Definition]{Examples}
\newtheorem{Remark}[Definition]{Remark}
\newtheorem*{Remarkx}{Remark}
\newtheorem{Remarks}[Definition]{Remarks}
\newtheorem{Caution}[Definition]{Caution}
\newtheorem{Conjecture}[Definition]{Conjecture}
\newtheorem*{Conjecturex}{Conjecture}
\newtheorem{Question}[Definition]{Question}
\newtheorem{Questions}[Definition]{Questions}
\newtheorem{RDP}[Definition]{Example}
\newtheorem*{Acknowledgements}{Acknowledgements}
\newtheorem*{Organization}{Organization}
\newtheorem*{Disclaimer}{Disclaimer}
\theoremstyle{plain}
\newtheorem{Theorem}[Definition]{Theorem}
\newtheorem*{Theoremx}{Theorem}
\newtheorem{Theoremy}{Theorem}
\newtheorem{Proposition}[Definition]{Proposition}
\newtheorem*{Propositionx}{Proposition}
\newtheorem{Lemma}[Definition]{Lemma}
\newtheorem{Corollary}[Definition]{Corollary}
\newtheorem*{Corollaryx}{Corollary}
\newtheorem{Fact}[Definition]{Fact}
\newtheorem{Facts}[Definition]{Facts}
\newtheorem{Claim}[Definition]{Claim}
\newtheoremstyle{voiditstyle}{3pt}{3pt}{\itshape}{\parindent}%
{\bfseries}{.}{ }{\thmnote{#3}}%
\theoremstyle{voiditstyle}
\newtheorem*{VoidItalic}{}
\newtheoremstyle{voidromstyle}{3pt}{3pt}{\rm}{\parindent}%
{\bfseries}{.}{ }{\thmnote{#3}}%
\theoremstyle{voidromstyle}
\newtheorem*{VoidRoman}{}

%
\newcommand{\prf}{\par\noindent{\sc Proof.}\quad}
\newcommand{\blowup}{\rule[-3mm]{0mm}{0mm}}
\newcommand{\cal}{\mathcal}
\newcommand{\Aff}{{\mathds{A}}}
\newcommand{\BB}{{\mathds{B}}}
\newcommand{\CC}{{\mathds{C}}}
\newcommand{\EE}{{\mathds{E}}}
\newcommand{\FF}{{\mathds{F}}}
\newcommand{\GG}{{\mathds{G}}}
\newcommand{\HH}{{\mathds{H}}}
\newcommand{\NN}{{\mathds{N}}}
\newcommand{\ZZ}{{\mathds{Z}}}
\newcommand{\PP}{{\mathds{P}}}
\newcommand{\QQ}{{\mathds{Q}}}
\newcommand{\RR}{{\mathds{R}}}
\newcommand{\Liea}{{\mathfrak a}}
\newcommand{\Lieb}{{\mathfrak b}}
\newcommand{\Lieg}{{\mathfrak g}}
\newcommand{\Liem}{{\mathfrak m}}
\newcommand{\ideala}{{\mathfrak a}}
\newcommand{\idealb}{{\mathfrak b}}
\newcommand{\idealg}{{\mathfrak g}}
\newcommand{\idealm}{{\mathfrak m}}
\newcommand{\idealn}{{\mathfrak n}}
\newcommand{\idealp}{{\mathfrak p}}
\newcommand{\idealq}{{\mathfrak q}}
\newcommand{\idealI}{{\cal I}}
\newcommand{\lin}{\sim}
\newcommand{\num}{\equiv}
\newcommand{\dual}{\ast}
\newcommand{\iso}{\cong}
\newcommand{\homeo}{\approx}
\newcommand{\mathds}[1]{{\mathbb #1}}
\newcommand{\mm}{{\mathfrak m}}
\newcommand{\pp}{{\mathfrak p}}
\newcommand{\qq}{{\mathfrak q}}
\newcommand{\rr}{{\mathfrak r}}
\newcommand{\pP}{{\mathfrak P}}
\newcommand{\qQ}{{\mathfrak Q}}
\newcommand{\rR}{{\mathfrak R}}
%
%
\newcommand{\OO}{{\cal O}}
\newcommand{\calA}{{\cal A}}
\newcommand{\calO}{{\cal O}}
\newcommand{\calU}{{\cal U}}
\newcommand{\numero}{{n$^{\rm o}\:$}}
\newcommand{\mf}[1]{\mathfrak{#1}}
\newcommand{\mc}[1]{\mathcal{#1}}
\newcommand{\into}{{\hookrightarrow}}
\newcommand{\onto}{{\twoheadrightarrow}}
\newcommand{\Spec}{\mathrm{Spec}\:}
\newcommand{\BigSpec}{{\rm\bf Spec}\:}
\newcommand{\Spf}{\mathrm{Spf}\:}
\newcommand{\Proj}{\mathrm{Proj}\:}
\newcommand{\Pic}{\mathrm{Pic }}
\newcommand{\Picloc}{\mathrm{Picloc }}
\newcommand{\Picloclocloc}{\Picloc^{\loc,\loc}}
\newcommand{\Br}{\mathrm{Br}}
\newcommand{\NS}{\mathrm{NS}}
\newcommand{\id}{\mathrm{id}}
\newcommand{\Sym}{{\mathfrak S}}
\newcommand{\Aut}{\mathrm{Aut}}
\newcommand{\Autp}{\mathrm{Aut}^p}
\newcommand{\End}{\mathrm{End}}
\newcommand{\Hom}{\mathrm{Hom}}
\newcommand{\Ext}{\mathrm{Ext}}
\newcommand{\ord}{\mathrm{ord}}
\newcommand{\coker}{\mathrm{coker}\,}
\newcommand{\divisor}{\mathrm{div}}
\newcommand{\Def}{\mathrm{Def}}
\newcommand{\et}{\mathrm{\acute{e}t}}
\newcommand{\loc}{\mathrm{loc}}
\newcommand{\ab}{\mathrm{ab}}
\newcommand{\pitop}{{\pi_1^{\mathrm{top}}}}
\newcommand{\pitoploc}{{\pi_{\mathrm{loc}}^{\mathrm{top}}}}
\newcommand{\piet}{{\pi_1^{\mathrm{\acute{e}t}}}}
\newcommand{\pietloc}{{\pi_{\mathrm{loc}}^{\mathrm{\acute{e}t}}}}
\newcommand{\pietlocab}{{\pi_{\mathrm{loc}}^{\mathrm{\acute{e}t},{\rm ab}}}}
\newcommand{\piN}{{\pi^{\mathrm{N}}_1}}
\newcommand{\piNloc}{{\pi_{\mathrm{loc}}^{\mathrm{N}}}}
\newcommand{\piNlocab}{{\pi_{\mathrm{loc}}^{{\rm N}, {\rm ab}}}}
\newcommand{\Het}[1]{{H_{\mathrm{\acute{e}t}}^{{#1}}}}
\newcommand{\Hfl}[1]{{H_{\mathrm{fl}}^{{#1}}}}
\newcommand{\Hcris}[1]{{H_{\mathrm{cris}}^{{#1}}}}
\newcommand{\HdR}[1]{{H_{\mathrm{dR}}^{{#1}}}}
\newcommand{\hdR}[1]{{h_{\mathrm{dR}}^{{#1}}}}
\newcommand{\Torsloc}{\mathrm{Tors}_{\mathrm{loc}}}
\newcommand{\prim}{\mathrm{prim}}
\newcommand{\reduced}{\mathrm{red}}
\newcommand{\fppf}{\mathrm{fppf}}
\newcommand{\defin}[1]{{\bf #1}}
\newcommand{\oX}{\mathcal{X}}
\newcommand{\oA}{\mathcal{A}}
\newcommand{\oY}{\mathcal{Y}}
\newcommand{\calC}{{\mathcal{C}}}
\newcommand{\calL}{{\mathcal{L}}}
\newcommand{\bmu}{\boldsymbol{\mu}}
\newcommand{\balpha}{\boldsymbol{\alpha}}
\newcommand{\bCW}{{\mathbf{CW}}}
\newcommand{\bG}{{\mathbf{G}}}
\newcommand{\bL}{{\mathbf{L}}}
\newcommand{\bM}{{\mathbf{M}}}
\newcommand{\bW}{{\mathbf{W}}}
\newcommand{\bD}{{\mathbf{D}}}
\newcommand{\BD}{{\mathbf{BD}}}
\newcommand{\BT}{{\mathbf{BT}}}
\newcommand{\BI}{{\mathbf{BI}}}
\newcommand{\BO}{{\mathbf{BO}}}
\newcommand{\C}{{\mathbf{C}}}
\newcommand{\Dic}{{\mathbf{Dic}}}
\newcommand{\SL}{{\mathbf{SL}}}
\newcommand{\MC}{{\mathbf{MC}}}
\newcommand{\GL}{{\mathbf{GL}}}
\newcommand{\Tors}{{\mathbf{Tors}}}
\newcommand{\Dieu}{{\mathds{D}}}
\newcommand{\Dieulocloc}{\Dieu_{\loc,\loc}}

\newcommand{\GM}[1]{\textcolor{blue}{(GM: #1)}}

\newcommand{\CL}[1]{\textcolor{red}{(CL: #1)}}

\newcommand{\YM}[1]{\textcolor{green}{(YM: #1)}}

\makeatletter
\@namedef{subjclassname@2020}{\textup{2020} Mathematics Subject Classification}
\makeatother

\title[Torsors over Rational Double Points]{Torsors over the Rational Double Points in Characteristic $\mathbf{p}$}

\author{Christian Liedtke}
\address{TU M\"unchen, Zentrum Mathematik - M11, Boltzmannstr. 3, 85748 Garching bei M\"unchen, Germany}
\email{liedtke@ma.tum.de}

\author{Gebhard Martin}
\address{Mathematisches Institut der Universit\"at Bonn, Endenicher Allee 60, 53115 Bonn, Germany}
\curraddr{}
\email{gmartin@math.uni-bonn.de}

\author{Yuya Matsumoto}
\address{Department of Mathematics, Faculty of Science and Technology, Tokyo University of Science, 2641 Yamazaki, Noda, Chiba, 278-8510, Japan}
\email{\url{matsumoto_yuya@rs.tus.ac.jp}}
\email{matsumoto.yuya.m@gmail.com}


\begin{abstract}
 We study torsors under finite group schemes over the punctured spectrum of 
 a singularity $x\in X$ in positive characteristic.
 
 We show that the Dieudonn\'e module of the (loc,loc)-part $\Picloclocloc_{X/k}$ of 
 the local Picard sheaf
 can be described in terms of local Witt vector 
 cohomology, making $\Picloclocloc_{X/k}$ computable. 
 Together with the class group and the abelianised local \'etale fundamental group, 
 $\Picloclocloc_{X/k}$ completely describes the finite abelian torsors over 
 $X\setminus\{x\}$.

 We compute $\Picloclocloc_{X/k}$ for every rational double point singularity,
 which complements results of Artin \cite{ArtinRDP} and Lipman \cite{Lipman},
 who determined $\pietloc(X)$ and ${\rm Cl}(X)$. 
 All three objects turn out to be finite.

 We extend the Flenner--Mumford criterion for smoothness of a normal surface germ $x \in X$
 to perfect fields of positive characteristic, generalising work of Esnault and Viehweg:
 If $k$ is algebraically closed, then $X$ is smooth if and only if 
 $\Picloclocloc_{X/k}$, $\pietloc(X)$, and ${\rm Cl}(X)$ are trivial.
 
 Finally, we study the question whether rational double point singularities
 are quotient singularities by group schemes and if so, whether the group scheme is uniquely determined by the singularity.
 We give complete answers to both questions, 
 except for some $D_n^r$-singularities in characteristic $2$. 
 In particular, we will give examples of (F-injective) 
 rational double points that are not quotient singularities.
\end{abstract}
\setcounter{tocdepth}{1}
\maketitle
\tableofcontents
\section{Introduction}
\subsection{Rational double points}
In 1884, Felix Klein \cite{Klein} classified quotients of $\Spec\CC[x,y]$ by finite subgroups of
$\SL_2(\CC)$ and obtained two series of singularities ($A_n$ and $D_n$), as well as three exceptional
singularities ($E_6$, $E_7$, $E_8$), the famous \emph{rational double point singularities} (RDPs for short).
In 1934, the same list of singularities showed up in Du~Val's \cite{DuVal} classification of normal surface
singularities ``that do not affect the condition of adjunction,'' which can be rephrased
as saying that rational double points are exactly the \emph{canonical surface singularities} in modern terminology.
Of course, much more can be said about these singularities, see for example \cite{Durfee} or \cite{Seminaire}.

\subsection{The linearly reductive case}
How much of the above holds over an algebraically closed field $k$ of positive 
characteristic $p>0$?  
When considering \emph{finite and linearly reductive} subgroup schemes 
(rather than finite subgroups) of $\SL_{2,k}$, a similar picture emerges:
Hashimoto \cite{Hashimoto} classified these subgroup schemes and by results of 
Hashimoto and Liedtke--Satriano \cite{Hashimoto, LiedtkeSatriano} it turns out
that canonical surface singularities in characteristic $p\geq7$ are the same
as quotient singularities by finite and linearly reductive subgroup schemes
of $\SL_{2,k}$. 
As over the complex numbers, we obtain two series and three exceptional cases.

Moreover, \emph{linearly reductive quotient singularities} (lrq singularities for short)
have many
beautiful properties as shown by the authors in \cite{LRQ}:
the quotient presentation is unique, their class groups, F-signatures, 
and local fundamental groups can be easily determined, etc. 
Since lrq singularities are \emph{F-regular}, they belong to a particularly
nice class of singularities.

In fact, for a normal surface singularity, being
F-regular (resp. F-regular and Gorenstein) is equivalent to being 
an lrq singularity with respect to a linearly reductive subgroup scheme of
$\GL_{2,k}$ (resp. $\SL_{2,k})$, see \cite[Theorem 11.2]{LRQ}.

\subsection{Beyond F-regularity}
The minimal model programme (MMP) leads us to study canonical singularities.
In dimension 2, these are precisely the RDP singularities and in 
characteristic $p>0$, they have been classified
by Artin \cite{ArtinRDP}.
We can deal with the F-regular ones by the results just mentioned. 
However, in characteristic $p\leq5$, not all of them are F-regular and in particular,
they cannot all be lrq singularities.
But, for the MMP in positive characteristic, we need to understand \emph{all} 
singularities that can occur there and not merely the ``nice'' ones, 
like the F-regular ones.
For example, in dimension 2, we can ask the following.
(We will only work with complete $k$-algebras and consider
formal isomorphisms all the time, as indicated by the hat on top of 
$\widehat{\Aff}^2_k$ in the sequel.)

\begin{Questions}
 \label{questions}
Let $x\in X$ be an RDP over an algebraically closed
field $k$ of positive characteristic.
\begin{enumerate}
    \item Is $x\in X$ quotient by a finite group scheme $G$, that is, 
$X\cong\widehat{\Aff}_k^2/G$?
\item If yes, are the group scheme $G$ and its action on $\widehat{\Aff}_k^2$ 
unique up to isomorphism?
\end{enumerate}
\end{Questions}

The goal of this article is three-fold:
\begin{enumerate}
\item We extend Klein's work \cite{Klein}
to characteristic $p$ and generalise
Artin's computations \cite{ArtinRDP} to the context of arbitrary
\emph{finite group schemes}, rather than merely groups.
\item We develop techniques and define invariants
to handle singularities that are not F-regular, in fact, not even
F-injective.
The main notion will be that of a \emph{local torsor} over a singularity.
These techniques and invariants will prove useful when studying 
singularities of the MMP in positive characteristic.
To illustrate this, we will compute these invariants for the
RDPs.
\item As an application, we establish characteristic-$p$
analogs of Mumford's theorem that characterises smoothness
in dimension 2 in terms of the local fundamental group.
\end{enumerate}

\subsection{Local torsors over singularities}
\label{subsec: intro local torsors}
A major theme of this article is the detection of (local) torsors
over singularities of dimension $d\geq2$, which leads to an 
\emph{inverse problem in invariant theory}:

Given a finite quotient singularity 
$x\in X=\widehat{\Aff}^d/G$ over the complex numbers, 
where $G$ acts freely outside the closed point $0$,
one can easily recover $G$ and 
its action on $\widehat{\Aff}^d$ as follows:
$q:(\widehat{\Aff}^d-\{0\})\to(\widehat{\Aff}^d/G-\{\overline{0}\})$
is the universal topological cover and thus,
we have $G\cong \pitoploc(X):=\pitop(X\setminus\{x\})$ 
and $G$ is the group of deck transformation of $q$.
In particular, this shows uniqueness of $G$ and uniqueness of 
the quotient presentation.

Next, let $k$ be an algebraically closed field of positive characteristic $p$.
Let $G$ be a finite group scheme over $k$, assume that $G$
acts on $\widehat{\Aff}^d$, such that the action is free 
outside the closed point and such that the closed point is fixed by $G$. 
Then, $x\in X:=\widehat{\Aff}^d/G$ is called the \emph{quotient singularity} 
associated to this datum, see also Definition \ref{def: quotient singularity}.
We are now interested in reversing this process, that is, given a singularity, 
we would like to know whether it is a quotient singularity as just defined. 
More precisely, we ask the following.

\begin{Questions}\label{question intro 2}
Let $x\in X$ be a singularity.
\begin{enumerate}
 \item How can we decide whether it is a quotient singularity?
 \item If $x\in X$ is a quotient singularity, does it
    determine the \emph{quotient presentation}, 
    that is, $G$ and its action on $\widehat{\Aff}^d$?
\end{enumerate}
\end{Questions}

Let $x\in X=\widehat{\Aff}^d/G$ be a quotient singularity (in the sense of 
Definition \ref{def: quotient singularity}). 
Then, the $G$-action on $\widehat{\Aff}^d$ makes the quotient
map $\pi:\widehat{\Aff}^d-\{0\}\to X-\{x\}$ a \emph{$G$-torsor}
that does not extend to a $G$-torsor $\widehat{\Aff}^d\to X$.
This leads to the notion of a \emph{local $G$-torsor (of primitive class)}, 
see Definition \ref{def: local torsor} and Definition \ref{def: primitive}.

Before proceeding, we recall the \emph{connected-\'etale sequence}, which is a 
canonical short exact and split sequence of group schemes
$$
 1\,\to\,G^\circ\,\to\,G\,\to\,G^\et\,\to\,1,
$$
where $G^\circ$ is infinitesimal and $G^\et$ is \'etale.

\subsubsection{\'Etale group schemes}
For a quotient singularity  $x\in X=\widehat{\Aff}^d/G$, we
have a factorisation of the quotient map
$$
   \widehat{\Aff}^d \,\to\, \widehat{\Aff}^d/G^\circ \,\to\, X.
$$
The \emph{local \'etale fundamental group}
$\pietloc(X) :=\piet(X-\{x\})$ is isomorphic to $G^\et$ and thus,
recovers $G^\et$ and shows uniqueness and canonicity of
$\widehat{\Aff}^d/G^\circ \to X$.
In particular, the local \'etale fundamental group describes
local torsors under \'etale group schemes.

\begin{RDP} 
For RDP singularities, the groups $\pietloc(X)$ 
have been computed by Artin \cite{ArtinRDP}.
\end{RDP}

\subsubsection{Diagonalisable group schemes}
If $x\in X=\widehat{\Aff}^d/G$ is a quotient singularity and 
$G$ is \emph{diagonalisable}, that is, a subgroup
scheme of $\GG_m^N$ for some $N$, then
the \emph{class group} ${\rm Cl}(X)$ is isomorphic
to the Cartier dual $G^D$, see \cite[Proposition 7.1]{LRQ}.
More generally, the class group describes local torsors
under diagonalisable group schemes, see Proposition \ref{prop: torsfinabcl}.

\begin{RDP}
For RDP singularities, the groups ${\rm Cl}(X)$ 
have been computed by Lipman \cite{Lipman}.
\end{RDP}

\subsubsection{Abelian group schemes}
If $x\in X=\widehat{\Aff}^d/G$ is a quotient singularity and 
$G$ is abelian, then there is a canonical embedding
$G^D\to\Picloc_{X/k}$, where $G^D$ is the Cartier dual of $G$ and 
where $\Picloc_{X/k}$ is the \emph{local Picard functor}.
This functor was extensively studied by Boutot in \cite{Boutot}.
However, it is \emph{not} true in general that $\Picloc_{X/k}$
is isomorphic to $G^D$.
In fact, this is one of the reasons why 
Questions \ref{question intro 2} are hard to answer.

\begin{Example}
 The RDP of type $E_6^0$ in characteristic $p=3$ is a
 quotient singularity by $\balpha_3$.
 Here, $\Picloc_{X/k}$ contains 
 $\bmu_3\cong(\C_3)^D\cong{\rm Cl}(X)^D$ and is
 thus strictly larger than $\balpha_3^D\cong\balpha_3$.
 Recall that $\balpha_3$ and $\bmu_3$ are the Frobenius kernels on
 $\mathbb{G}_a$ and $\mathbb{G}_m$, respectively, and that
 $\C_3$ is the cyclic group of order $3$.
\end{Example}

\subsubsection{Group schemes of (loc,loc)-type}
Given a singularity $x\in X$, we have the associated local
Witt vector cohomology groups $H^i_\idealm(X,W_n \OO_X)$.
For a group scheme $G$ over $k$ (or a $W(k)$-module $M$) with an
endomorphism $F$, we denote the kernel of $F$ by $G[F]$ 
(resp. $M[F]$).
The colimit
$$
 \varinjlim_{m,n}\, H^i_{\idealm}(X,W_n\OO_X)[F^m]
$$
is a $W(k)$-module with actions of Frobenius $F$
and Verschiebung $V$ turning it into a left module under the
\emph{Dieudonn\'e ring} $\Dieu:=W(k)\{F,V\}/(FV-p)$, a non-commutative
polynomial ring over $W(k)$.

Let us recall the \emph{(covariant) (loc,loc)-Dieudonn\'e module} 
associated to a finite and commutative group scheme $G$
$$
   \Dieulocloc(G) \,:=\, \varinjlim_{n,m}\, \Hom(\bL_{n,m},G),
$$
which is a $\Dieu$-module, and
where the $\bL_{n,m}$ are certain finite, commutative, and infinitesimal
group schemes, see, for example, \cite{Oort}.
If $G$ is of \emph{(loc,loc)-type}, that is, both $G$ and $G^D$ are
infinitesimal, then $\Dieulocloc$ gives rise to an equivalence of
categories with a category of certain $\Dieu$-modules, see Section \ref{subsec: loclocCartierDieudonne} for more details.

In Section \ref{subsec: Cartier--Dieudonne theory}, we will consider, more generally, (loc,loc)-Dieudonn\'e modules of fppf sheaves of abelian groups.
In Section \ref{subsec: abelian loc loc}, we define the subfunctor
$$
\Picloclocloc_{X/k} \,:=\, \varinjlim_{m,n} \,\Picloc_{X/k}[F^n,V^m].
$$
Even if the functor $\Picloc_{X/k}$ is not representable by a scheme,
this colimit will be an ind-group scheme.

\begin{Theorem}[Corollary \ref{cor: Dieudonneofpicloc}]
 For a singularity $x\in X$ of dimension $d \geq 2$ in characteristic $p>0$, $\Picloclocloc_{X/k}$ is an ind-(loc,loc) group scheme and there exist isomorphisms of Dieudonn\'e modules
 $$
  \Dieulocloc(\Picloclocloc_{X/k}) \,\cong\, 
  \Dieulocloc(\Picloc_{X/k}) \,\cong\, 
  \varinjlim_{m,n}\, H^2_{\idealm}(X,W_n\OO_X)[F^m].
 $$
\end{Theorem}

Thus, the $\Dieu$-module structure on the right hand side
describes local torsors over $x\in X$ 
under abelian group schemes of (loc,loc)-type.

\begin{Remark}
If $x\in X$ is a normal surface singularity that is 
\emph{not} F-injective, then the Frobenius action on $H^2_\idealm(X,\OO_X)$
has a non-trivial kernel.
Then, $\varinjlim_{m,n}\, H^2_{\idealm}(X,W_n\OO_X)[F^m]$ is
non-zero, thus, $\Picloclocloc_{X/k}$ is non-trivial,
and thus, there exist non-trivial torsors under (loc,loc) group
schemes over $x\in X$.
Thus, failure of F-injectivity gives rise to 
geometric objects.
\end{Remark}

\begin{RDP}
For RDPs, we compute these $\Dieu$-modules and the associated group schemes
of (loc,loc)-type in Table \ref{table: torsors with Picloclocloc} and 
Table \ref{table: wittvectorcohomology}.
\end{RDP}

An interesting corollary of our computations is the following
finiteness result, which should be compared with Artin's finiteness
result for $\pietloc(X)$ and Lipman's finiteness result for ${\rm Cl}(X)$
if $x\in X$ is an RDP.

\begin{Theorem}[Theorem \ref{thm: picloclocloc finite}]
 If $x\in X$ is an RDP in characteristic $p>0$, then
 the length of $\varinjlim_{m,n} H^2_{\idealm}(X,W_n\OO_X)[F^m]$
 as a $\Dieu$-module is finite.
 In particular, $\Picloc_{X/k}^{\loc,\loc}$ 
 is a finite group scheme.
\end{Theorem}

We note that there is also a curious coincidence of numbers between 
this finite length, the Artin co-index of an RDP, and
certain local cohomology groups of the minimal resolution
of singularities that are related to equisingular deformations, see
Section \ref{subsec: curious}.

\subsubsection{Torsors and the local fundamental group scheme}
To describe all torsors over a (suitable) scheme, 
Nori \cite{Nori} constructed the fundamental group scheme.
Using this, Esnault and Viehweg \cite{EsnaultViehweg} constructed
a \emph{local fundamental group scheme} $\piNloc(U,X,x)$ 
for a singularity $x\in X$, where we set $U:=X-\{x\}$.
As shown in loc.cit., the maximal \'etale quotient of $\piNloc(U,X,x)$
is isomorphic to the local fundamental group $\pietloc(X)$.

Now, let $x\in X=\widehat{\Aff}^d/G$ be a quotient singularity.
Then, we have $G^\et\cong\pietloc(X)$, and one might be tempted to think
that $\piNloc(U,X,x)$ is isomorphic to $G$.
One can prove (Lemma \ref{lem: quotientsofpinloc})
that $G$ is always a quotient of $\piNloc(U,X,x)$ that
induces an isomorphism on maximal \'etale quotients.

\begin{Examples}
There exist examples of the following
\begin{enumerate}
  \item There exists a quotient singularity 
   $x\in X=\widehat{\Aff}^2/G$ that is an RDP,
   such that the natural surjection $\piNloc(U,X,x)\to G$ 
   is \emph{not} an isomorphism, see
   Example \ref{ex: weirdpinlocs}.
   We refer to Remark \ref{rem: cokerneltoolarge} for some implications of this phenomenon.
  \item There exists an RDP $x\in X$ and a tower of two local torsors 
   that is \emph{not} dominated by a local torsor over $x\in X$,
   see Lemma \ref{lemma example 2}.
   Thus, Galois closures of towers of torsors may fail to exist.
   Of course, non-existence of Galois closure of torsors was already known since
   \cite{ABETZ}, but this shows that this phenomenon has geometric significance in the study 
   of quotient singularities.
  \item There exists an RDP $x\in X$ and a finite and infinitesimal group scheme $G$
  that acts freely on $X\setminus\{x\}$ and fixes $\{x\}$, such that
  $X/G\cong X$.
  Thus, we obtain a local $G$-torsor $X\to X$ and thus, even an infinite tower,
  see Proposition \ref{prop: infinitetowers}.
\end{enumerate}
We note that the group schemes used in these examples are finite and commutative, 
so these pathological phenomena do not stem from the complexity of finite and non-commutative 
group schemes.
\end{Examples}

Thus, the local fundamental group scheme
$\piNloc(U,X,x)$ displays quite strange properties already for RDPs, 
making it very different from classical local fundamental groups.

\subsection{Detecting smoothness via local torsors}
A famous theorem of Mumford \cite{Mumford} states that a normal
two-dimensional singularity
$x\in X$ is smooth if and only if its local fundamental group
is trivial.
Moreover, Flenner \cite{Flenner} showed that already triviality of the local \'etale
fundamental group $\pietloc(X)$ is sufficient for smoothness.

This result was generalised to characteristic $p>0$ by Esnault and Viehweg \cite{EsnaultViehweg}
by showing that a normal two-dimensional singularity $x\in X$ is smooth if and only if 
the local Nori 
fundamental group scheme $\piNloc(U,X,x)$ is trivial 
(here, we set $U:=X\backslash\{x\}$) with the possible exception of three 
RDPs in characteristic $p=2,3$.
Using the results and techniques developed in this article, 
we can remove these exceptions:

\begin{Theorem}(Theorem \ref{thm: Mumfordstheorem})
 Let $x\in X$ be a normal two-dimensional singularity over
 an algebraically closed field of characteristic $p>0$.
 Then, the following are equivalent:
 \begin{enumerate}
     \item $x\in X$ is a smooth point.
     \item $\piNloc(U,X,x)$ is trivial.
     \item ${\rm Cl}(X)$ is trivial, $\pietloc(X)$ is trivial, and $X$ is F-injective.
 \end{enumerate}
\end{Theorem}

If  $p\geq7$, then the condition on F-injectivity in (3) can
also be dropped, see Corollary \ref{cor: mumford}.
We also establish a version of this theorem over perfect fields.
We refer to Section \ref{sec: Mumford} for details, discussion,
and examples that illustrate that this result is rather sharp.

\subsection{Quotient and non-quotient RDPs}
We now address Questions \ref{questions}: let $x\in X$ be an RDP.
If it is a quotient singularity, say
$X\cong\widehat{\Aff}^2/G$, then there exists a local
$G$-torsor over $x\in X$.
The techniques and results described in 
Section \ref{subsec: intro local torsors} allow us to either
determine $G$ directly or to at least limit the possibilities.
In some cases, these limitations are the key to showing that
a given RDP \emph{cannot} be a quotient singularity at all.

As already mentioned, if an RDP is F-regular, then
it is a lrq singularity by \cite[Theorem 11.2]{LRQ}.
By \cite[Theorem 8.1]{LRQ}, we even have
uniqueness of the quotient presentation in this case.
Note that all RDPs in characteristic $p\geq7$ are F-regular.
It remains to treat the RDPs that are not F-regular:

\begin{Theorem} \label{thm: QuotientRDPs}
 (Theorem \ref{thm: quotient} and Theorem \ref{thm: not quotient})
Let $x\in X$ be an RDP in characteristic $p>0$
and assume that it is not of type $D_n^r$ with $ 2 < 4r < n$. 
Then, $x\in X$ is a quotient singularity if and only if one of the following hold:
\begin{enumerate}
\item $p \geq 5$.
\item $p = 3$ and $X$ is not of type $E_8^0$.
\item $p = 2$ and $X$ is of type $A_n$, $D_{4r}^r$, $D_{4n}^0$, $E_6^0$, $E_6^1$, $E_7^3$, $E_8^0$, $E_8^2$, or $E_8^4$.
\end{enumerate}
Moreover, if $x\in X$ is a quotient singularity, 
then there is a unique group scheme realizing it as a quotient singularity.
\end{Theorem}

We refer to Table \ref{table: Quotient- and non-quotient RDPs}, Theorem \ref{thm: quotient}, 
and Theorem \ref{thm: not quotient} for the group schemes that show up and more details. 
For a study of the RDPs of type $D_n^r$ with $2 < 4r < n$ in characteristic $p=2$, 
we refer the reader to Section \ref{subsec: dnr}.
To give an impression of our findings, we give an example. 

\begin{Example}
 Artin \cite{ArtinRDP} showed that there exist three analytic 
 isomorphism classes of RDPs of type $E_8$  in characteristic $p=3$.

\begin{enumerate}
 \item $E_8^2$ is a quotient singularity by the group ${\BT}_{24}$ of order $24$,
 \item $E_8^1$ is a quotient singularity by the infinitesimal group scheme $\bM_2$ of order $9$, and
 \item $E_8^0$ cannot be realised as a quotient singularity.
\end{enumerate}
In particular, this shows that Koll\'ar's conjectural characterisation 
of quotient singularities is false in positive characteristic,
even if one allows quotient singularities by finite group schemes
-  in fact, it already fails for RDPs in dimension two,
see Remark \ref{rem: KollarConjecture}.
\end{Example}

In case we have a quotient presentation $x\in X=\widehat{\Aff}^2/G$ where 
$G$ is not linearly reductive and $G$ is not \'etale, 
then we do not know whether the $G$-action on $\widehat{\Aff}^2$ is unique
up to conjugation.
To give an idea of why these questions are subtle,
we note that the action of a finite group scheme $G$ acting on 
$\widehat{\Aff}^d$, such that the action is free outside the closed point and such that the closed point is fixed,
is \emph{linearisable} if and only if $G$ is linearly reductive,
see \cite{LRQ}.
In particular, the two actions of (1) and (2) in the example
\emph{cannot} be linearised, which makes such actions hard to write down, 
let alone to classify.

\subsection*{Organisation}
This article is organised as follows:

In Section \ref{sec: recap}, we recall the notions of F-regularity
and F-injectivity, as well as the explicit equations of the RDP singularities.

In Section \ref{sec: group schemes}, we recall some results on finite group
schemes needed for this article, especially some, which are 
`well-known to the experts' and hard to find.

In Section \ref{sec: maximalsubsheaves}, we define the notion of \emph{maximal $\mathcal{C}$-subsheaf}, which we use to extract representable subfunctors of the local Picard functor.

In Section \ref{sec: local torsors}, we introduce the notion of 
a \emph{local torsor} over a singularity $x\in X$
and develop techniques to detect such torsors using $\pietloc(X)$, 
${\rm Cl}(X)$, $\varinjlim_{m,n}\, H^2_{\idealm}(X,W_n\OO_X)[F^m]$,
and $\Picloc_{X/k}$.

In Section \ref{sec: torsors over rdp}, we use these techniques
to establish Table \ref{table: torsors with Picloclocloc}, that is,
the structure of local torsors over the RDPs.

In Section \ref{sec: Mumford}, we extend theorems of Mumford
and Flenner \cite{Flenner, Mumford}
characterising smoothness in dimension 2 to positive 
characteristic and remove some assumptions from Esnault and Viehweg's
theorem \cite{EsnaultViehweg}.

In Section \ref{sec: Nori}, we recall the notion of a \emph{very small}
group scheme action from \cite{LRQ} and discuss
the results of Section \ref{sec: local torsors} from the
point of view of the (local) Nori fundamental group scheme.

In Section \ref{sec: rdp}, we address Question \ref{questions}, that is,
which RDPs are quotient singularities, which are not, and 
uniqueness of the group scheme.

In Section \ref{sec: pathologies}, we collect some results
concerning pathologies of local torsors over RDPs, which 
include (non-)uniqueness, (non-)normality, and (non-)rationality 
of such local torsors, as well as non-existence of Galois closures and 
some pathologies in the deformation theory of quotient RDPs.

\begin{VoidRoman}[Acknowledgements]
 We thank Fabio Bernasconi, Brian Conrad, H\'el\`ene Esnault, Leo Herr, 
 Hiroyuki Ito, Matt Satriano, Karl Schwede, Nick Shepherd-Barron, and 
 Claudia Stadlmayr for discussion.
  We also thank the referees for their many comments.
   Part of the research on this article was done whilst the first named author was on sabbatical at the University of Oxford and he thanks the Mathematical Institute for kind hospitality during his stay.
 The first named author was supported by the ERC Consolidator Grant 681838 (K3CRYSTAL).
 The second named author was supported by the DFG Research Grant MA 8510/1-1. 
 The third named author was supported by JSPS KAKENHI Grant Numbers 
 JP16K17560 and JP20K14296.
\end{VoidRoman}

\begin{VoidRoman}[Notations and Conventions]

Throughout this article, $k$ denotes a field of characteristic 
$p\geq0$ and it will be algebraically closed until otherwise stated.
By abuse of notation, we shall not distinguish between finite groups and
their associated constant group schemes over $k$.
\end{VoidRoman}

\section{F-regular and canonical singularities in dimension two}
\label{sec: recap}
First, recall that a local, complete, and Noetherian $k$-algebra $(R,\mathfrak{m})$ of dimension $d\geq2$ is called:

\begin{enumerate}
 \item \emph{F-regular} if all ideals of all localisations of 
 $R$ are tightly closed, and
 \item \emph{F-injective} if the map induced by Frobenius
on local cohomology $F:H^i_{\idealm}(R)\to H^i_{\idealm}(R)$ 
is injective for all $i=0,...,\dim(R)$.
\end{enumerate}
It is well-known that F-regular rings are F-injective.

\begin{Remark}
For the sake of completeness, we also recall that a 
Cohen--Macaulay local ring $(R,\mathfrak{m})$ is called 
\begin{enumerate}
    \item \emph{F-split} if the Frobenius map $F: R \to F_*R$ splits as a map of $R$-modules, and
    \item \emph{F-rational} if for every $c \in R$ that is not in a minimal prime of $R$, there exists $e > 0$ such that $c \cdot F^e: H^d_{\idealm}(R)\to H^d_{\idealm}(R)$ is injective.
\end{enumerate}
If $R$ is Gorenstein, then it is F-rational if and only if it is
F-regular \cite[Corollary 4.7]{HochsterHunekeTestelements}, and it is
F-split if and only if it is F-injective \cite[Lemma 3.3]{Fedder}. 
Since we will mainly work with rational double points and these are Gorenstein, 
we will not use the notions of F-splitting and F-rationality 
in this article.
\end{Remark}

The formal isomorphism classes of rational double points were calculated by Artin \cite{ArtinRDP}. 
In Table \ref{table: RDPEquations}, we recall the classification of rational double 
points and we also recall which of them are F-regular or F-injective 
(see \cite[Theorem 1.1. and Theorem 1.2]{Hara}).

\begin{Convention} \label{convention}
For the notation of rational double points, we follow \cite{ArtinRDP} with the modification 
made by the third named author in \cite[Convention 1.2]{Matsumotoinseparable}:
We call $z^2 + x^2 y + z y^m + z x y^{m-s} = 0$ ($m \geq 2$, $0 \leq s \leq m-1$) in characteristic $2$ 
to be of type $D_{2m+1}^{s + 1/2}$, instead of Artin's notation $D_{2m+1}^s$.
Consequently, the range of $r$ for $D_{2m+1}^r$ is $\{\frac{1}{2}, \frac{3}{2}, \dots, \frac{2m-1}{2}\}$.
\end{Convention}

\begin{table}
  $$\begin{array}{|c|c|c|c|l|} \hline
      \mbox{char.} & X    &  \mbox{F-regular} & \mbox{F-injective} & \mbox{equation} \\
       \hline \hline
   \mbox{all} &   A_n &  \checkmark  &  \checkmark & z^{n+1} + xy\\
  > 2 &    D_n &  \checkmark  &  \checkmark & z^2 + x^2y + y^{n-1} \\
  > 3 &   E_6 &  \checkmark  &  \checkmark & z^2 + x^3 + y^4 \\
   > 3 &   E_7 &  \checkmark  &  \checkmark & z^2 + x^3 + xy^3 \\
    > 5  & E_8 &  \checkmark  &  \checkmark & z^2 + x^3 + y^5 \\ \hline
    5 &  E_8^1 & \times & \checkmark & z^2 + x^3 + y^5 + xy^4 \\
    5 &   E_8^0 & \times & \times & z^2 + x^3 + y^5 \\
\hline
   3 &   E_6^1 & \times & \checkmark & z^2 + x^3 + y^4 + x^2y^2 \\
    3 &  E_6^0 & \times & \times & z^2 + x^3 + y^4  \\
    3 & E_7^1 & \times & \checkmark & z^2 + x^3 + xy^3 + x^2y^2 \\
    3 &   E_7^0 & \times & \times & z^2 + x^3 + xy^3 \\
    3 &  E_8^2 & \times & \checkmark & z^2 + x^3 + y^5 + x^2y^2 \\
   3 &    E_8^1 & \times & \times & z^2 + x^3 + y^5 +  x^2y^3 \\
   3 &    E_8^0 & \times & \times & z^2 + x^3 + y^5  \\
\hline
  2 &  D_{2m}^r & \times & \mbox{iff $r = m -1$} & 
    z^2 + x^2 y + x y^m + z x y^{m-r}  \\
  2 &  D_{2m+1}^r & \times & \mbox{iff $r = m -1/2$} & 
    z^2 + x^2 y + z y^m + z x y^{m-r+1/2}  \\
   2 &    E_6^1 & \times & \checkmark &  z^2 + x^3 + y^2z +xyz \\
    2 &  E_6^0 & \times &  \times & z^2 + x^3 + y^2z \\
    2 &  E_7^3 & \times & \checkmark& z^2 + x^3 + xy^3 +xyz \\
     2 & E_7^2 & \times & \times & z^2 + x^3 + xy^3 + y^3z\\
    2 &  E_7^1 & \times & \times & z^2 + x^3 + xy^3 + x^2yz \\
     2 & E_7^0 & \times & \times & z^2 + x^3 + xy^3  \\
     2 & E_8^4 & \times & \checkmark & z^2 + x^3 + y^5 + xyz \\
     2 & E_8^3 & \times & \times & z^2 + x^3 + y^5  + y^3z  \\
    2 &  E_8^2 & \times & \times & z^2 + x^3 + y^5 + xy^2z  \\
    2 &  E_8^1 & \times & \times & z^2 + x^3 + y^5 +xy^3z  \\
    2 &  E_8^0 & \times & \times & z^2 + x^3 + y^5  \\ \hline

  \end{array}$$
  \caption{Equations of rational double points}
    \label{table: RDPEquations}
\end{table}

\section{Finite group schemes}
\label{sec: group schemes}

In this section, we give a short review of some well-known results about 
finite flat group schemes.

Let $G$ be a finite group scheme over an algebraically closed 
field $k$ of characteristic $p\geq0$. 
Since $k$ is perfect, there exists a split short exact sequence 
of finite group schemes over $k$
$$
   1\,\to\,G^\circ\,\to\,G\,\to\, G^{\et} \,\to\,1,
$$
where $G^\circ$ is the connected component of the identity and where $G^{\et}$
is an \'etale group scheme over $k$.
Thus, we have a canonical isomorphism $G\cong G^\circ\rtimes G^\et$.
Moreover, $G^\circ$ is an infinitesimal group scheme of length equal to 
some power of $p$.
In particular, if $p=0$ or if the length of $G$ is prime to $p$, then $G^\circ$ is trivial and $G$ is \'etale. 

\subsection{Finite \'etale group schemes}
\label{sec: finet}
Since $k$ is algebraically closed, a finite and \'etale group scheme
over $k$ is the \emph{constant} group scheme associated to a finite group.
In particular, the classification of finite \'etale group schemes over $k$
is equivalent to that of finite groups.
We will use the notation $\C_n$ for the cyclic group of order $n$
and the associated constant group scheme.

\subsection{Diagonalisable group schemes}
\label{sec: findiag}
If $M$ is a finitely generated abelian group, then the group algebra
$k[M]$ carries a Hopf algebra structure, and the associated commutative
group scheme is denoted $D(M):=\Spec k[M]$.
By definition, such group schemes are called \emph{diagonalisable}.
For example, we have $D(\ZZ)\cong\GG_m$ and
$D(\C_n)\cong\bmu_n$.
Recall that every diagonalisable group scheme can be embedded
into $\GG_m^N$ for some $N\geq1$.
Moreover, $\bmu_n$ is \'etale over $k$ if and only if
$p\nmid n$.

\subsection{Linearly reductive group schemes}
A finite group scheme $G$ over $k$ is said to be \emph{linearly reductive} 
if every (finite-dimensional) representation of $G$ is semisimple.
If $p=0$, then all finite group schemes over $k$ are \'etale and
linearly reductive.
If $p>0$, then, by a theorem of  Nagata \cite[Theorem 2]{Nagata61}
(see also \cite[Proposition 2.10]{AOV} and \cite[Section 2]{Hashimoto}),
a finite group scheme over $k$ is
linearly reductive if and only if it is an extension of a finite and \'etale
group scheme, whose order is prime to $p$ 
by a diagonalisable group scheme.

\subsection{Finite abelian group schemes and Cartier duality}
\label{sec: finab}
If $G$ is a finite and \emph{abelian} group scheme, 
then $G^D:=\mathcal{H}\mathit{om}(G,\GG_m)$ is again a group scheme, 
called
the \emph{Cartier dual} group scheme. Recall that in this setting, the \emph{Verschiebung} homomorphism is the endomorphism of $G$ induced by Frobenius on its Cartier dual.
Since $k$ is algebraically closed, we have a canonical
decomposition
\begin{equation}
 \label{eq: canonical decomposition}
   G \,\cong\, G^{\loc,\loc}\times G^{\loc,\et}\times G^{\et,\loc}\times G^{\et,\et}
\end{equation}
where $\loc$ (resp. $\et$) in the first argument means that the group scheme in 
question is infinitesimal (resp. \'etale), and the second argument refers to
the corresponding property of the Cartier dual group scheme.

We have Frobenius $F \colon G\to G^{(p)}$ and Verschiebung $V \colon G^{(p)}\to G$ that
factor multiplication by $p$, that is, $p=FV=VF$.
Moreover, with respect to the canonical decomposition 
\eqref{eq: canonical decomposition} of $G$, they
behave as follows:
$$
\begin{array}{lll}
 G^{\loc,\loc} & \mbox{$F$ is nilpotent} & \mbox{$V$ is nilpotent} \\
 G^{\loc,\et} & \mbox{$F$ is nilpotent} & \mbox{$V$ is an isomorphism} \\
 G^{\et,\loc} & \mbox{$F$ is an isomorphism} & \mbox{$V$ is nilpotent} \\
 G^{\et,\et} & \mbox{$F$ is an isomorphism} & \mbox{$V$ is an isomorphism} 
\end{array}
$$

\subsection{Local-local group schemes and Cartier--Dieudonn\'e theory}
\label{subsec: loclocCartierDieudonne}

By Section \ref{sec: finet}, we have a good understanding of the \'etale factors 
in the canonical decomposition \eqref{eq: canonical decomposition}
and, via Cartier duality, also of the $G^{\loc,\et}$-part of $G$. 
It thus remains to understand $G^{\loc,\loc}$.
Finite and abelian group schemes $G$ with
$G=G^{\loc,\loc}$ will be called \emph{finite group schemes of (loc,loc)-type}.
To describe them, we follow \cite{Oort} and 
define
$$
\bL_{n,m} \,:=\, \bW_n[F^m],
$$
where $m,n \in \NN$ are natural numbers, $\bW_n$ is the group scheme of Witt vectors of length $n$, 
$F^m$ is the $m$-fold Frobenius on $\bW_n$, and $\bW_n[F^m]$ denotes the kernel of $F^m$. 
We recall that $\bW_1 = \GG_a$ and that $\balpha_p := \bL_{1,1} = \GG_a[F]$ is the 
simplest example of a finite group scheme of (loc,loc)-type. 
In general, $\bL_{n,m}$ is a finite and abelian group scheme of (loc,loc)-type and 
length $p^{nm}$ over $k$.

For every $m,n \in \NN$, there are four homomorphisms
\begin{eqnarray*}
F: &\bL_{n,m} \twoheadrightarrow \bL_{n,m-1} &\text{ induced by Frobenius } F: W_n \to W_n \\
R: &\bL_{n,m} \twoheadrightarrow \bL_{n-1,m} &\text{ induced by the restriction } R: W_n \to W_{n-1} \\
I: &\bL_{n,m} \hookrightarrow \bL_{n,m+1} &\text{ the natural inclusion }  \\
V: &\bL_{n,m} \hookrightarrow \bL_{n+1,m} &\text{ induced by Verschiebung } V: W_n \to W_{n+1}.
\end{eqnarray*}
We note that $\bL_{n,m}^D = \bL_{m,n}$ and that Cartier duality interchanges $F$ with $V$ and $R$ with $I$.
Then, the $\{\bL_{n,m}\}_{m,n\in \mathbb{N}}$ form an inverse system $(\{\bL_{n,m}\}, R,F)$ 
with respect to $F$ and $R$ and they form a 
directed system $(\{\bL_{n,m}\}, I,V)$ with respect to $I$ and $V$.

The importance of the group schemes $\bL_{m,n}$ comes from the fact that they 
can be used to classify \emph{all} finite group schemes of (loc,loc)-type.
Let $W(k)$ be the ring of Witt vectors over $k$, and let
$$
  \Dieu \,:=\, W(k)\{F,V\}/(FV-p)
$$
be the \emph{Dieudonn\'e ring}, which is by definition the
non-commutative polynomial ring in $F,V$ over $W(k)$ subject to the
relations
\begin{equation}
    \label{dieudonne}
    FV=VF=p,\quad Fw=\sigma(w)F,\quad wV=V\sigma(w),\quad \forall w\in W(k),
\end{equation}
where $\sigma$ denotes the Frobenius on $W(k)$.

The ring $\Dieu$ acts from the left on the directed system $(\{\bL_{n,m}\},I,V)$ by letting $F$ act as $F$, $V$ as $V$, and a scalar $a$ as $\sigma^{-n}(a)$.
Dualising everything and using the identification $\bL_{n,m}^D \cong \bL_{m,n}$, $\Dieu$ acts from the right on the inverse system $(\{\bL_{n,m}\}, R,F)$ by letting $F$ act as $V$, $V$ as $F$, and a scalar $a$ as $\sigma^{-m}(a)$.  
 
Now, if $G$ is a finite and commutative group scheme over $k$, then 
there is an associated left $\Dieu$-module 
$$
\Dieulocloc(G) \,:=\, \varinjlim_{n,m}\, \Hom\left(\bL_{n,m},G\right),
$$
where $\{\bL_{n,m}\}_{m,n\in \mathbb{N}}$ is viewed as an inverse system 
with transition maps $R$ and $F$ and the left $\Dieu$-module structure comes from the 
right-action of $\Dieu$ on $\bL_{n,m}$, that is,
$a \in W(k)$ acts by $\sigma^{-m}(a)$, and $F$ (resp.\ $V$) by the pre-composition 
$- \circ V$ (resp.\ $- \circ F$).
The module $\Dieulocloc(G)$ is called the
\emph{(covariant) Dieudonn\'e module} of 
the $(\loc,\loc)$-factor of $G$. Recall the main theorem of (covariant $(\loc,\loc)$-)Dieudonn\'e theory:

\begin{Theorem} \label{thm: Dieudonnemain}
Let $k$ be a perfect field of characteristic $p>0$.
Then, the functor $\Dieulocloc$ induces an equivalence of categories between 
finite (loc,loc) group schemes over $k$ and left $\Dieu$-modules that
are finite as $W(k)$-modules and on which $F$ and $V$ are nilpotent.
\end{Theorem}

We refer to \cite{Oort, Oort2} for background, proofs, and further references. 
In particular, note that we are only using the $(\loc,\loc)$-part of Dieudonn\'e theory.

\begin{Remark} \label{rem: choice of dieustructure}
We remark that there are at least three other (more or less common) 
variants of Dieudonn\'e theory that we could have used. 
These different variants correspond to different choices of the $\Dieu$-action 
on $\bL_{n,m}$. 
More precisely, one has to choose whether $\Dieu$ acts from the left or from the right 
and whether one uses $\varinjlim_{n,m}\, \Hom(\bL_{n,m},G)$, in which case the
$\Dieu$-action on $\bL_{n,m}$ must be compatible with $R$ and $F$, or else
$\varinjlim_{n,m}\, \Hom(G,\bL_{n,m})$, in which case the $\Dieu$-action on 
$\bL_{n,m}$ must be compatible with $I$ and $V$.
\end{Remark}

Let us also recall a couple of facts that we will need in the sequel:
if $G$ is of length $p^n$, then $\Dieulocloc(G)$ is of length $n$ as a 
$W(k)$-module.
Moreover, $\Dieulocloc(F)$ is left-multiplication by $V$ and thus,
$$
   \Dieulocloc( \bL_{n,m} ) \,\cong\, \Dieu / (F^n,V^m).
$$

\begin{Example} We record the following important special cases:
\begin{enumerate}
    \item $\balpha_{p^m} = \bL_{1,m}$, hence $\Dieulocloc(\balpha_{p^m}) \cong \Dieu/(F,V^m)$.
    \item $\balpha_{p^n}^D = \bL_{n,1}$, hence $\Dieulocloc(\balpha_{p^n}^D) \cong \Dieu/(F^n,V)$.
    \item The kernel $\bM_n$ of $F^n$ on a supersingular elliptic curve is $\bM_n = \bL_{n,n}[F-V] \cong \bL_{n,n}/(V-F)$. Hence,
    $$
    \hspace{10mm}
    \begin{array}{lclcl}
 \Dieulocloc( \bM_n) & = & (\Dieu/(F^n,V^n))[F-V]  &\cong&  \Dieu/\Dieu(F^n,V^n, F-V) \\
 &\cong &  \Dieu/\Dieu(F^n, F-V) &\cong&  \Dieu/\Dieu(V^n, F-V),
 \end{array}
$$
where $\Dieu(F^n,V^n,F-V)$ is the left ideal generated by the indicated elements.
\end{enumerate}
\end{Example}

More generally, we have
$$
\begin{array}{lcl}
\Dieulocloc ( \bL_{n,m}[F^a-V^b] ) &\cong& (\Dieu / (F^n,V^m))[F^b-V^a], \\
\end{array}
$$
where $[F^b-V^a]$ on the right-hand side denotes the kernel of the right multiplication by $F^b-V^a$.

\subsection{Derivations, p-Lie algebras, and group schemes of height $1$} \label{sec: height1}
Let $G$ be a finite and infinitesimal group scheme over $k$.
We say that $G$ has \emph{height one} if $G = G[F]$.
Let $\idealg$ be its Lie algebra (see \cite[Exp. VIIA, Section 6]{SGA31}), which is in fact a $p$-Lie algebra, that is, a finite 
dimensional $k$-vector space with a Lie bracket and a $p$-operation $x\mapsto x^{[p]}$ satisfying certain axioms and compatibilities \cite[Exp. VIIA, Section 5]{SGA31}.

Conversely, given a $p$-Lie algebra $\idealg$, the dual of the universal envelopping $p$-algebra $\underline{U}_p(\idealg)$ is a commutative Hopf algebra whose spectrum $\mathcal{G}_p(\idealg)$ is a finite and infinitesimal group scheme of height one.
The functors $G \mapsto \idealg$ and $\idealg \mapsto \mathcal{G}_p(\idealg)$ are essential inverses and they establish an equivalence of categories between
finite and infinitesimal group schemes of height one over $k$
and finite dimensional $p$-Lie algebras over $k$ (see \cite[Exp. VIIA, Section 8]{SGA31}).

\begin{Examples}
The group scheme $G=\bmu_{p^n}$ (resp. $G=\balpha_{p^n}$) is of height one if and only if $n=1$. In this case, $\idealg$ is a one-dimensional $k$-vector space with trivial Lie bracket and there is a basis $e\in\idealg$ such that $e^{[p]}=e$ (resp. $e^{[p]}=0$).

Conversely, if $\idealg$ is a $1$-dimensional $p$-Lie algebra, then its Lie bracket must be trivial. If $k$ is algebraically closed, we can find a basis $e\in\idealg$ such that $e^{[p]}=e$ (resp. $e^{[p]}=0$), so $\bmu_p$ and $\balpha_p$ are the only two finite and infinitesimal group schemes of height $1$ with $1$-dimensional Lie algebra.
\end{Examples}

Let $X$ be any $k$-scheme and let $H^0(X,\Theta_X)$ be the space of global vector fields on $X$, which is a (not necessarily finite-dimensional) $p$-Lie algebra over $k$. 
Let $\idealg$ be a $p$-Lie algebra and let ${\rm Aut}_{X/k}$ be the automorphism functor of $X$ over $k$. 
By \cite[Exp. VIIA, Th\'eor\`eme 7.2]{SGA31}, there exists a bijection
\begin{equation} \label{eq: integrationofvectorfields}
{\rm Hom}(\mathcal{G}_p(\idealg),{\rm Aut}_{X/k}) 
\,\to\, 
{\rm Hom}(\idealg,H^0(X,\Theta_X)),
\end{equation}
where the left-hand side denotes morphisms of group schemes over $k$ and the right-hand side denotes morphisms of $p$-Lie algebras over $k$. 
The quotient of $X$ by $G$ can be computed via $\idealg$-invariants on an open affine cover.
We refer to \cite[Chapter 0.3]{CDL} for details, discussion, and further references.

\begin{Example} \label{ex: integratevectorfield}
In particular, giving a $\bmu_p$-action (resp. $\balpha_p$-action)
on $X$ is the same as giving a $\delta\in H^0(X,\Theta_X)$ with 
$\delta^{[p]}=\delta$ (resp. $\delta^{[p]}=0$).

Let us be more explicit in the case of $G = \balpha_p$: If $X = \Spec R$, then the $\balpha_p$-action induced by $\delta$ is given by the co-action
\begin{eqnarray*}
R & \to & R[x]/(x^p) \\
r & \mapsto & \sum_{i=0}^{p-1} \frac{\delta^i(r)}{i!} x.
\end{eqnarray*}
In particular, note that $r \in R^{\balpha_p}$ if and only if $\delta(r) = 0$.
\end{Example}

We refer to \cite{Ekedahl} and to 
\cite[Section 1]{Schroeer} for further details.

\subsection{The height filtration}
\label{sec: height}

Given a finite and infinitesimal group scheme $G$ over $k$,
the Frobenius morphism yields a canonical decomposition 
of $G$ 
$$
   1\,\unlhd\,G[F]\,\unlhd\,G[F^2]\,\unlhd\,\ldots\,\unlhd\,G[F^n]\,=\,G
$$
for some sufficiently large $n$.
By definition, the minimal $n$ for which we have $G=G[F^n]$ is called
the \emph{height} of $G$.
Each subquotient in this series is of height one.

For later use, let us mention that
a finite group scheme is linearly reductive if and only if
it does not contain $\balpha_p$ or $\C_p$,
see \cite[Lemma 2.3]{LRQ}.

\subsection{Finite group schemes of length $\mathbf{p},\mathbf{p}^{\mathbf{2}}$, or $\mathbf{p}^{\mathbf{3}}$}

Group schemes of length $p$ over an algebraically closed field $k$ of characteristic $p$ are classified in \cite{TateOort} and are all abelian. 
More precisely, they are the constant group scheme $\C_p$, the diagonalisable group scheme $\bmu_p$,
and the group scheme $\balpha_p$, which is neither \'etale nor linearly reductive.

In the following proposition we summarise the classification of finite group schemes of length $p^2$, which is probably known to the experts.

\begin{Proposition}
    \label{prop: simple infinitesimal}
    Let $G$ be a finite group scheme of length $p^2$ over an
    algebraically closed field $k$ of characteristic $p>0$. 
    \begin{enumerate}
      \item If $G$ is abelian, then it is either a product of group schemes of length $p$ 
      or one of the following group schemes
    \begin{enumerate}
    \item $\C_{p^2}$,
	\item $\bmu_{p^2} \cong D(\C_{p^2}) \cong (\C_{p^2})^D$,
	\item $\balpha_{p^2} \cong \bL_{1,2}$ ,
	\item $\balpha_{p^2}^D \cong \bL_{2,1}$, or
	\item $\bM_2 \cong \bL_{2,2}[F-V]$.
   \end{enumerate}
    \item If $G$ is non-abelian, then $G$ is the non-trivial semidirect product $\balpha_p\rtimes\bmu_p$. 
    In particular, $G$ is of height $1$.
    \item If $G$ is infinitesimal, 
      then $G$ is of height $2$ if and only if it is isomorphic to 
      $\balpha_{p^2}$, $\bM_2$ or $\bmu_{p^2}$.
    \end{enumerate}
    In particular,  there exist no simple group schemes of length $p^2$ over $k$.
\end{Proposition}

\begin{proof}
If $G$ is of length $p$ over $k$, then it is commutative and isomorphic 
to $\balpha_p$, $\bmu_p$, or $\C_p$, see \cite{TateOort}.
The classification of abelian group schemes over $k$ is classical, see, for
example \cite{Oort}.
In particular, the only complicated ones are the ones of $(\loc,\loc)$-type,
and for length $p^2$, we refer to \cite[(15.5)]{Oort} for the classification list.
This establishes (1).

Before proceeding, let us note that $\Lieg = k\cdot a\oplus k\cdot m$
with $[m,a]=a$, $a^{[p]}=0$, and $m^{[p]}=m$ defines a non-abelian restricted
Lie algebra.
In fact, this is the restricted Lie algebra of a non-abelian semidirect
product $\balpha_p\rtimes\bmu_p$.

Let $G$ be a non-abelian group scheme of length $p^2$ over $k$.
If it is infinitesimal, then it corresponds to a local Hopf algebra of dimension $p^2$
over $k$.
By \cite[Corollary 7.5]{Wang}, there is only one that is not co-commutative 
(namely, number (5) in the list), which then necessarily corresponds to the previously constructed
semidirect product $\balpha_p\rtimes\bmu_p$.
If $G$ is \'etale and of length $p^2$ over $k$, then it is a constant
group scheme, and thus, by the classification of finite groups of order $p^2$, it is abelian.
It remains to deal with the case where $G$ is non-abelian of length $p^2$
and neither infinitesimal nor \'etale.
Using the local-\'etale sequence, we then obtain a semidirect product 
$G\cong G^\circ\rtimes G^\et$, where both $G^\circ$ and $G^\et$
are of length $p$.
Thus, $G^\circ\in\{\balpha_p,\bmu_p\}$ and $G^\et\iso\C_p$, and we note that
non-abelian semidirect products correspond to non-trivial
homomorphisms  $\C_p\to\Aut(G^\circ)$.
Since $\Aut(\bmu_p)\iso\Aut(\bmu_p^D)\iso(\ZZ/p\ZZ)^\times$, there
are no non-abelian extensions of $\C_p$ by $\bmu_p$.
Since $\Aut(\balpha_p)\iso\GG_m$, 
there are also no non-abelian extensions of $\C_p$ by $\balpha_p$.

From the classification it is clear that there are no simple group schemes of length $p^2$.
The height of an infinitesimal group scheme of length $p^2$ is equal to $2$ if and only if
its Hopf algebra is isomorphic to $k[x]/(x^{p^2})$, which is true only for
$\balpha_{p^2}$, $\bmu_{p^2}$ or $\bM_2$ by \cite[Corollary 7.5]{Wang}.
(Alternatively, one can also argue via their Dieudonn\'e-modules in the abelian
case, and note that the unique non-abelian group scheme of length $p^2$ is
of height $1$).
\end{proof}

The classification of finite group schemes of length $p^3$ is more complicated and 
in the following we will only describe the class of infinitesimal group schemes 
with $1$-dimensional tangent space. 
This restriction is very natural for the purposes of our article 
(see Proposition \ref{prop: linearlyreductive or unipotent}). 
The classification can be derived from the literature on Hopf algebras as follows.

\begin{Proposition} \label{prop: groupschemesoforderp3}
Let $G$ be a finite infinitesimal group scheme with $1$-dimensional tangent space and length $p^3$ over an algebraically closed field $k$ of characteristic $p > 0$.
\begin{enumerate}
    \item If $G$ is abelian, then it is one of the following four group schemes:
    \begin{enumerate}
        \item $\bmu_{p^3}$.  In particular, $G[F^2] \cong G/G[F] \cong \bmu_{p^2}$.
        \item $\balpha_{p^3}$. In particular, $G[F^2] \cong G/G[F] \cong \balpha_{p^2}$.
        \item $\bM_3$. In particular, $G[F^2] \cong G/G[F] \cong \bM_2$.
        \item $\bL_{3,2}[V-F^2]$. In particular, $G[F^2] \cong G/G[F] \cong \balpha_{p^2}$.
    \end{enumerate}
    \item \label{item: non-commutative} If $G$ is non-abelian, then $G[F^2] \cong G/G[F] \cong \balpha_{p^2}$.
    \item \label{item: p=2 non-commutative} If $p = 2$ and $G$ is non-abelian, then $G$ is isomorphic to the subgroup scheme of $\GL_{3,k}$, whose $S$-valued points are as follows
    $$
    G(S) \,=\,\left\{ \left( \begin{matrix} 1 & x^2 & x \\
    0 & 1 & x^4 \\
    0 & 0 & 1 
    \end{matrix} \right) \Bigg| ~ x \in H^0(S,\mathcal{O}_S), ~ x^8 = 0
    \right\}.
    $$ 
    In other words, it is given by $G = \Spec k[x]/(x^8)$ with the comultiplication formula $\Delta(x) = x \otimes 1 + 1 \otimes x + x^2 \otimes x^4$.
\end{enumerate}
\end{Proposition}

\begin{proof}
Write $G = \Spec R$ for a $p^3$-dimensional Hopf algebra $R$. 
Since $R$ is local, $R^\vee$ is connected in the sense that its co-radical is $1$-dimensional.
Moreover, the primitive subspace of $R^\vee$ is $1$-dimensional, see \cite[Proposition 2.2]{Wang}.
Such $R^\vee$ are classified in \cite[Theorem 1.1]{NguyenWangWang}. 
There are four isomorphism classes of commutative $R^\vee$ and one irreducible 
family ($0$-dimensional if $p = 2$)
of isomorphism classes of non-commutative $R^\vee$. 
The four abelian group schemes we listed are pairwise non-isomorphic, 
hence they are in fact all infinitesimal abelian group schemes of order $p^3$ 
with $1$-dimensional tangent space. 
If $G$ is non-abelian, then $G[F^2] \cong G/G[F] \cong \balpha_{p^2}$ is easily 
derived from the description of $R^\vee$ in \cite[Theorem 1.1]{NguyenWangWang}. 
Finally, if $p = 2$, then the non-abelian $G$ is unique, so it coincides with the 
group scheme we describe in (\ref{item: p=2 non-commutative}).
\end{proof}

\begin{Remark}
If $p>2$, then the group schemes of type (\ref{item: non-commutative}) of the proposition even
form one-dimensional families.
It would be very interesting to know whether
quotient singularities by such families of group schemes show up in the 
MMP in higher dimensions.
\end{Remark}

\section{Maximal $\mathcal{C}$-subsheaves} \label{sec: maximalsubsheaves}

Let $G$ be a sheaf of abelian groups on $(Sch/k)_{\fppf}$
and let $\mathcal{C}$ be a category of abelian group schemes over $k$.
Under suitable assumptions on $\mathcal{C}$,
we define a maximal $\mathcal{C}$-subsheaf $G^{\mathcal{C}}$,
which is the smallest subsheaf of $G$ that receives all morphisms
from group schemes in $\mathcal{C}$.
We use this to define a canonical decomposition of $G$
into factors of type (loc,loc), (loc, \'et), (\'et,loc), and (\'et,\'et) 
as familiar from finite abelian group schemes.
Moreover, we set up a Cartier--Dieudonn\'e theory for 
$G^{\loc,\loc}$.
The motivation for developing this machinery is that we want to 
study the local Picard sheaf $\Picloc_{X/k}$ (Section \ref{subsec: picloc}) for a singularity
$x\in X$ and that this sheaf is usually \emph{not} representable
by a group scheme.

\subsection{Sheaves of abelian groups}

In this section, $k$ is a field, $G$ is an fppf sheaf of abelian groups on 
$(Sch/k)_{\fppf}$ and $\mathcal{C}$ is a category of abelian group schemes over $k$ 
such that
\begin{enumerate}
   \item for $G_1,G_2 \in \mathcal{C}$, also $G_1 \times G_2 \in \mathcal{C}$, and
   \item for $G_1 \in \mathcal{C}$ and for every normal subgroup scheme $H \subseteq G_1$ we have $G_1/H \in \mathcal{C}$.
\end{enumerate}
That is, $\mathcal{C}$ is stable under taking finite products and quotients.

\begin{Remark}
 Since we are working over a field $k$, the quotient $G/H$ of a group scheme $G$ by a normal subgroup 
 scheme $H$ exists and it represents the categorical quotient in the category of fppf sheaves.
 This can be seen by combining \cite[Tag 047T]{StacksProject} with 
 \cite[Corollaire 4.2.3]{Perrin}.
\end{Remark}

\begin{Example}
The conditions on $\mathcal{C}$ hold, for example, for the category of 
finite abelian group schemes over $k$, the category of geometrically reduced abelian group schemes 
over $k$, 
the category of finite and linearly reductive abelian group schemes over $k$, 
or the category of finite abelian group schemes of $(\loc,\loc)$-type over $k$.
\end{Example}

We now come to the main definition of this section.

\begin{Definition}\label{def: maximal C-subsheaf}
Let $G$ be an fppf sheaf of abelian groups on $(Sch/k)_{\fppf}$ and let $\mathcal{C}$ be as above. 
The \emph{maximal $\mathcal{C}$-subsheaf} $G^{\mathcal{C}}$ of $G$ is the subsheaf of $G$ generated 
by the images of all homomorphisms of sheaves of abelian groups from group schemes in 
$\mathcal{C}$ to $G$.
\end{Definition}

By construction, $G^{\mathcal{C}}$ is the smallest subsheaf of $G$ that receives 
all morphisms from group schemes in $\mathcal{C}$. 
We would like to say that $G^{\mathcal{C}}$ is itself a union of subsheaves 
that are representable by objects in $\mathcal{C}$. 
We can prove this in the following case, which is
sufficient for our applications below.

\begin{Proposition}\label{prop: maximal C-subsheaf}
Let $G$ be a sheaf of abelian groups on $(Sch/k)_{\fppf}$,
such that the identity section $e_G:\Spec k\to G$ 
is representable by a closed immersion. 
Then, we have
$$
G^{\mathcal{C}} \,=\, \bigcup_{i \in I}\, G_i,
$$
where $G_i$ runs over all subsheaves of $G$ that are representable by 
objects in $\mathcal{C}$. 
\end{Proposition}

\begin{proof}
Since $e_G$ is representable by a closed immersion, the kernel of every 
homomorphism $f: H \to G$ from a group scheme $H \in \mathcal{C}$ is a group scheme, 
so $H/\Ker(f)$ is a group scheme as well. 
Moreover, by our assumptions on $\mathcal{C}$, we have $H/\Ker(f) \in \mathcal{C}$ 
and since $G$ is a sheaf and $H/\Ker(f)$ is the fppf-quotient of $H$ by $\Ker(f)$, 
$f$ factors through a monomorphism of sheaves of abelian groups
$\alpha': H/\Ker(f) \to G$. 
In particular, $\bigcup_{i \in I} G_i \subseteq G^{\mathcal{C}}$ and $G^{\mathcal{C}}$ 
is generated by the $G_i$.

To finish the proof, it thus suffices to show that $\bigcup_{i \in I} G_i(S)$ 
is a subgroup of $G(S)$ for every $k$-scheme $S$. 
For this, we take two elements $g_i \in G_i(S)$ and $g_j \in G_j(S)$. 
The inclusions $f_i: G_i \to G$ yield homomorphisms $f_{ij} := f_i + f_j: G_i \times G_j \to G$. 
As in the previous paragraph, $\Ker(f_{ij})$ is representable. 
Thus, by our assumptions on $\mathcal{C}$, both 
$G_i \times G_j$ and $(G_i \times G_j)/\Ker(f_{ij})$ 
are representable by group schemes in $\mathcal{C}$. 
Hence, ${\rm im}(f_{ij}) = G_k$ for some $k$. 
Thus, $g_i + g_j \in G_i(S) + G_j(S) \subseteq G_k(S) \subseteq \bigcup_{i \in I} G_i$. 
\end{proof}

In particular, $G^\mathcal{C}$ is the filtered colimit (in the category of sheaves 
of abelian groups on $(Sch/k)_{\fppf}$) over those subsheaves $G_i$ 
that are representable by objects in $\mathcal{C}$. 
Therefore, under suitable assumptions on $\mathcal{C}$, $G^{\mathcal{C}}$ 
is an ind-group scheme.

\begin{Lemma} \label{lem: indgroupscheme}
Let $G$ be a sheaf of abelian groups on $(Sch/k)_{\fppf}$,
such that the identity section $e_G:\Spec k\to G$ is representable by a closed immersion. 
If all objects in $\mathcal{C}$ are affine, then $G^{\mathcal{C}}$ is an ind-group scheme.
\end{Lemma}

\begin{proof}
We already know that $G^{\mathcal{C}}$ is a filtered colimit over those subsheaves $G_i$ 
that are representable by objects in $\mathcal{C}$. 
If all the $G_i$ are affine, then the morphisms between them are closed immersions 
(see, for example, \cite[Corollaire 4.2.4]{Perrin}), and then, 
$G^\mathcal{C}$ is an ind-group scheme in this case.
\end{proof}

If $k$ is perfect, then we have the functorial decomposition 
\eqref{eq: canonical decomposition} of finite and abelian group schemes.
Hence, for an fppf sheaf of abelian groups $G$ over $k$, such that $e_G$ is 
representable by a closed immersion and if $\mathcal{C}$ is the category of 
finite abelian group schemes over $k$, then 
we get the analogous canonical decomposition of ind-group 
schemes
$$
G^{\mathcal{C}} \,\cong\, G^{\loc,\loc} \times G^{\loc,\et} \times G^{\et,\loc} \times G^{\loc,\loc}.
$$
As with finite group schemes, if $k$ is non-perfect, then we only get a 
canonical filtration of $G^{\mathcal{C}}$ with graded pieces in the 
four subcategories.

\begin{Definition}
A sheaf $G$ of abelian groups on $(Sch/k)_{\fppf}$ is said to be an 
\emph{ind-(loc,loc) group scheme} if $G = G^{\loc,\loc}$.
\end{Definition}

\subsection{Cartier--Dieudonn\'e theory} \label{subsec: Cartier--Dieudonne theory}
In this section, we describe the maximal $(\loc,\loc)$-subsheaf $G^{\loc,\loc}$ of $G$ using Dieudonn\'e theory.

\begin{Definition}
Let $G$ be an fppf sheaf of abelian groups on $(Sch/k)_{\fppf}$.
We define the \emph{(covariant (loc,loc)-) Dieudonn\'e module of $G$} 
to be 
$$
 \Dieulocloc(G) \,:=\, \varinjlim_{n,m}\, \Hom\left(\bL_{n,m},G\right),
$$
which is a left $\Dieu$-module via precomposition with the right-action of $\Dieu$ on $\bL_{n,m}$. 
\end{Definition}

If $G$ is the fppf sheaf associated to some finite group scheme over $k$,
then $\Dieulocloc(G)$ coincides with the $\Dieu$-module associated
to that particular group scheme as discussed in Section \ref{subsec: loclocCartierDieudonne}.
Moreover, since $G^{\loc,\loc}$ receives all morphisms from 
finite $(\loc,\loc)$-group schemes, we have 
$\Dieulocloc(G) = \Dieulocloc(G^{\loc,\loc})$.
The main theorem of covariant Dieudonn\'e theory extends 
to ind-group schemes of $(\loc,\loc)$-type as follows: 

\begin{Proposition} \label{prop: colimitsofdieudonne}
Let $k$ be a perfect field of characteristic $p>0$. 
Then, the functor $\Dieulocloc$ induces an equivalence of categories 
between the category of ind-(loc,loc) group schemes over $k$ and 
colimits of left $\Dieu$-modules that are finite as $W(k)$-modules and
on which $F$ and $V$ are nilpotent.
\end{Proposition}

\begin{proof}
Let $G = \varinjlim_i G_i$ and $H = \varinjlim_j H_j$ 
be ind-(loc,loc) group schemes over $k$ presented as colimits
over finite (loc,loc) group schemes.

First, we show essential surjectivity of the functor $\Dieulocloc$.
For this, we use that the $\bL_{n,m}$ are finite, hence quasi-compact in the fppf-topology, 
and thus, $\Hom(\bL_{n,m},G) = \varinjlim_i \Hom(\bL_{n,m},G_i)$,
see \cite[Tag 090G]{StacksProject}. 
Therefore,
\begin{eqnarray*}
 \Dieulocloc(G) &=&  \Dieulocloc(\varinjlim_i G_i) 
 = \varinjlim_{n,m}\, \Hom(\bL_{n,m}, \varinjlim_i  G_i)  \\
 &=& \varinjlim_{n,m}\, \varinjlim_i\, \Hom(\bL_{n,m}, G_i)  \,=\, \varinjlim_i\, \varinjlim_{n,m}\, \Hom(\bL_{n,m}, G_i)  \\
 &=& \varinjlim_{i}\, \Dieulocloc(G_i).
\end{eqnarray*}
Thus, essential surjectivity follows from classical Dieudonn\'e theory, see 
Theorem \ref{thm: Dieudonnemain}.

To prove that $\Dieulocloc$ is fully faithful, we use the following facts:
\begin{enumerate}
    \item $\Hom(G_i,H_j) = \Hom(\Dieulocloc(G_i),\Dieulocloc(H_j))$ by 
    Theorem \ref{thm: Dieudonnemain},
    \item being finite, the $G_i$ are quasi-compact in the fppf-topology,
    \item $\Dieulocloc(G_i)$ is finitely presented as a $\Dieu$-module,
    \item $\Dieulocloc(G) = \varinjlim_i \Dieulocloc(G_i)$ and 
    $\Dieulocloc(H) = \varinjlim_j \Dieulocloc(H_j)$.
\end{enumerate}
Using these, we deduce that
\begin{eqnarray*}
\Hom(G,H) &=& \Hom(\varinjlim_i G_i, \varinjlim_j H_j) \,=\, 
\varprojlim_i \Hom(G_i,\varinjlim_j H_j) \\
&\overset{(2)}{=}& \varprojlim_i\, \varinjlim_j\, \Hom(G_i,H_j) \\
&\overset{(1)}{\cong}&  \varprojlim_i\, \varinjlim_j\, \Hom(\Dieulocloc(G_i),\Dieulocloc(H_j)) \\
&\overset{(3)}{=}& \varprojlim_i \Hom(\Dieulocloc(G_i), \varinjlim_j \Dieulocloc(H_j)) \\
&=& \Hom(\varinjlim_i \Dieulocloc(G_i), \varinjlim_j \Dieulocloc(H_j)) \\
&\overset{(4)}{\cong} & \Hom(\Dieulocloc(G),\Dieulocloc(H))
\end{eqnarray*}
so $\Dieulocloc$ is fully faithful, hence an equivalence of categories.
\end{proof}

In particular, we have the following:

\begin{Corollary}\label{cor: finiteifdieufinite}
Let $G$ be a sheaf of abelian groups on $(Sch/k)_{\fppf}$, such that the 
identity section $e_G:\Spec k\to G$ is representable by a closed immersion.
Then, $G^{\loc,\loc}$ is a finite group scheme if and only if 
$\Dieulocloc(G)$ is a finite $W(k)$-module.
\end{Corollary}

\section{Local torsors and how to detect them}
\label{sec: local torsors}

In this section, we introduce the notion of \emph{local torsors} over a
singularity $x\in X$ and show how to detect and classify them.
On our way, we discuss a number of subtleties. We will be working over an algebraically closed field $k$ and refer the reader to Section \ref{sec: Mumford} for a discussion of the situation over non-closed fields.

\begin{Definition} \label{def: singularity}
A \emph{($d$-dimensional) singularity} is a pair $x\in X$, where $X=\Spec R$, where $(R,\idealm)$ is 
an integral, local, complete, and Noetherian $k$-algebra of 
dimension $d \geq 2$ that satisifies Serre's condition $S_2$, 
and where $x$ is the closed point corresponding to $\idealm$. 
In this context, we will set $U:=X\backslash\{x\}$.
\end{Definition}

Let $G$ be a finite group scheme over $k$.

\subsection{Local torsors}
First, we describe $G$-torsors over $U$.
Abstractly, these are classified by the first flat cohomology group
$\Hfl{1}(U,G)$:
if $G$ is abelian, then this is an abelian group, whereas if $G$ is non-abelian,
then this is merely a pointed set (the distinguished point corresponds
to the trivial $G$-torsor).
We note that if $G$ is \'etale over $k$, then the natural morphism
$\Het{*}(-,G)\to \Hfl{*}(-,G)$ is an isomorphism.
Next, the canonical inclusion $\imath:U\to X$ induces a pullback map
\begin{equation}
 \label{eq: imath}
   \imath^* \,:\, \Hfl{1}(X,G) \,\to\, \Hfl{1}(U,G)
\end{equation}
of pointed sets and of abelian groups if $G$ is abelian. This pullback map is injective by Lemma \ref{lem: carvajal-rojas} below. 

If $G$ is \'etale over $k$, then $\Hfl{1}(X,G)$ is trivial, see Section \ref{subsec: etale} below,
but if $G$ is not \'etale, then this is not necessarily the case, 
which motivates the following definition.

\begin{Definition} \label{def: local torsor}
Let $x\in X=\Spec R$ be a singularity in the sense of Definition \ref{def: singularity}.
  Let $G$ be a finite group scheme over $k$. 
  \begin{enumerate}
      \item  A \emph{local $G$-torsor} over $X$ is a $G$-torsor over $U$. In particular, elements of $\Hfl{1}(U,G)$ are isomorphism classes of local $G$-torsors over $X$. 
   \item If $G$ is abelian or if $\Hfl{1}(X,G)$ is trivial, then two local $G$-torsors are 
     called \emph{equivalent} if their classes in $\Hfl{1}(U,G)$ differ by an element of $\Hfl{1}(X,G)$. 
     In this case, we write $\overline{\Hfl{1}(U,G)}$ for the set of equivalence classes of local $G$-torsors over $X$.
  \end{enumerate}
\end{Definition}

\begin{Remark}
If $G$ is abelian, then the set $\overline{\Hfl{1}(U,G)}$ coincides with 
$\Coker(\imath^*)$ in the category of abelian groups.

If $G$ is non-abelian and $\Hfl{1}(X,G)$ is non-trivial, 
then we cannot even talk about the ``difference'' between two classes in $\Hfl{1}(U,G)$, 
hence we cannot talk about $\overline{\Hfl{1}(U,G)}$. 
Also, note that $\Coker(\imath^*)$ in the category of pointed sets is not a suitable 
replacement for $\overline{\Hfl{1}(U,G)}$.
\end{Remark}

To explain our terminology, let $V\to U$ be a local $G$-torsor as in the definition.
In this case, the \emph{integral closure} of $X$ in $V$ 
(see \cite[Corollaire 6.3.2, Proposition 6.3.4]{EGA2}) is 
given by $\Spec H^0(V, \mathcal{O}_V) \to X$ and it is a natural finite morphism,
which restricts to the given $G$-torsor over $U$.
By abuse of terminology, we will also use the term \emph{local $G$-torsor over $X$} 
for the integral closure of a local $G$-torsor $V \to U$ over $X$ 
as in Definition \ref{def: local torsor} together with the $G$-action on $V$. 

While it is clear from the universal properties of integral closures that a morphism between
$G$-torsors over $U$ extends to a morphism between their integral closures, it is not 
immediately clear whether the action of a (non-constant) finite group scheme extends to the 
integral closure.
This is indeed the case, because $X$ satisfies Serre's condition $S_2$, which is the 
content of the following lemma, which can be found in \cite[Lemma 4.1]{Carvajal-Rojas}.

\begin{Lemma}[Carvajal-Rojas] \label{lem: carvajal-rojas}
Let $V \to U$ be a $G$-torsor and let $Y$ be the integral closure of $X$ in $V$. 
Then,
\begin{enumerate}
\item \label{item: Y S2} $Y$ is $S_2$, and
\item the $G$-action on $V$ extends to a unique $G$-action on $Y$.
\end{enumerate}
\end{Lemma}

\begin{Example}
Let $V \to U$ be a $G$-torsor and assume that the integral closure $Y$ of $X$ in $V$ 
is a regular local ring. 
Then, $Y \cong \Spec k[[u_1,\hdots,u_d]]$ and the $G$-action on $Y$, whose existence is 
guaranteed by Lemma \ref{lem: carvajal-rojas}, is free outside the closed point. 

Conversely, given a $G$-action on $Y := \Spec k[[u_1,\hdots,u_d]]$ that is free outside 
the closed point, the quotient morphism $Y \to Y/G =: X$ is (the integral closure of) 
a local $G$-torsor over $X$.

This is the connection between local torsors over $X$ and the question whether $X$ is 
a quotient singularity, which we will use to study quotient and non-quotient RDPs 
in Section \ref{sec: rdp}.
\end{Example}

\begin{Remark} \label{rem: quotientisnormal}
It is a basic fact that singularities that are quotients of normal singularities by a finite group scheme action are again normal. Conversely, Lemma \ref{lem: carvajal-rojas} shows that local torsors over normal singularities are $S_2$. However, we warn the reader that a local torsor over a normal singularity may not be normal. For an explicit example of this phenomenon, see Example \ref{ex: non-normal}. 
\end{Remark}

Note that a morphism between local $G$-torsors over $X$ is, by definition, a $G$-equivariant morphism over $X$. In the following two remarks, we shall see that this implies that both the automorphism group of $G_X$ and the automorphism group of $X$ act on the set of local $G$-torsors over $X$.

\begin{Remark}\label{rem: reductionofstructure}
A homomorphism of group schemes $\varphi: H\to G$ over $k$
induces a homomorphism of group schemes $\varphi:H_X\to G_X$ over $X$,
which induces a homomorphism 
of pointed sets $\varphi_*: \Hfl{1}(X,H) \to \Hfl{1}(X,G)$. 
(Note that a special feature of the group scheme situation is that the 
natural injective homomorphism $\Hom(H,G) \to \Hom(H_X,G_X)$ might be very far away 
from being surjective. 
For example, if $H = G = \balpha_p$, then $\Hom(H,G) \cong k$,
while $\Hom(H_X,G_X) \cong R$.) 

If $H$ and $G$ are abelian, then we also get an induced homomorphism of abelian
groups
$$
  \varphi_*\,:\,  \overline{\Hfl{1}(U,H)} \,\to\, \overline{\Hfl{1}(U,G)}.
$$
An equivalence class of local $G$-torsors lies in the image of $\varphi_*$ 
if and only if it contains a local $G$-torsor that admits a reduction of 
structure group to the image of $H$ under $\varphi$. Here, a $G$-torsor $Z \to U$ is said to admit a reduction of structure group to the image of $H$ under $\varphi$, if there exists an $H$-torsor $Z' \to U$ and a $\varphi$-equivariant morphism $Z' \to Z$ over $U$.
\end{Remark}

Equivalence classes of local torsors that cannot be represented by local torsors that admit a reduction of structure group to proper subgroup schemes play a special role in this article (see Proposition \ref{prop: primitive class}). In the following, we define such ``primitive'' classes.

\begin{Definition} \label{def: primitive}
Let $x\in X=\Spec R$ be a singularity in the sense of Definition \ref{def: singularity}.
Let $G$ be a finite and abelian group scheme. The set of \emph{primitive} equivalence classes of local $G$-torsors over $X$ is defined as
$$
\overline{\Hfl{1}(U, G)}_{\prim} \,:=\, 
\overline{\Hfl{1}(U, G)} \setminus 
 \left(\bigcup_{H \subsetneq G} \overline{\Hfl{1}(U, H)}\right).
$$
A \emph{local $G$-torsor of primitive class} is a local $G$-torsor, whose class in 
$\overline{\Hfl{1}(U, G)}$ is contained in $\overline{\Hfl{1}(U, G)}_{\prim}$. 
\end{Definition}

\begin{Remark}
In particular, if $G$ is simple and abelian (for example, if $G \in \{\bmu_p,\balpha_p\})$, 
then $\overline{\Hfl{1}(U, G)}_{\prim} =  \overline{\Hfl{1}(U, G)} \setminus \{0\}$ and the class 
of a local $G$-torsor $Y \to X$ is primitive if and only if it 
``does not extend to a global torsor over $X$''.
If $G$ is abelian but not simple, then the class of a local $G$-torsor that does not extend to $X$ 
is not necessarily primitive. 
For example, if $Y \to X$ is a local $\bmu_2$-torsor that does not extend to $X$, then its extension 
of structure group along $\bmu_2 \to \bmu_4$ still does not extend to $X$, but its class
in $\overline{\Hfl{1}(U,\bmu_4)}$ is no longer primitive.  
\end{Remark}

\begin{Remark}\label{rem: automorphismsoftorsors}
There are two groups acting naturally on $\Hfl{1}(U,G)$ and
$\Hfl{1}(X,G)$:
\begin{enumerate}
    \item By Remark \ref{rem: reductionofstructure}, we have a natural action of 
    $\Aut(G_X)$
    on both via reduction of structure group along automorphisms of $G_X$.
    \item The group $\Aut(X,x)$ of automorphisms of $X$ preserving $x$ acts via pullbacks on 
    $\Hfl{1}(U,G)$ and $\Hfl{1}(X,G)$.
\end{enumerate}
If $G$ is abelian, then these two group actions descend to actions on the abelian
group $\overline{\Hfl{1}(U,G)}$ and its subset $\overline{\Hfl{1}(U, G)}_{\prim}$.
\end{Remark}

\begin{Example} \label{ex: endomorphismsofalphap}
If $G = \balpha_p$, then $\End(G_X) = R$ and $\Aut(G_X) = R^\times$. 
Reduction of structure group along endomorphisms of $G_X$ turns  $\Hfl{1}(U,G)$,  
$\Hfl{1}(X,G)$, and $\overline{\Hfl{1}(U,G)}$ into $R$-modules. 
In Corollary \ref{cor: torslocloc}, we will see that the $R$-module 
$\overline{\Hfl{1}(U,G)}$ is naturally isomorphic to $H^2_{\idealm}(X,\OO_X)[F]$, 
the kernel of Frobenius on the second local cohomology group of $X$. 
In particular, two classes in $H^2_{\idealm}(X,\OO_X)[F]$ that are the same up to 
multiplication by a unit in $R$ can be represented by local $\balpha_p$-torsors over $X$
that differ by an automorphism of $\balpha_{p,X}$. 
We will exploit this fact in Section \ref{sec: rdp} to describe the total spaces of 
local $\balpha_{p}$-torsors over certain RDPs.
\end{Example}

We will now recall several objects parametrising (equivalence classes of) 
local $G$-torsors for \'etale, abelian, and general finite group schemes $G$, respectively.

\subsection{\'Etale torsors: the local \'etale fundamental group}
\label{subsec: etale}

Let $G$ be a finite and \'etale group scheme over $k$.
Since $k$ is assumed to be algebraically closed, $G$ is the constant group scheme associated to a finite group. 
By Hensel's lemma \cite[Proposition 5]{Nagata50} and since $k$ is algebraically closed, 
$X$ admits no non-trivial finite \'etale covers by \cite[Proposition 18.8.1]{EGA4}. 
In particular, $\Hfl{1}(X,G)$ is trivial and thus, each equivalence class of local $G$-torsors 
over $X$ contains a unique isomorphism class of local $G$-torsors.

By \cite{SGA1}, there exist \emph{Galois categories} classifying torsors
under \emph{all} finite and \'etale group schemes over $k$ over $X$ and $U$,
which leads to the \emph{\'etale fundamental groups} $\piet(X,x)$ and $\piet(U,u)$.
Here, $x\in X$ and $u\in U$ are base-points that have to be chosen
to begin with.
For example, if we choose the geometric generic points $\overline{\eta}$ 
of $U$ and $X$
as base-points, then we obtain a continuous homomorphism of profinite groups
$$
 \piet(U,\overline{\eta}) \,\to\, \piet(X,\overline{\eta}).
$$
Since there are no non-trivial torsors under finite and \'etale group 
schemes over $X$, we have $\piet(X,\overline{\eta})=\{e\}$, that is, 
$X$ is algebraically simply connected.
The group 
$$
  \pietloc(X) \,:=\, \piet(U,\overline{\eta})
$$
is called the \emph{local (\'etale) fundamental group} of $X$ and classifies 
local $G$-torsors in the following sense, see
\cite[Proposition 5.1]{LRQ}.

\begin{Proposition} \label{prop: torsfinet}
Let $x\in X=\Spec R$ be a singularity in the sense of Definition \ref{def: singularity}.
If $G$ is finite and \'etale, then there is a canonical 
bijection of pointed sets (resp. isomorphism of abelian groups if $G$ is abelian)
     $$
    \Hfl{1}(U,G) \, = \, \overline{\Hfl{1}(U,G)} 
      \quad\cong\quad
      \Hom(\pietloc(X),G)/\sim,
     $$
where the right hand side denotes homomorphisms of profinite groups
modulo inner automorphisms of $G$.
\end{Proposition}

\subsection{Abelian torsors: the local Picard scheme} 
\label{subsec: picloc}
Let $G$ be a finite and abelian group scheme over $k$ and 
let $G^D$ be its Cartier dual group scheme.

There is a notion of \emph{local Picard functor} $\Picloc_{X/k}$ discussed by 
Boutot \cite{Boutot}.
This functor is representable by a scheme under finite-dimensionality assumptions 
on certain local cohomology groups, which are satisfied, for example, 
if $d = \dim X \geq 3$, $U$ is smooth, and $X$ is Cohen--Macaulay. 
In dimension $2$, the local Picard functor $\Picloc_{X/k}$ is usually 
not representable.
However, it is still an fppf-sheaf \cite[Corollaire II.2.4]{Boutot},
whose identity section is representable by a closed immersion 
\cite[Th\'eor\`eme 6.6.]{Boutot}.
By \cite[Corollaire III.4.9]{Boutot}, it classifies local $G$-torsors 
in the following sense.

\begin{Proposition} \label{prop: torsfinab}
Let $x\in X=\Spec R$ be a singularity in the sense of Definition \ref{def: singularity}.
If $G$ is finite and abelian, then there is a canonical isomorphism of abelian groups
     $$
     \overline{\Hfl{1}(U,G)}   
      \quad\cong\quad 
      \Hom(G^D,\Picloc_{X/k}),
     $$
where the right hand side denotes morphisms of sheaves of abelian groups.
\end{Proposition}

Even though $\Picloc_{X/k}$ is not representable in general if $d = 2$, its restriction
$(\Picloc_{X/k})_{\reduced}$ to normal $k$-schemes is representable 
if $X$ is normal by \cite[Corollaire V.1.8.1]{Boutot}.  
More precisely, if $X$ is a normal surface singularity, then $(\Picloc_{X/k})_{\reduced}$ 
is representable by a smooth group scheme locally of finite type over $k$, whose 
tangent space is isomorphic to $H^1(\widetilde{X},\mathcal{O}_{\widetilde{X}})$, 
where $\widetilde{X}$ is the minimal resolution of $X$, see 
\cite[Corollaire IV.2.9]{Boutot}.
Moreover, by \cite[Proposition II.3.2]{Boutot},
$$
\Picloc_{X/k}(k) \,=\, (\Picloc_{X/k})_{\reduced}(k) \,=\, \Pic(U) \,=\, {\rm Cl}(X)
$$
is nothing but the \emph{class group} of $X$. 
Hence, we have the following addition to Proposition \ref{prop: torsfinab}.

\begin{Proposition} \label{prop: torsfinabcl}
Let $x\in X=\Spec R$ be a singularity in the sense of Definition \ref{def: singularity}.
If $G$ is finite and abelian and $G^D$ is \'etale, then there is a canonical 
isomorphism of abelian groups
     $$
     \overline{\Hfl{1}(U,G)} 
      \quad\cong\quad 
      \Hom(G^D,{\rm Cl}(X)),
     $$
where the right hand side denotes morphisms of abelian groups.
\end{Proposition}

In Section \ref{sec: maximalsubsheaves}, we studied sheaves $G$ 
of abelian groups on $(Sch/k)_{\fppf}$.
To a category $\mathcal{C}$ of abelian group schemes that is closed under products and 
quotients, we associated a maximal $\mathcal{C}$-subsheaf $G^{\mathcal{C}}$ of $G$,
see Definition \ref{def: maximal C-subsheaf}.
If the identity section $e_G:\Spec k\to G$ is representable by a closed immersion, 
then Proposition \ref{prop: maximal C-subsheaf}
gives a description of $G^{\mathcal{C}}$.
In particular, the results of Section \ref{sec: maximalsubsheaves} apply to the fppf sheaf
$\Picloc_{X/k}$ and we gather them in the following.

\begin{Proposition}\label{prop: decompositionofpicloc}
Let $x\in X=\Spec R$ be a singularity in the sense of Definition \ref{def: singularity}. Assume that $d \geq 3$ or that $d = 2$ and $X$ is normal. Let $\Picloc_{X/k}$ be the local Picard sheaf.
  \begin{enumerate}
      \item If $\mathcal{C}$ is the category of constant abelian group schemes, 
      then $\Picloc_{X/k}^{\mathcal{C}}$ is (the constant sheaf associated to) the class group of $R$.
      \item If $\mathcal{C}$ is the category of finite diagonalisable group schemes, 
      then $\Picloc_{X/k}^{\mathcal{C}}$ is the colimit over the Cartier duals of the 
      finite abelian quotients of $\pietloc(X)$.
      \item If $k$ is perfect and $\mathcal{C}$ is the category of reduced abelian group schemes, 
      then $\Picloc_{X/k}^{\mathcal{C}}$ coincides with $(\Picloc_{X/k})_{\reduced}$.
      \item If $k$ is perfect and $\mathcal{C}$ is the category finite abelian group schemes,
      then we obtain a canonical decomposition
\begin{eqnarray*}
\Picloc_{X/k}^{\mathcal{C}} &\cong& 
 \Picloclocloc_{X/k} \quad\times\quad 
 \varinjlim_n\, \left(\pietlocab(X)/(G_n)\right)^D \quad\times \\
 & & \varinjlim_n\,  \underline{{\rm Cl}(X)[p^n]} \quad\times\quad 
 \varinjlim_{p \nmid m}\, \underline{{\rm Cl}(X)[m]},
\end{eqnarray*}
where $\pietlocab(X)/(G_n)$ runs over all finite $p$-torsion quotients 
of $\pietlocab(X)$. 
  \end{enumerate}
\end{Proposition}

The class group ${\rm Cl}(X)$, the local fundamental group $\pietloc(X)$, and 
the Picard variety $(\Picloc_{X/k})_{\reduced}^\circ$ are ``classical'' invariants and are 
known for many singularities. It remains to deal with $\Picloclocloc_{X/k}$.

\subsection{Abelian loc-loc torsors: local Witt vector cohomology}
\label{subsec: abelian loc loc}
In this subsection, we assume ${\rm char}(k) = p > 0$.  
For a finite and abelian group scheme $G$ over $k$, we have the canonical 
decomposition \eqref{eq: canonical decomposition}.
By Proposition \ref{prop: decompositionofpicloc} and Proposition \ref{prop: torsfinab}, 
we have fairly explicit descriptions of local \mbox{$G^{\et,\loc}$-,} $G^{\et,\et}$-, and  $G^{\loc,\et}$-torsors. 
The purpose of this section is to study local $G^{\loc,\loc}$-torsors.

We will describe $\Picloclocloc_{X/k}$ in terms of local Witt vector cohomology on $X$, 
making it amenable to computations. 
By Proposition \ref{prop: colimitsofdieudonne} and Proposition \ref{prop: torsfinab}, 
we have isomorphisms of $W(k)$-modules
$$
\begin{array}{lclcl}
\Dieu(\Picloclocloc_{X/k}) &=& \Dieu(\Picloc_{X/k}) &\cong& 
\varinjlim_{n,m}\, \Hom(\bL_{n,m},\Picloc_{X/k})  \\ 
&\cong & \varinjlim_{n,m}\, \overline{\Hfl{1}(U,\bL_{n,m}^D)} &\cong& 
\varinjlim_{n,m} \overline{\Hfl{1}(U,\bL_{m,n})}.   
\end{array}
$$
So, to understand $\Picloclocloc_{X/k}$, we have to understand local $\bL_{n,m}$-torsors over $X$.

\begin{Proposition} \label{prop: torslocloc}
 There is a canonical isomorphism of abelian groups
    $$
    \overline{\Hfl{1}(U,\bL_{n,m})} 
      \quad\cong\quad
     H^2_{\idealm}(X,W_n\OO_X)[F^m].
     $$
\end{Proposition}

\begin{proof}
We write $\bW_n$ for the group scheme of Witt vectors of length $n$ 
and $W_n(R)$ (resp. $W_n \mathcal{O}_X$) for the ring (resp. sheaf) 
of Witt vectors of length $n$.
By definition, there is a short exact sequence of group schemes
\begin{equation}\label{seslnm}
  0\,\to\,\bL_{n,m}\,\to\,\bW_n\,\stackrel{F^m}{\longrightarrow}\,\bW_n\,\to\,0.
\end{equation}
Since $R$ is $S_2$, we have $W_n(R) =\Hfl{0}(X,\bW_n) = \Hfl{0}(U,\bW_n)$ and since $\bW_n$ 
is an iterated extension of $\GG_a$'s, 
we know that $\Hfl{1}(X,\bW_n) = H^1(X,W_n\mathcal{O}_X) = 0$ using Serre vanishing of the cohomology of coherent sheaves on 
affine schemes. 
Therefore, from the long exact sequence in flat cohomology on $X$ associated to 
the short exact sequence (\ref{seslnm}), we deduce 
$\Hfl{1}(X,\bL_{n,m}) = W_n(R)/F^m(W_n(R)) = H^0(U,W_n\mathcal{O}_U)/F^m(H^0(U,W_n\mathcal{O}_U))$.

Hence, from the long exact sequence in flat cohomology on $U$ associated to the short 
exact sequence (\ref{seslnm}), we obtain the short exact sequence
\begin{equation}\label{seslnm2}
0\,\to\, \Hfl{1}(X,\bL_{n,m}) \,\stackrel{\mathrm{res}}{\longrightarrow}\, \Hfl{1}(U,\bL_{n,m})  \,\to\, \Hfl{1}(U,\bW_n)[F^m] \,\to\, 0
\end{equation}
The long exact sequence in local cohomology shows that 
$\Hfl{1}(U,\bW_n)= H^1(U,W_n\mathcal{O}_U) = H^2_{\idealm}(X,W_n \OO_X)$.
\end{proof}

\begin{Remark}[Concerning the $\Dieu$-module structures in Proposition \ref{prop: torslocloc}] \label{rem: dieustructures} 
When we recalled Dieudonn\'e theory in Section \ref{subsec: loclocCartierDieudonne}, 
we have chosen a $\Dieu$-action on $\bL_{n,m}$ to turn $\Hom(\bL_{n,m},G)$ 
into a $\Dieu$-module for any $G$. 
Chasing through the various identifications, we obtain a left $\Dieu$-module structure 
on $H^2_{\idealm}(X,W_n\OO_X)[F^m]$. 
It is the left-module structure where $F$ and $V$ in $\Dieu$ act as the endomorphisms 
$F$ and $V$ on $H^2_{\idealm}(X,W_n\OO_X)[F^m]$ induced by $F$ and $V$ on $W_n\OO_X$ and 
the scalar multiplication is chosen so that $\{H^2_{\idealm}(X,W_n\OO_X)[F^m]\}_{m,n}$ 
becomes a directed system of left $\Dieu$-modules with respect to $I$ and $V$. 
For the reader's convenience, we summarise the steps necessary for this translation 
in the following table. 

\begin{tabular}{l|lllll}
& & compat.\ with & $a$ & $F$ & $V$ \\ \hline
$\bL_{n,m}$       & left  & $V,I$     & $\sigma^{-n}(a)$ & $F$ & $V$ \\
$\bL_{n,m}^D$     & right & $V^D,I^D$ & $\sigma^{-n}(a)$ & $F^D$ & $V^D$ \\
$\Hom(\bL_{n,m}^D, \Picloc_{X/k})$ & left & $ \circ V^D,  \circ I^D$& $ \circ \sigma^{-n}(a)$ & $ \circ F^D$ & $ \circ V^D$ \\
$H^2_{\idealm}(X,W_n \OO_X)[F^m]$ & left & $V,I$ & $\sigma^{-n}(a)$ & $F$ & $V$ 
\end{tabular}

\end{Remark}

Therefore, with respect to the $\Dieu$-actions of Remark \ref{rem: dieustructures}, we have the following corollary.

\begin{Corollary} \label{cor: torsloclocwithdieu}
 There is a canonical isomorphism of left $\Dieu$-modules
      $$
    \Hom(\bL_{n,m}^D,\Picloc_{X/k}) \,\cong\,
    \overline{\Hfl{1}(U,\bL_{n,m})} 
    \quad\cong\quad
     H^2_{\idealm}(X,W_n\OO_X)[F^m].
     $$
\end{Corollary}

We now make this correspondence explicit for (loc,loc)-group schemes of length $p^2$,
whose classification we recalled in Proposition \ref{prop: simple infinitesimal}.

\begin{Corollary} \label{cor: torslocloc}
There are the following isomorphisms of left $\Dieu$-modules
$$
\begin{array}{lcl}
\overline{\Hfl{1}(U,\balpha_p)} &\cong & H^2_{\idealm}(X,\OO_X)[F] \\
\overline{\Hfl{1}(U,\balpha_{p^2})} & \cong & H^2_{\idealm}(X,\OO_X)[F^2] \\
\overline{\Hfl{1}(U,\balpha_{p^2}^D)}  & \cong & H^2_{\idealm}(X,W_2 \OO_X)[F] \\
\overline{\Hfl{1}(U,\bM_2)}  & \cong & H^2_{\idealm}(X,W_2 \OO_X)[F-V].
\end{array}
$$
\end{Corollary}
\begin{proof}
The first three isomorphisms are special cases of Proposition \ref{prop: torslocloc}. 
For the fourth isomorphism, we use that $\bM_2$ can be described as $\bL_{2,2}[F-V]$ and then the isomorphism follows from the fact that the isomorphism in Corollary \ref{cor: torsloclocwithdieu} is an isomorphism of $\Dieu$-modules.
\end{proof}

Next, we want to recover $\Picloclocloc_{X/k}$ from Corollary \ref{cor: torsloclocwithdieu}. For this, we need the 
following finiteness result.

\begin{Lemma} \label{lem: finitelength}
Assume that $d \geq 3$ or that $d=2$ and $X$ is normal. Then, for all $m,n > 0$, the left $\Dieu$-module $H^2_{\idealm}(X,W_n\OO_X)[F^m]$ has finite 
length as $W(k)$-module and satisfies $V^n = F^m = 0$. 
\end{Lemma}

\begin{proof}
Clearly, $V^n = F^m = 0$ holds and the only non-trivial claim is that $H^2_{\idealm}(X,W_n\OO_X)[F^m]$ 
has finite length as $W(k)$-module. 
We argue by induction on $n$. 
There is a natural commutative diagram
$$
\xymatrix{
0 \ar[r] & \mathcal{O}_X \ar^(0.35){V^n}[r] \ar^{F^m}[d]& W_{n+1}\mathcal{O}_X \ar^{R}[r] \ar^{F^m}[d]& W_n\mathcal{O}_X \ar[r] \ar^{F^m}[d] & 0 \\
0 \ar[r] & \mathcal{O}_X \ar^(0.35){V^n}[r] &  W_{n+1}\mathcal{O}_X \ar^{R}[r] &  W_{n}\mathcal{O}_X \ar[r] & 0.
}
$$
Since $X$ is $S_2$, we have $H^1_{\idealm}(X,\mathcal{O}_X)=0$, which yields
$H^1_{\idealm}(X,W_n \mathcal{O}_X) = 0$ by induction.
Thus, the long exact sequences in local cohomology for the diagram above yield an exact sequence
$$
0 \to H^2_{\idealm}(X,\mathcal{O}_X)[F^m] \to H^2_{\idealm}(X,W_{n+1}\OO_X)[F^m] \to H^2_{\idealm}(X,W_{n}\OO_X)[F^m]
$$
and thus, by the induction hypothesis, it suffices to show that $H^2_{\idealm}(X,\mathcal{O}_X)[F^m]$ 
has finite length.

If $\dim(X) \geq 3$, then the claim is true even before taking Frobenius kernels by 
\cite[Exp. VIII. Corollaire 2.3]{SGA2}. 
If $\dim(X) = 2$ and $X$ is normal, then $U$ is regular, hence $F$-rational by assumption, hence $H^2_{\idealm}(X,\mathcal{O}_X)$ has finite length as $R[F]$-module as explained in \cite[Section 4.17.3]{Smith}. 
In particular, for a fixed $m$, $H^2_{\idealm}(X,\mathcal{O}_X)[F^m]$ has finite length as $R$-module, hence it is finite-dimensional as a $k$-vector space, which is what we had to show.
\end{proof}

\begin{Corollary} \label{cor: Dieudonneofpicloc}
With assumptions as in Lemma \ref{lem: finitelength}, we set 
$$
 \Picloclocloc_{X/k}[V^m,F^n] \,:=\, \Ker(V^m) \cap \Ker(F^n).
$$
Then, $\Picloclocloc_{X/k}[V^m,F^n]$ is a finite abelian group scheme of (loc,loc)-type 
with Dieudonn\'e module $H^2_{\idealm}(X,W_n\OO_X)[F^m]$. 
In particular,
$$
\Picloclocloc_{X/k} \,=\, \varinjlim_{m,n}\, \Picloclocloc_{X/k}[V^m,F^n]
$$
is a description of $\Picloclocloc_{X/k}$ as an ind-(loc,loc) group scheme.
\end{Corollary}

\begin{proof}
Being the maximal $(\loc,\loc)$-subsheaf of $\Picloc_{X/k}$, 
$\Picloclocloc_{X/k}$ is an ind-$(\loc,\loc)$ group scheme by Lemma \ref{lem: indgroupscheme}. 
By Corollary \ref{cor: torsloclocwithdieu}, we have isomorphisms
\begin{eqnarray*}
\Dieulocloc(\Picloclocloc_{X/k}) &=& \varinjlim_{n,m}\, \Hom(\bL_{n,m},\Picloclocloc_{X/k})  \\
&\cong &  \varinjlim_{n,m}\, H^2_{\idealm}(X,W_m\OO_X)[F^n].
\end{eqnarray*}
Since $\Dieulocloc(F)$ is left-multiplication by $V$ and $\Dieulocloc(V)$ is left-multiplication 
by $F$, we can use that $\Dieulocloc$ induces an equivalence of categories to conclude
\begin{eqnarray*}
  \Dieulocloc(\Picloclocloc_{X/k}[V^m,F^n]) &=& H^2_{\idealm}(X,W_n\OO_X)[F^m].
\end{eqnarray*}
By Lemma \ref{lem: finitelength} and Proposition \ref{prop: colimitsofdieudonne}, 
this implies that the $\Picloclocloc_{X/k}[V^m,F^n]$ are finite group schemes, 
hence, using Proposition \ref{prop: colimitsofdieudonne} once more, the equality
\begin{eqnarray*}
\Dieulocloc(\Picloclocloc_{X/k}) &=& \Dieulocloc(\varinjlim_{m,n} \Picloclocloc_{X/k}[V^m,F^n])
\end{eqnarray*}
implies $\Picloclocloc_{X/k} = \varinjlim_{m,n} \Picloclocloc_{X/k}[V^m,F^n]$. 
\end{proof}

\begin{Corollary} \label{cor: picloclocloc trivial}
With assumptions as in Lemma \ref{lem: finitelength}, the following three assertions are equivalent:
\begin{enumerate}
    \item $\Picloclocloc_{X/k}$ is trivial,
    \item $H^2_{\idealm}(X,W_n\OO_X)[F^m] = 0$ for all $m$ and $n$, and
    \item Frobenius is injective on $H^2_{\idealm}(X,\OO_X)$.
\end{enumerate}
\end{Corollary}

\begin{proof}
The equivalence $(1) \Leftrightarrow (2)$ follows from Corollary \ref{cor: Dieudonneofpicloc}, as $\Picloclocloc_{X/k}$ is trivial 
if and only if $\Picloclocloc_{X/k}[V^m,F^n] = 0$ for all $m$ and $n$,
and a finite group scheme of $(\loc,\loc)$-type is trivial 
if and only if its Dieudonn\'e module is trivial. 
The equivalence $(2) \Leftrightarrow (3)$ follows by induction on $n$.
\end{proof}

Thus, the size of $\Picloclocloc_{X/k}$ can be viewed as a measure for the
non-F-injectivity of $X$. 
For example, we have the following result.

\begin{Proposition} \label{prop: finiteVlengthsemiabelian}
Assume that $X$ is normal. Consider the following three assertions:
\begin{enumerate}
    \item $\Picloclocloc_{X/k}$ is a finite group scheme,
    \item $V$ is nilpotent on $\Dieulocloc(\Picloc_{X/k})$, and
    \item $(\Picloc_{X/k})_{\reduced}^\circ$ is an ordinary semiabelian variety.
\end{enumerate}
Then, we have the implications $(1) \Rightarrow (2) \Rightarrow (3)$. 
If $\dim(X) \geq 3$, then the three assertions are equivalent.
\end{Proposition}

\begin{proof}
The implication $(1) \Rightarrow (2)$ is clear.

Let us show $(2) \Rightarrow (3)$.
By \cite[Th\'eor\`eme II.7.8, Corollaire V.1.8.1]{Boutot}, 
$(\Picloc_{X/k})_{\reduced}^\circ$ is a smooth and connected group scheme over $k$. 
By the Barsotti--Chevalley theorem (see \cite{Barsotti} and \cite{Chevalley}), 
we can write $(\Picloc_{X/k})_{\reduced}^\circ$ as an extension of an abelian 
variety $\mathbf{A}$ by a smooth affine algebraic group $\mathbf{L}$. 
Since $(\Picloc_{X/k})_{\reduced}^\circ$ is commutative, so is $\mathbf{L}$ and hence, 
we can write $\mathbf{L} = \mathbf{T} \times \mathbf{U}$, where $\mathbf{T}$ 
is a torus and $\mathbf{U}$ is unipotent.

If $\mathbf{U}$ was non-trivial, then it would contain $\mathbb{G}_a$.
Then, we would find
$$
\Dieulocloc(\mathbb{G}_a) \subseteq
\Dieulocloc(\mathbf{U}) \subseteq \Dieulocloc((\Picloc_{X/k})_{\reduced}^\circ) \subseteq \Dieulocloc(\Picloc_{X/k}).
$$
In particular, $V$ would not be nilpotent on $\Dieulocloc(\Picloc_{X/k})$, 
contradicting our assumption.
Thus, $\mathbf{U}$ is trivial.

Since $\mathbf{U}$ is trivial and there are no non-trivial homomorphisms 
from a finite commutative group scheme of $(\loc,\loc)$-type to $\mathbf{T}$, 
we have 
$$
\Dieulocloc(\textbf{A}) \,\cong\, \Dieulocloc((\Picloc_{X/k})_{\reduced}^\circ) 
\,\subseteq\, \Dieulocloc(\Picloc_{X/k}),
$$ 
where $\Dieulocloc(\textbf{A})$ is the Dieudonn\'e module of the 
$(\loc,\loc)$-part of the $p$-divisible group of $\textbf{A}$. 
In particular, if $V$ is nilpotent on $\Dieulocloc(\Picloc_{X/k})$, then 
$\Dieulocloc(\textbf{A})$ is trivial and hence, $\textbf{A}$ is ordinary. 
Therefore, $(\Picloc_{X/k})_{\reduced}^\circ$ is an extension of the 
ordinary abelian variety $\textbf{A}$ by the torus $\textbf{T}$, 
which is what we had to show.

Finally, assume that $\dim(X) \geq 3$. 
Then, we have to show that $(3) \Rightarrow (1)$. 
Since $\dim(X)\geq3$, we have that $\Picloc_{X/k}$ 
is representable by a group scheme locally of finite type over $k$  
by \cite[Th\'eor\`eme II.7.8]{Boutot}.
Thus, $\Picloc_{X/k}^\circ$ is a group scheme of finite type over $k$. 
In particular, $(\Picloc_{X/k})_{\reduced}^\circ$ has finite index in $\Picloc_{X/k}^\circ$. 
Now, if $(\Picloc_{X/k})_{\reduced}^\circ$ is an ordinary semiabelian variety, 
then $\Dieulocloc((\Picloc_{X/k})_{\reduced}^\circ) = 0$ and hence, 
$\Dieulocloc(\Picloc_{X/k}) = \Dieulocloc(\Picloc_{X/k}^\circ)$ has finite length as a 
$\Dieu$-module and so $\Picloclocloc_{X/k}$ is a finite group scheme by 
Corollary \ref{cor: finiteifdieufinite}.
\end{proof}

\begin{Question} \label{qu: length of picloclocloc}
Are all three assertions equivalent if $\dim(X)=2$?
\end{Question}

Let us shortly digress on cones - 
we will not need these results in the sequel, but they may be
helpful in order to put the results of this section in a broader
perspective.
Also, it provides some evidence that 
Question \ref{qu: length of picloclocloc} may have a positive answer.
Let $Y$ be a positive-dimensional projective variety over $k$, let $\mathcal{L}$ be 
a very ample invertible sheaf on $Y$, such that $Y$ is projectively normal with respect 
to the embedding determined by $\mathcal{L}$. 
Let $x\in X$ be the completion of the affine cone 
$\Spec \bigoplus_{n\geq 0} H^0(Y,\mathcal{L}^{\otimes n})$ 
over $(Y,\mathcal{L})$. 
Note that our assumptions imply that $x \in X$ is a normal singularity of dimension at least $2$.
By \cite[Corollaire V.1.8.2, Proposition V.5.1]{Boutot} and 
\cite[Section 3]{EsnaultViehweg}, we have
$$
(\Picloc_{X/k})_{\reduced}^\circ \,\cong\, \mathbf{U} \times (\Pic_{Y/k})_{\reduced}^\circ,
$$
where $\mathbf{U}$ is smooth, commutative, and unipotent (see \cite[p.9 Corollaire]{Oort}) and 
$(\Pic_{Y/k})_{\reduced}^\circ$ is the Picard variety of $Y$.  
Moreover, $\mathbf{U}$ is trivial if and only if $H^1(Y,\mathcal{L}) = 0$. In particular, by Proposition \ref{prop: finiteVlengthsemiabelian}, if $V$ is nilpotent 
on  $\varinjlim_{m,n} H^2_{\idealm}(X,W_n\OO_X)[F^m]$, then $(\Picloc_{X/k})_{\reduced}^\circ \cong (\Pic_{Y/k})_{\reduced}^\circ$ is an ordinary abelian variety and $H^1(Y,\mathcal{L}) = 0$.

\begin{Proposition}\label{prop: cone}
  Let $Y$ be a smooth projective curve of genus $g$ over $k$,
  let $\mathcal{L}$ be a very ample invertible sheaf of degree $e> 2g + 1$, and
  let $x\in X$ be the completed affine cone over $(Y,\mathcal{L})$.
  \begin{enumerate}
      \item  $\Pic_Y$ is an ordinary abelian variety if and only if $x \in X$ is F-injective.
      \item  All three assertions of Proposition \ref{prop: finiteVlengthsemiabelian} are equivalent for $X$.
  \end{enumerate}
\end{Proposition}

\begin{proof}
First, note that $e > 2g+1$ guarantees that the embedding of $Y$ determined by $\mathcal{L}$ is projectively normal (see e.g. \cite[p.55]{Mumford}), so that $x \in X$ is a normal singularity.

First, we show that Claim (2) follows from Claim (1).
As explained before the statement of the Proposition, Assertion (3) of Proposition \ref{prop: finiteVlengthsemiabelian} implies that $\Pic_Y$ is
an ordinary abelian variety (by our assumptions on $Y$,
$\Pic_Y$ contains no torus), so it is enough to show that $\Pic_Y$ being ordinary implies that $X$ is 
F-injective, for then $\Picloclocloc_{X/k}$ 
will be trivial by Corollary \ref{cor: picloclocloc trivial}.

So, let us prove Claim (1). By Proposition \ref{prop: finiteVlengthsemiabelian} and the discussion before the statement of the Proposition, it is clear that if $x \in X$ is $F$-injective, then $\Pic_Y$ is ordinary. For the converse, let us assume that $\Pic_Y$ is an ordinary abelian  variety.
By a standard \v{C}ech cohomology computation, we have an isomorphism
$$
 H^2_\idealm(X,\calO_X) \,\cong\,
 \bigoplus_{m\in\ZZ}\, H^1(Y,\mathcal{L}^{\otimes m}).
$$
Since $e>2g+1$, Serre duality gives $H^1(Y,\mathcal{L}^{\otimes m})=0$ for
all $m>0$. 
Let $D$ be a non-zero divisor in $|\mathcal{L}|$.
Using the notation of \cite{HaraWatanabe}, we have $D'=0$ in this
case and since $p\deg D > 2g + 1$, we have that Frobenius is injective
on $H^1(Y,\mathcal{L}^{\otimes m})$ for all $m<0$ by 
\cite[Corollary 2.6]{HaraWatanabe}.
In particular,
$$
H^2_{\idealm}(X,\OO_X)[F] \,\cong\, H^1(Y,\OO_Y)[F].
$$
Since $\Pic_Y$ is ordinary, the right-hand side vanishes, hence $H^2_{\idealm}(X,\OO_X)[F] = 0$, that is, $X$ is F-injective. 
\end{proof}

\begin{Remarks}
Concerning finiteness of $\Dieulocloc(\Picloc_{X/k})$:
\begin{enumerate}
\item If $x\in X$ is an RDP, then $\Dieulocloc(\Picloc_{X/k})$ is a $\Dieu$-module
of finite length by Theorem \ref{thm: picloclocloc finite} below.
In this case, all three assertions of Proposition \ref{prop: finiteVlengthsemiabelian}
hold true.
\item If $x \in X$ is the completed affine cone over a polarised elliptic curve $(Y,\mathcal{L})$ with $\mathcal{L}$ very ample, then the degree assumption in Proposition \ref{prop: cone} is satisfied, hence all three assertions of Proposition \ref{prop: finiteVlengthsemiabelian} are equivalent.
\end{enumerate}
\end{Remarks}

\subsection{General torsors: the local Nori fundamental group scheme}
Let $G$ be a finite group scheme over $k$.

A scheme-theoretic analogue $\piNloc(U,X,x)$ of $\pietloc(X)$ was described in \cite{EsnaultViehweg} using Nori's fundamental group scheme $\pi^{\rm N}(X,x)$ from \cite{Nori}. 
First, one constructs a pro-finite group scheme $\piNloc(U,x)$ with the property that there is a canonical bijection of pointed sets
$$
\Hfl{1}(U,G)  \quad\leftrightarrow\quad \Hom(\piNloc(U,x),G)/\sim,
$$
where the right-hand side denotes morphisms of profinite group schemes up to inner automorphisms of $G$. 
Under this bijection, surjective homomorphisms on the right hand side correspond to
$G$-torsors that do not admit a reduction of structure group to strict 
subgroup schemes of $G$. 
Such torsors are called \emph{Nori-reduced}.
 
Using the canonical morphism $\piNloc(U,x) \to \pi^{\rm N}(X,x)$, 
which is faithfully flat, one defines
$$
\piNloc(X) \,:=\, \piNloc(U,X,x) \,:=\, \Ker\left(\piNloc(U,x) \to \pi^{\rm N}(X,x)\right).
$$
However, in contrast to $\piet(X,x)$, the fundamental group scheme $\pi^{\rm N}(X,x)$ is 
highly non-trivial since $R$ is not perfect and therefore, 
$\piNloc(X)$ does not satisfy the analogue of Proposition \ref{prop: torsfinet}. 
Nevertheless, we have the following weaker observation, which follows immediately from the definitions.

\begin{Proposition}\label{prop: torspinloc}
 If $G$ is finite, then there is a canonical bijection of pointed sets
  $$
   \Hfl{1}(U,G)
      \quad\cong\quad
      \Hom(\piNloc(U,x),G)/\sim,
  $$
 where $\Hom$ denotes morphisms of profinite group schemes and where $\sim$ means that we identify homomorphisms that differ by an inner automorphism of $G$. 
 Moreover,
 \begin{enumerate}
     \item this bijection identifies the subsets $\Hfl{1}(X,G)$ and $\Hom(\piN(X,x),G)/\sim$.
     \item If $G$ is abelian, then there is a canonical isomorphism
of abelian groups
$$
 \overline{\Hfl{1}(U,G)}
 \quad\cong\quad
  {\rm Im}\left(\Hom(\piNloc(U,x),G) \,\to\, \Hom(\piNloc(U,X,x),G)\right)/\sim.
 $$
 \end{enumerate}
\end{Proposition}

\begin{proof}
Using the local Nori fundamental group scheme of \cite{RTZ}
we reduce to the case where $G$ is finite and \'etale.
In this latter case, the result is classical
(see also Proposition \ref{prop: torsfinet}).
\end{proof}

\begin{Remark}
In the setting of Proposition \ref{prop: torspinloc}, \cite[Theorem 4.4]{ToniniZhang}
shows that $U$ admits a Nori fundamental gerbe $\Pi_{U/k}$ in the sense of 
Borne--Vistoli \cite[Definition 5.1]{BorneVistoli}. 
It seems reasonable to expect that $\Pi_{U/k}$ can be identified with the classifying stack of $\piNloc(U,x)$. 
Such an identification would upgrade the bijection in Proposition \ref{prop: torspinloc} 
to an equivalence of categories.
\end{Remark}

\begin{Corollary}\label{cor: noritrivial}
The local Nori fundamental group scheme $\piNloc(X)$ is trivial if and only if 
$\Hfl{1}(U,G) = \Hfl{1}(X,G)$ for every finite group scheme 
$G$ over $k$.
\end{Corollary}

In other words, the local Nori fundamental group scheme $\piNloc(X)$ is trivial
if and only if for every finite group scheme $G$ over $k$ every
$G$-torsor over $U$ uniquely extends to a $G$-torsor over $X$.
For example, if $x\in X$ is regular, then purity for torsors implies
that $\piNloc(X)$ is trivial. 
As a generalisation of the Flenner--Mumford smoothness criterion, 
we show in Section \ref{sec: Mumford} that the converse holds for normal surface 
germs over perfect fields.

\subsection{Local torsors over surfaces}
Now, assume that $x \in X$ has dimension $d = 2$ so that our assumption that $X$ 
is $S_2$ implies that $X$ is Cohen--Macaulay. 
Then, the notion of F-injectivity has the following 
geometric interpretation in terms of local torsors. 

\begin{Lemma}\label{lem: Finjective}
 For a Cohen--Macaulay surface singularity $x\in X$, the following are 
 equivalent
 \begin{enumerate}
 \item it is F-injective,
 \item for every finite and abelian group scheme $G$ of (loc, loc)-type, the 
 group $\overline{\Hfl{1}(U,G)} $ is trivial.
 \item $\overline{\Hfl{1}(U,\balpha_p)} $ is trivial, that is, every local 
 $\balpha_p$-torsor over $X$ extends to a global $\balpha_p$-torsor over $X$.
 \end{enumerate}
\end{Lemma}

\begin{proof}
If $X = \Spec R$ is F-injective, then one easily sees, using induction on $n$, 
that $F$ is also injective on $H^i_{\idealm}(W_n R)$ for every $n \in \NN$. 
Thus, by Proposition \ref{prop: torslocloc}, every $G$-torsor over $U$ as in (2) 
extends to a $G$-torsor over $X$.
This establishes $(1)\Rightarrow(2)$.

The implication $(2)\Rightarrow(3)$ is trivial.

Finally, since $R$ is Cohen--Macaulay, the only non-trivial local cohomology 
group is $H^2_{\idealm}(R)$.
Thus, $(3)\Rightarrow(1)$ follows from Proposition \ref{prop: torslocloc} applied 
to $G = \balpha_p=\bL_{1,1}$.
\end{proof}

\section{Local torsors over the rational double points}
\label{sec: torsors over rdp}

In this section, we specialise to the case where $x \in X$ is a rational double point
(RDP for short).
We recall the local fundamental group $\pietloc(X)$ was determined by Artin and 
that the class group ${\rm Cl}(X)$ was calculated by Lipman. 
We now complement these classical results on RDPs by computing the scheme
$\Picloclocloc_{X/k}$, which turns out to be finite in every case. 
As explained in Section \ref{sec: local torsors}, these three objects together 
control the finite and \'etale or abelian local torsors over $X$.

Our results are displayed in Table \ref{table: torsors with Picloclocloc}. 
In it, we denote by $\C_n$ the cyclic group of order $n$,
by $\BD_n$ the binary dihedral group of order $n$, and by $\bD_n$ the dihedral group of order $n$.
For the loc-loc group schemes $\bL_{m,n}$ and $\bM_n$, see Section \ref{subsec: loclocCartierDieudonne}.
The loc-loc group schemes $\bG_n^r$ will be introduced and described in 
Section \ref{subsec loclocloc torsors}.
Moreover, we denote 
by $\Dic_{12}$ the dicyclic group of order $12$ (which is also isomorphic to 
a metacyclic group of order $12$), 
by $\BT_{24}$ the binary tetrahedral group of order $24$,
by $\BO_{48}$ the binary octahedral group of order $48$, and
by $\BI_{120}$ the binary icosahedral group of order $120$.
For a singularity in characteristic $p>0$, we denote by $n'$ 
the prime-to-$p$ part of an integer $n$.
  
   \begin{table}[!h]
  $$\begin{array}{|c|c|c|c|c|c|}
  \hline
  &\mbox{char.}     & X  & \pietloc(X) & {\rm Cl}(X) & \Picloc^{\loc,\loc}_{X/k} \\ \hline \hline
  
\parbox[t]{3mm}{\multirow{6}{*}{\rotatebox[origin=c]{90}{\text{F-regular}}}}  & \mbox{all} &     A_n  & \C_{(n+1)'} & \C_{n+1} & 0 \\
 
 & > 2 & D_n & \BD_{(4(n-2))'} & \begin{array}{c}
    \C_2^2 \mbox{ if } 2 \mid n \\
      \C_4 \mbox{ if } 2\nmid n
      \end{array}
      & 0  \\
      
 &  > 3 & E_6 & \BT_{24} & \C_3 & 0  \\
 & > 3  & E_7 & \BO_{48} & \C_2 & 0 \\
 & > 5 & E_8 & \BI_{120} & 0 & 0 \\ \hline \hline
 
\parbox[t]{3mm}{\multirow{20}{*}{\rotatebox[origin=c]{90}{\text{Not F-regular}}}}  & 5 & E_8^1  & \C_5 & 0 & 0 \\
 & 5 &  E_8^0 & 0 & 0 & \balpha_5 \\ \cline{2-6}
  
 & 3 & E_6^1 & \C_3 & \C_3 & 0  \\
 &  3 & E_6^0 & 0 & \C_3 & \balpha_3 \\
 & 3 & E_7^1  & \C_6 & \C_2 & 0  \\
 & 3 &E_7^0& \C_2 & \C_2 & \balpha_3 \\
 & 3 & E_8^2& \BT_{24} & 0 & 0 \\
 &3 & E_8^1 & 0 & 0 & \bM_2 \\
 & 3 &  E_8^0  & 0 & 0 & \balpha_{9}  \\ \cline{2-6}
 & 2 &  D_n^r  & 
 \begin{array}{cc}
 \bD_{2(4r-n)'} & \text{if $4r>n$}, \\
\C_2 & \text{if $4r=n$}, \\
0 & \text{if $4r<n$},
 \end{array}
 & \begin{array}{cc}
      \C_2^2 &\mbox{ if } 2 \mid n \\
      \C_4  & \mbox{ if } 2\nmid n
      \end{array} & \bG_n^r  \\
 & 2 & E_6^1 & \C_6      & \C_3 & 0 \\
 & 2 & E_6^0 & \C_3      & \C_3 &  \balpha_2  \\
 & 2 & E_7^3 & \C_4      & \C_2 & 0  \\
 & 2 & E_7^2 & 0         & \C_2 & \bM_2 \\
 & 2 & E_7^1 & 0         & \C_2 & \bL_{2,3}[V+F^2] \\
 & 2 & E_7^0 & 0         & \C_2 & \balpha_8 \\
 & 2 & E_8^4 & \Dic_{12} & 0    & 0  \\
 & 2 & E_8^3 & 0         & 0    & \bM_3  \\
 & 2 & E_8^2 & \C_2      & 0    & \bL_{2,3}[V+F^2]  \\
 & 2 & E_8^1 & 0         & 0    & \bL_{2,3}[p] \\
 & 2 & E_8^0 & 0         & 0    & \balpha_2 \times \balpha_8 \\ \hline
  \end{array}$$
  \caption{Torsors over the rational double points}
    \label{table: torsors with Picloclocloc}
  \end{table}

\subsection{Local \'etale torsors} \label{subsec: etale torsors}

The local fundamental groups of the rational double points
have been computed by Artin \cite{ArtinRDP}.

The F-regular RDPs are precisely those that are
lrq singularities by
\cite[Theorem 11.2]{LRQ}.
Thus, their local fundamental groups can be computed
using \cite[Proposition 7.1 and Theorem 11.1]{LRQ}
as $G^{\et}$, where $G$ is the group scheme attached to $X$.

\subsection{Local $\bmu_n$-torsors} \label{subsec: mun torsors}

By Proposition \ref{prop: torsfinabcl}, local $\bmu_n$-torsors can be computed
from the class group. 
The class groups of rational double points have been calculated by Lipman \cite[Section 24]{Lipman} and they are summarised in Table \ref{table: torsors with Picloclocloc}.

As with the local fundamental groups, also the class groups of the F-regular 
RDPs can be computed
using \cite[Proposition 7.1, Theorem 11.1, Theorem 11.2]{LRQ} 
as $\Hom(G,\GG_m) = (G^{\mathrm{ab}})^D$, where $G$ is the group scheme attached to $X$.

\subsection{Local Witt vector cohomology} \label{subsec: loc-loc torsors}

By Proposition \ref{prop: torslocloc}, local (loc, loc)-torsors over a singularity are 
in bijection with elements in the kernel of certain Frobenius and Verschiebung 
maps on local Witt vector cohomology, see also Corollary \ref{cor: torslocloc}.
By Lemma \ref{lem: Finjective}, there are no such torsors
over F-injective RDPs and thus, not over F-regular RDPs. 
Thus, the results of this section are mainly relevant for RDPs that are not F-injective.  
In Proposition \ref{prop: torsors loclocloc new}, we compute
$\varinjlim_{m,n} H^2_{\idealm}(X,W_n\OO_X)[F^m]$ together
with the actions of $F$ and $V$ on it for these RDPs.

While the calculation of 
$H^2_{\idealm}(X, \OO_X)[F^\infty] \,:=\, 
 \varinjlim_{m} H^2_{\idealm}(X, \OO_X)[F^m]$
is more or less standard, the calculation of 
$$
  H^2_{\idealm}(X, W_n\OO_X)[F^{\infty}]
  \,:=\, 
   \varinjlim_{m}\, H^2_{\idealm}(X, W_n\OO_X)[F^m]
$$
for $n \geq 2$ is more involved. 
We will use the following technical lemma, which allows us to check whether elements of $H^2_{\idealm}(X, W_{n-1}\OO_X)[F^{\infty}]$ lift to $H^2_{\idealm}(X, W_n\OO_X)[F^{\infty}]$, thus allowing us to bound the difference between $H^2_{\idealm}(X, W_n\OO_X)[F^{\infty}]$ and $V^{n-1}(H^2_{\idealm}(X, \OO_X)[F^\infty])$.

\begin{Lemma} \label{lem: extension} 
Let $x\in X$ be a normal surface singularity.
Let 
$$
f \,\in\, H^2_{\idealm}(X,W_{n-1} \OO_X)[F^{\infty}]
$$
and choose $\tilde{f} \in H^2_{\idealm}(X,W_{n} \OO_X)$ with $R(\tilde{f}) = f$.
Next, suppose that there exists an element 
$$
g \,\in\, H^2_{\idealm}(X, \OO_X) \,\setminus\, 
\left( F(H^2_{\idealm}(X, \OO_X))+ H^2_{\idealm}(X, \OO_X)[F^{\infty}]\right),
$$
such that 
$F(\tilde{f}) - V^{n-1}(g) \in H^2_{\idealm}(X,W_{n} \OO_X)[F^{\infty}]$.
Then, $f$ does not admit a lift to $H^2_{\idealm}(X,W_{n} \OO_X)[F^{\infty}]$.

\end{Lemma} 

\begin{proof}
By applying the snake lemma to the short exact sequence 
$$
0 \,\to\, H^2_{\idealm}(X,\OO_X) 
\,\overset{V^{n-1}}{\to}\, 
H^2_{\idealm}(X,W_n \OO_X) 
\,\overset{R}{\to}\,  
H^2_{\idealm}(X,W_{n-1} \OO_X) \,\to\, 0
$$
and the map $F^{m'}$, we obtain an exact sequence 
$$
H^2_{\idealm}(X,W_n \OO_X)[F^{m'}] 
\,\overset{R}{\to}\, 
H^2_{\idealm}(X,W_{n-1} \OO_X)[F^{m'}] \,\overset{\delta_{m'}}{\to}\,
H^2_{\idealm}(X,\OO_X)/F^{m'}.
$$
By our assumptions, $f$ lies in the middle term for $m'\gg0$
and its image under the connecting map is 
$\delta_{m'}(f) = F^{m'-1}(g)$ (in $H^2_{\idealm}(X,\OO_X)/F^{m'}$). 
It suffices to show that this is nonzero for all $m' \gg 0$.
If it is zero, then $F^{m'-1}(g) = F^{m'}(h)$ (in $H^2_{\idealm}(X,\OO_X)$), 
hence $g - F(h) \in H^2_{\idealm}(X, \OO_X)[F^{m'-1}] \subseteq H^2_{\idealm}(X, \OO_X)[F^{\infty}]$, 
which contradicts our assumptions.
\end{proof}

\begin{Proposition} \label{prop: torsors loclocloc new}
Let $x\in X$ be a non-F-injective RDP.
Let 
\[
 f_{j} \,:=\, \frac{z}{x y^{j}} \,\in\, H^2_{\idealm}(X,\OO_X)
\]
with respect to the equations in Table \ref{table: RDPEquations} and the identification given in the proof.
Let 
\[
 f_{j}^{(l)} \,:=\, (f_{j}, 0, \dots, 0) \,\in\, 
 H^2_{\idealm}(X,W_l\OO_X)
\]
be the Teichm\"uller lift of $f_{j}$, and observe that $f_j^{(l)} = 0$ for $j \leq 0$. 

If $x$ is of type $D_n^r$ and $p = 2$, then we set
let $C(l,j) := 2^{l-1}(2j-1)$ and 
$L_{n}^r := \floor{{\rm min}\{n/2 - 1, n - 2r - 2\}}$.

Then, a presentation of $\varinjlim_{m,l} H^2_{\idealm}(X,W_l\OO_X)[F^m]$ as a left
$\Dieu$-module is as given in Table \ref{table: wittvectorcohomology}.

  \begin{table}[!h]
  $$\begin{array}{|c|c|l|l|c|}
  \hline
  \mbox{char.}     & X  & \text{generators} & \text{relations} & \text{length} \\ \hline \hline
 
  5 &  E_8^0 & f_{1} & F \equiv V \equiv 0 & 1 \\ \hline
  
   3 & E_6^0 & f_{1} & F \equiv V \equiv 0 & 1 \\
  3 &E_7^0 & f_{1} & F \equiv V \equiv 0 & 1 \\
 3 & E_8^1 & f_{1}^{(2)} & (F + V) \equiv V^2 \equiv 0 & 2 \\
  3 &  E_8^0  & f_{1}^{(2)} &  F \equiv V^2 \equiv  0  & 2 \\ \hline
  2 &  D_n^r  & 
 \begin{array}{l}
 \{f_{j}^{(l)}\} \text{ such that } \\
 C(l,j) \leq L_{n}^r
 \end{array}
 & \begin{array}{l}
      V(f_j^{(l)}) = f_j^{(l-1)}, \\
      F(f_j^{(l)}) = f_{C(l,j) - ((n/2) - r - 1)}^{(1)}
      \end{array} &  L_{n}^r  \\
 2 & E_6^0 & f_{1} & F \equiv  V \equiv  0 & 1  \\
 2 & E_7^2  &f_{1}^{(2)} & F + V \equiv  V^2 \equiv  0 & 2 \\
 2 & E_7^1  &f_{1}^{(3)} & F + V^2 \equiv  V^3 \equiv  0 & 3  \\
 2 & E_7^0  &f_{1}^{(3)} & F  \equiv  V^3 \equiv  0 & 3  \\
 2 & E_8^3  &f_{1}^{(3)} & F+V \equiv  V^3 \equiv  0 & 3  \\
 2 & E_8^2  &f_{1}^{(3)} & F+V^2 \equiv  V^3 \equiv  0 & 3  \\
 2 & E_8^1  &f_{1}^{(3)},f_2 & 
  \begin{array}{l}
      F(f_1^{(3)}) = V^3(f_1^{(3)})  = 0, \\
    F(f_2) - V^2(f_1^{(3)}) = V(f_2) = 0
      \end{array}
      & 4 \\
 2 & E_8^0  &f_{1}^{(3)},f_2 & 
   \begin{array}{l}
      F(f_1^{(3)}) = V^3(f_1^{(3)})  = 0, \\
      F(f_2)  = V(f_2) = 0
      \end{array}
& 4 \\ \hline
  \end{array}$$
  \caption{The $\Dieu$-module $\protect\varinjlim_{m,n} H^2_{\idealm}(X,W_n\OO_X)[F^m]$ for non-F-injective RDPs}
    \label{table: wittvectorcohomology}
  \end{table}
\end{Proposition}

\begin{proof}
Write $X = \Spec ~R$ with $R = k[[x,y,z]]/(h)$, where $h$ is one of the equations in Table \ref{table: RDPEquations}.
Using \v{C}ech cohomology with respect to the cover $\{D(x),D(y)\}$, we obtain an identification
$$
H^2_{\idealm}(X,W_l \OO_X) \,=\, W_lR[(xy)^{-1}] \,/\,
\left( W_lR[x^{-1}] + W_lR[y^{-1}] \right).
$$
Since $h$ has degree $2$ in $z$ and contains a 
monomial equal to $z^2$, the space
$H^2_{\idealm}(X,\OO_X)$ has basis 
$\{ e_{i,j}, f_{i,j} \mid i,j \geq 1 \}$ 
with $e_{i,j} = x^{-i} y^{-j}$, $f_{i,j} = x^{-i} y^{-j} z$. 
In particular, the $f_j$ in the statement is equal to $f_{1,j}$.

First, assume $X$ is of type $E_n^r$. 
Straightforward but rather tedious calculations show 
the following (the brackets are explained below):

\begin{itemize}
    \item $E_8^0$ ($p=5$): $F(f_1^{(2)}) = -f_5$.
    \item $E_6^0$ ($p=3$): $F(f_1^{(2)}) = -f_{3,1}$.
    \item $E_7^0$ ($p=3$): $F(f_1^{(2)}) = -f_3$.
    \item $E_8^r$ ($p=3$): $F(f_1^{(3)}) = 
\begin{cases}
f_{3,2} - f_{4} -  f_{2,1} + [f_1^{(2)}] & (r = 1), \\ 
f_{3,2} & (r = 0).
\end{cases}$ 
    \item $E_6^0$ ($p=2$): $F(f_1^{(2)}) = f_2$.
    \item $E_7^2$ ($p=2$): $F(f_1^{(3)}) = f_2 + [f_1^{(2)}] $. 
    \item $E_7^r$ ($p=2$): $F(f_1^{(4)}) = \begin{cases}
e_{2,1} + f_3 + [f_1^{(2)}] & (r = 1), \\
e_{2,1} & (r = 0).
\end{cases}$ 
    \item $E_8^r$ ($p=2$): $F(f_1^{(4)}) = \begin{cases}
e_{1,1} + [f_1] + f_2 + [f_1^{(3)}] & (r = 3),\\
e_{1,1} + [f_1^{(2)}] & (r = 2),  \\
e_{1,1} & (r = 1,0). 
\end{cases}$
    \item $E_8^r$ ($p=2$): $F(f_2^{(2)}) = \begin{cases}
e_{1,3} + [f_1^{(2)}] & (r = 1), \\ 
e_{1,3} & (r = 0)
\end{cases}$ 
\end{itemize}

The above list shows two things: 

\begin{enumerate}
    \item Restricting to Witt vectors of shorter length, we see that the elements in the third column of Table \ref{table: wittvectorcohomology} are indeed contained in the space $\varinjlim_{m,l} H^2_{\idealm}(X,W_l\OO_X)[F^m]$ and that they satisfy the relations of the fourth column of Table \ref{table: wittvectorcohomology}. 
    To see that there are no other relations, it suffices to check that if $\sum_{j,l} [a_{j,l}] f_{j}^{(l)} = 0$ for $a_{j,l} \in k$ then $a_{j,l} = 0$, where $[-]$ is the Teichm\"uller lift and the index $(j,l)$ runs through the asserted generators and their $V$-descendants, but this is clear.
    
    \item \label{item: not extend} In each case, let $\tilde{f} := f_{i,j}^{(l+1)}$ be the element for which we calculate $F(\tilde{f})$ and let $g$ be the sum of the terms without bracket in $F(\tilde{f})$. Straightforward calculations in $H^2_{\idealm}(X, \OO_X)[F^\infty]$ show that $g \not \in F(H^2_{\idealm}(X, \OO_X))$. Moreover, by (1), the terms in brackets are killed by some power of $F$, hence $g$ satisfies the assumptions of Lemma \ref{lem: extension}.
    Thus, $f_{i,j}^{(l)}$ does not extend to $H^2_{\idealm}(X,W_{l+1}\OO_X)[F^{\infty}]$.  
\end{enumerate}
This finishes the proof in the $E_n^r$-cases. 
Indeed, we know that the asserted generators give a subspace of 
$\varinjlim_{m,l} H^2_{\idealm}(X,W_l\OO_X)[F^m]$.
If the latter was strictly larger,  then at least one of the asserted generators would
have to lift to a higher level, which contradicts (\ref{item: not extend}).

Now we consider the $D_n^r$ cases ($p = 2$). 
We argue by induction on $l$.

For $l=1$, that is, the $W_1$ part, we have 
\begin{equation}\label{eq: Dnr formula}
\begin{aligned} 
F\left(\sum_{i,j} a_{i,j} e_{i,j} + \sum_{i,j} b_{i,j} f_{i,j}\right) 
&= \sum_{i,j} a_{i,j}^2 e_{2i,2j} \,+\,   
\sum_{i,j} b_{i,j}^2 e_{2i-2,2j-1} \\
 & +  \sum_{i,j} b_{i,j}^2 f_{2i-1, 2j + r - n/2 } 
\\ & +   
 \begin{cases}
\sum_{i,j} b_{i,j}^2 e_{2i-1, 2j - n/2}  & (\text{ if } 2 \mid n) \\
\sum_{i,j} b_{i,j}^2 f_{2i, 2j - (n-1)/2} & (\text{ if } 2 \nmid n)
\end{cases}
\end{aligned}
\end{equation}
Checking parities of the indices, one easily sees 
that if the right hand side of Equation \eqref{eq: Dnr formula} is zero, then $a_{i,j} = 0$ for all $i$ and $j$ and 
$b_{i,j} = 0$ for all $i \geq 2$ and $j$.
Thus, every element of $H^2_{\idealm}(X, \OO_X)[F]$ is of the form $\sum_{j} b_{1,j} f_{1,j}$, and, by induction, so is 
every element of $H^2_{\idealm}(X, \OO_X)[F^{\infty}]$. The latter implies that $b_{1,j} = 0$ for $2j > \floor{n/2}$. Then, Equation \eqref{eq: Dnr formula} simplifies to
$$
F\left(\sum_{2j \leq \floor{n/2}} b_{1,j} f_{1,j}\right) =  \sum_{2j \leq \floor{n/2}} b_{1,j}^2 f_{1, 2j + r - n/2 }.
$$
The integer-valued function $\varphi: j \mapsto 2j + r - n/2$ is injective, so we have
$F^k\left(\sum_{2j \leq \floor{n/2}} b_{1,j} f_{1,j}\right) = 0$ for some $k \geq 0$ if and only if for every $j$ either $b_{1,j} = 0$ or $\varphi^k(j) \leq 0$ for $k \gg 0$.
The latter happens if and only if $j <  ((n/2) - r)$, or, equivalently, $2j \leq \floor{n - 2r - 1}$. Together with $2j \leq \floor{n/2}$, we obtain $C(1,j) \leq L_n^r$, hence $H^2_{\idealm}(X, \OO_X)[F^{\infty}]$ is as claimed. 

Next, consider the $W_l$ part for $l>1$. 
For a pair $(l,j)$ satisfying 
$C(l-1,j) \leq L_n^r$ (that is, we have $f_j^{(l-1)} \in H^2_{\idealm}(X, W_{l-1} \OO_X)[F^{\infty}]$), 
we have the equation
\[
 F\left(f_j^{(l)}\right) \,=\, 
 [f_{j_1}^{(2)}] \,+\, f_{j_3} \,+\, 
 \begin{cases}
 e_{j_2} & \text{ if } 2 \mid n,\\
 f_{2,j_2} & \text{ if } 2 \nmid n ,
 \end{cases}
\]
where $j_1 = C(l-1, j) - (n/2 - r - 1)$, 
$j_2 = C(l, j) - \floor{n/2} + 1$, 
and $j_3 = C(l, j) - (n/2 - r - 1)$
(cf.\ \cite[Proof of Proposition 4.5]{Matsumotoheight}).
Again, if $j_2 \leq 0$ and $2j_3-1 \leq L_n^r$ 
(or, equivalently, if $C(l,j) \leq L_n^r$), 
then $j_1 \leq 0$ and $f_j^{(l)}$ belongs 
to $H^2_{\idealm}(X, W_{l} \OO_X)[F^{\infty}]$.
Morever, using the injectivity of the function $C \colon \NN_{\geq 1} \times \NN_{\geq 1} \to \NN_{\geq 1}$,
we observe that no non-zero linear 
combination of $f_{j'}^{(l')}$ with $l' \leq l$ 
and $C(l',j') > L_n^r$ belongs to 
$H^2_{\idealm}(X, \OO_X)[F^{\infty}]$.
\end{proof}

\begin{Remark}
For useful formulas concerning calculations with
Witt vectors of short length, we refer the reader
to \cite[Section 2]{Matsumotoheight}.
Also, since the computations with Witt vectors 
can be described in terms of the so-called 
\emph{ghost polynomials}, we can reduce the 
computations to certain equalities in rings 
of characteristic zero.
For example, to show the formula for $F(f_1^{(4)})$ 
in $E_8^1$ ($p = 2$), it suffices to check 
$(z/xy)^{16} \equiv (1/xy)$ modulo the ideal generated 
by $(z/x^iy^j)^8$, $2(z/x^iy^j)^4$, $4(z/x^iy^j)^2$,
and $8(z/x^iy^j)$ with $i \leq 0$ or $j \leq 0$,  
in the ring $\tilde{B}[1/xy]$, where $\tilde{B} = \ZZ[x,y,z]/(z^2 + x^3 + y^5 + z x y^3) $.
\end{Remark}

\subsection{Local (loc,loc)-torsors}
\label{subsec loclocloc torsors}
By Proposition \ref{prop: torsfinab}, local torsors under finite abelian group 
schemes of $(\loc,\loc)$-type can be described via $\Picloc_{X/k}^{\loc,\loc}$.
As explained in Section \ref{subsec: abelian loc loc},
Dieudonn\'e theory and Corollary \ref{cor: Dieudonneofpicloc} imply that 
$\Picloc_{X/k}^{\loc,\loc}$ is determined by 
$\varinjlim_{m,n} H^2_{\idealm}(X,W_n\OO_X)[F^m]$, considered
as a left $\Dieu$-module.

In fact, we can be explicit about $\Picloc_{X/k}^{\loc,\loc}$
if $x\in X$ is a non-F-injective RDP.
For all but the $D_n^r$-singularities in $p=2$, 
Proposition \ref{prop: torsors loclocloc new} and the examples in 
Section \ref{subsec: loclocCartierDieudonne} immediately yield the answer and we
refer to Table \ref{table: torsors with Picloclocloc}.
Thus, let $x\in X$ be a $D_n^r$-singularity in $p=2$ and 
set
$$
\bG_n^r \,:=\,\Picloc_{X/k}^{\loc,\loc} \mbox{\quad and \quad}
L_n^r \,:=\, \floor{\min\{ n/2 - 1,n-2r-2\}} .
$$
By Proposition \ref{prop: torsors loclocloc new}, the 
$(\loc,\loc)$-group scheme $\bG_n^r$ is finite
of length $p^{L_n^r}$ with $p=2$.
We now establish some more information about it.

\begin{Proposition}
\label{prop: picloclocloc of DNr}
 Let $\bG_n^r$  be as before.
 \begin{enumerate}
 \item \label{item: GNr killed by p} $\bG_n^r$ is killed by $p$, that is, $\bG_n^r=\bG_n^r[p]$.
 \item \label{item: GNr BT1} If $n\leq 4r+3$, then $\Image(F)=\Ker(V)$, $\Image(V)=\Ker(F)$, 
 and $\bG_N^r$ is self-dual.
 
   If $n>4r+3$, then we have $\Image(F) \subsetneq \Ker(V)$, $\Image(V) \subsetneq \Ker(F)$, and $\bG_n^r$ is not self-dual in general.
 \item \label{item: GNr decomposition} There exist isomorphisms
 \[
   \bG_n^r[F]\,\cong\, \bigoplus_i \balpha_{p^{a_i}}^D \mbox{ \quad and \quad } 
   \bG_n^r/\bG_n^r[F] \,\cong\, \bigoplus_i \balpha_{p^{b_i}}
 \]
 where $1 \leq i \leq \floor{(n-2r)/4}$ and where the
 the $a_i$ and $b_i$ are determined by $n$ and $r$ (we give precise formulas in the proof).
 In particular, Frobenius gives rise to a short exact sequence
 \[
   0\,\to\,  \bigoplus_i  \balpha_{p^{a_i}}^D \,\to\, \bG_n^r\,\to\, \bigoplus_i \balpha_{p^{b_i}} \,\to\,0.
 \]

 \end{enumerate}
\end{Proposition}

\begin{proof}
By Proposition \ref{prop: torsors loclocloc new},
$\Dieulocloc(\bG_n^r) = \varinjlim_{m,l} H^2_{\idealm}(X,W_{l}\OO_X)[F^m]$
is generated as a $\Dieu$-module by 
\[
 f_{j}^{(l)} \mbox{\quad with \quad }2^{l-1}(2j-1) \leq L_n^r.
\]
and the generators satisfy $p f_j^{(l)} = VF f_j^{(l)}= V f_{j'}^{(1)} = 0$ for some $j'$, which implies $\bG_n^r[p]=\bG_n^r$.
This shows (\ref{item: GNr killed by p}) and (\ref{item: GNr decomposition}) 
up to determining the parameters.

The set of the indices $(l,j)$ appearing in the basis corresponds bijectively to the set 
$S = \{ x \in \ZZ_{>0} \mid x < \min \{ n/2, n-2r-1 \} \} $
under the map $C$. 
Moreover, the action of $F$ and $V$ on the elements $f_j^{(l)}$ are described by the functions $F,V \colon S \to S \cup \{0\}$, 
\[
 F(x) = \begin{cases}
    2x - (n-2r-1) & (2x > n-2r-1) \\
    0             & (2x < n-2r-1),
    \end{cases}
    \qquad
 V(x) = \begin{cases}
    x/2 & (2 \mid x), \\
    0 & (2 \nmid x),
 \end{cases}
\]
where $0 \in S$ corresponds to $0 \in \Dieulocloc(\bG_n^r)$.
Hence the inequalities in (\ref{item: GNr BT1}) are reduced to the corresponding assertions on the maps $F,V$ on $S$.

First, suppose that $n \leq 4r+3$. 
Then, $S = \{ x \in \ZZ_{>0} \mid x < n-2r-1 \}$, (note that $n - 2r - 1$ is always odd,) and in this case we have, for $x \in S$, 
\begin{itemize}
    \item $F(x) = 0$ if and only if $x < (n-2r-1)/2$ if and only if $x = V(y)$ for some $y \in S$,
    \item $V(x) = 0$ if and only if $x$ is odd if and only if $x = F(y)$ for some $y \in S$,
\end{itemize}
thus the desired equalities hold.

Second, suppose that $n > 4r+3$. 
Then, the corresponding set is $S' = \{ x \in \ZZ_{>0} \mid x < n/2 - 1\} \subsetneq S$, 
and by looking at the maps $F,V \colon S' \to S' \cup \{0\}$ (defined by the same formulas),
we conclude the desired inequalities.
This shows (\ref{item: GNr BT1}).

In both cases, we can determine the parameters $a_i$ and $b_i$ in (\ref{item: GNr decomposition}) 
by counting the sizes of the $F$-orbits and the $V$-orbits.
Their precise values are as follows
\begin{eqnarray*}
a_i &=& \ceil{\log_2 \frac{n-2r-1}{2i-1}} - \ceil{\log_2 \max\{2, \frac{n-2r-1-(n/2)}{2i-1} \}} ,\\
b_i &=& \ceil{\log_2 \frac{\min\{n/2, n-2r-1\}}{2i-1}} - 1.
\end{eqnarray*}
Details are omitted (when considering $F$-orbits, $(n-2r-1)-x$ is a useful coordinate).
\end{proof}

\begin{Remark}
Finite abelian group schemes $G$ in characteristic $p>0$ that satisfy
$G[p]=G$, ${\rm Im}(F)=\Ker(V)$, and ${\rm Im}(V)=\Ker(F)$ are called 
\emph{truncated Barsotti--Tate groups of level 1} or $\mathrm{BT}_1$ 
for short.
They play an important r\^{o}le in the theory of $p$-divisible 
groups and have been classified by Oort \cite{OortSimple}.
The proposition  says that $\bG_n^r$ is $\mathrm{BT}_1$ 
if and only if $n\leq 4r+3$.

In the BT$_1$ case, the set $S$ constructed in the proof is decomposed into sequences 
$x_i = (x_{i,j})_{1 \leq j \leq e_i}$, where the index $j$ is considered cyclically, 
such that $2 x_{i,j} \equiv x_{i,j+1} \pmod{n-2r-1}$. 
This amounts to the decomposition of $\bG_n^r$ into indecomposable subgroup schemes.
For each $i$, there is a sequence in $\{F,V\}^{e_i}$ whose $j$-th letter is $F$ (resp. $V$) 
if and only if $x_{i,j+1} = F(x_{i,j})$ (resp.\ $V(x_{i,j+1}) = x_{i,j}$). 
This is the word corresponding to the indecomposable group scheme defined by Oort in 
\cite{OortSimple}.
\end{Remark}

\subsection{Finiteness}
By a theorem of Lipman \cite{Lipman} (Theorem \ref{thm: Lipman}),
rational surface singularities are characterised by having finite
class groups.

Next, local fundamental groups of F-regular singularities
are finite by \cite{CST}.
In particular, the local fundamental groups 
of F-regular RDPs are finite.
Moreover, by Artin's explicit computations \cite{ArtinRDP}, 
the local fundamental groups of all RDPs are finite.
More recently, and more generally, Carvajal-Rojas--Yasuda 
\cite[Theorem B]{Carvajal-RojasYasuda} proved that $2$-dimensional 
klt singularities have finite local fundamental group.

F-injective singularities have trivial $\Picloc^{\loc,\loc}_{X/k}$,
which applies in particular to F-injective RDPs.
As an application of our computations, we now establish the
following finiteness results.
In contrast to the local fundamental groups,
we are not aware of
an apriori reason for this finiteness.

\begin{Theorem}\label{thm: picloclocloc finite}
 Let $x\in X$ be an RDP in characteristic $p>0$.
 Then, the length of $\varinjlim_{m,n} H^2_{\idealm}(X,W_n\OO_X)[F^m]$
 as a $W(k$)-module is finite.
 In particular, $\Picloc_{X/k}^{\loc,\loc}$ 
 is a finite group scheme.
\end{Theorem}

\begin{proof}
If $x\in X$ is F-injective (this holds, for example, if $p>5$), then 
$\varinjlim_{m,n} H^2_{\idealm}(X,W_n\OO_X)[F^m]$ is even trivial.
If $x \in X$ is not F-injective, then this follows from 
Proposition \ref{prop: torsors loclocloc new}.
\end{proof}

\begin{Remark}
Let $x\in X$ be an RDP in characteristic $p>0$.
\begin{enumerate}
 \item It would be desirable to find common properties of the group schemes $\Picloc_{X/k}^{\loc,\loc}$ apart from their finiteness, but such properties seem to be hard to find: $\Picloc_{X/k}^{\loc,\loc}$ is killed by $p$ except 
 precisely if it is of type $E_8^3$ in $p = 2$.
 If it is killed by $p$, then there are
 examples where it is a  $\mathrm{BT}_1$ and there are examples where it is not.
 Moreover, there are examples where $\Picloclocloc_{X/k}$ is autodual and
 there are examples where it is not.
 \item If $p=2$, then the length of $\Picloc_{X/k}^{\loc,\loc}$ can be arbitrarily
 large as Proposition \ref{prop: picloclocloc of DNr} shows.
 Moreover:
\begin{enumerate}
 \item  If $x\in X$ is of type $D_{2^e + 2}^{1}$, then the element 
    $f_{1}^{(e)}$ (notation as in Proposition \ref{prop: torsors loclocloc new})
    is annihiliated by $F$ and $V^e$, but not by $V^{e-1}$.
 \item  We will see in Proposition \ref{prop: alpha2e} below that there exist
  RDPs that are $\balpha_{2^e}$-quotients, which induces an element annihilated by 
  $V$ and $F^e$, but not by $F^{e-1}$.
\end{enumerate}
\end{enumerate}
\end{Remark}

\subsection{A curious coincidence}\label{subsec: curious}
Recall that an RDP is called \emph{taut} if its formal isomorphism class is uniquely 
determined by its resolution graph $\Gamma$. From the explicit classification, 
it follows that an RDP is taut if and only if it is F-regular, 
but see also \cite{Tanaka}.
If $x\in X$ is an RDP of type $\Gamma$ that is not taut, let
$r$ be its Artin co-index as defined in \cite{ArtinRDP} with the modification described 
in Convention \ref{convention}.
If $x\in X$ is taut, we set $r:=0$. 
We note that the co-index is lower semicontinuous in families of RDPs of the same type. 

Let $r_{\max}(\Gamma)+1$ (resp.\ $r_{\max}(\Gamma) + 1/2$) be the number of analytic isomorphism 
classes of RDPs with Dynkin diagram $\Gamma$ in some fixed characteristic $p\geq0$ if 
$(p, \Gamma)$ is not (resp.\ is) $(2, D_{n})$ with $n$ odd. 
Equivalently, $r_{\max}(\Gamma)$ is the difference between the maximal and the minimal 
possible value of $r$.
Thus, $r_{\max}(\Gamma)=0$ if and only if RDPs of type $\Gamma$ are taut.

Next, let $\pi:Y\to X$ be the minimal resolution of singularities of an RDP
$x\in X$, let $E$ be the
exceptional divisor of $\pi$, and let $S:=\Theta_Y(-\log E)$ be the sheaf of
logarithmic derivations.
In characteristic zero, Wahl \cite{Wahl} showed that $H^1_E(Y,S)$ is zero,
and in characteristic $p>0$, Hirokado \cite{HirokadoWahl} computed
the dimensions of $H^1_E(Y,S)$ and $H^1_E(Y,S(E))$.

By \cite[Proposition 2.3]{HirokadoWahl}, we have $H^1_E(Y,S(E)) \cong H^1(Y,S)$ and the 
latter space is the tangent space to the functor of \emph{equisingular} deformations of $Y$, 
which are precisely those deformations of $Y$ where no component of $E$ disappears
(see \cite{Wahlequisingular}).

\begin{Proposition}
 Let $x\in X$ be an RDP of type $\Gamma$ over an algebraically closed
 field of characteristic $p>0$.
 Let $r$ be the Artin co-index of $x$ and let $r_{\max}=r_{\max}(\Gamma)$ as above.
 Then,
 \begin{enumerate}
     \item $\length\, \Dieulocloc(\Picloc_{X/k})$, 
     as a function of $r$, is decreasing (not always strictly decreasing) and attains the maximum value $\floor{r_{\max}}$ at $\floor{r} = 0$.
     \item 
 Let $\pi:Y\to X$ be the minimal resolution of singularities and
 let $E$ be the exceptional divisor of $\pi$. 
 Then,
 \begin{eqnarray*}
  \length\, (\Dieulocloc(\Picloc_{X/k}))[F] &=& 
  \dim H^1_E(Y,\Theta_Y(-\log E)(E))\\
   &=&r_{\max} - r\,.
 \end{eqnarray*}
 \end{enumerate}
\end{Proposition}

\begin{proof}
The assertions about the lengths follow immediately from Proposition \ref{prop: torsors loclocloc new} and
Table \ref{table: wittvectorcohomology}.
The equality $\dim H^1_E(Y,\Theta_Y(-\log E)(E))=r_{\max} - r$ is shown in 
\cite[Theorem 1.1(ii)]{HirokadoWahl}.
\end{proof} 

\begin{Remark}
This coincidence of numbers deserves further study.
It might point to connections between Frobenius actions on local Witt vector 
cohomology, deformations of surface singularities, and their tautness.
It might generalise Tanaka's  tautness (resp. partial tautness)
results \cite{Tanaka} for F-regular (resp. F-split) surface singularities.
\end{Remark}

\section{Detecting smoothness via local torsors}
\label{sec: Mumford}

In this section, we characterise smooth points on surfaces
in terms of local torsors.
The starting point is the following classical result.

\begin{Theorem}[Mumford \cite{Mumford}]
 Let $x\in X$ be a normal complex analytic surface germ.
 Then, $x\in X$ is smooth if and only if $\pitoploc(X)$ is trivial.
\end{Theorem}

Here, the local topological fundamental group of $x\in X$
is defined by choosing a sufficiently small $\varepsilon$-neighbourhood
$B_\varepsilon$ of $x\in X$, setting 
$U:=B_\varepsilon\setminus\{x\}$, choosing a point $u\in U$, 
and defining $\pitoploc(X):=\pitop(U,u)$.

We will generalise this result to arbitrary perfect fields. 
First, we need to define what we mean by a singularity over such a field, which generalises Definition \ref{def: singularity}.

\begin{Definition} \label{def: singularitynonclosed}
 Let $k$ be a perfect field.
 A \emph{normal ($d$-dimensional) singularity} over $k$ is a pair $x\in X$, where $X=\Spec R$, where $(R,\idealm)$ is an integral, normal, local, complete, and 
 Noetherian $k$-algebra of dimension $d \geq 2$ with residue field $k$,
and where $x$ is the closed point corresponding to $\idealm$. 
In this context, we will set $U:=X\backslash\{x\}$.
\end{Definition}

In particular, note that we are assuming that $x \in X$ is a $k$-rational point. 
This guarantees that $X$ is even geometrically integral. 
Since $k$ is perfect, $X$ is geometrically normal, so that $X_{\bar{k}}$ is 
$S_2$, hence $x \in X_{\bar{k}}$ is a singularity in the sense of 
Definition \ref{def: singularity}. 
Finally, again because $k$ is perfect, $X$ is regular if and only if it 
is geometrically regular if and only if it is smooth over $k$.

\subsection{Invariants over non-closed fields}
Let $k$ be a perfect field and let $x \in X$ be a normal singularity as in Definition \ref{def: singularitynonclosed}. When defining the objects parametrising local torsors in Section \ref{sec: local torsors}, we worked over an algebraically closed field, so let us take a moment to see how these objects behave over non-closed fields.

Define 
$$
 \pietloc(X) \,:=\, \Ker\left( \piet(U,\overline{\eta})\to\piet(X,\overline{\eta})\right)
$$
where $\overline{\eta}$ is the geometric generic point of 
$U$ and $X$.
We also define
$$
 \piNloc(X) \,:=\, \Ker\left( \piN(U,x)\to\piN(X,x)\right).
$$
The local Picard functor of $x \in X$ is defined in \cite{Boutot}. The next result, which is probably known to the experts, 
shows the geometric nature of some invariants of $x\in X$.

\begin{Lemma}\label{lem: piloc geometric}
 Let $x\in X$ be as above (recall that $k$ is assumed to be perfect) and let 
 $k\subseteq K$ be an algebraic field extension.
 \begin{enumerate}
     \item The natural homomorphism $\piet(X_K)\to\piet(X)$
 induces an isomorphism
 $$
   \pietloc(X_K) \,\to\,\pietloc(X).
 $$
 \item The natural homomorphism
 $\piN(X_K)\to\piN(X)$ induces an isomorphism
 $$
   \piNloc(X_K) \,\to\,\piNloc(X)\times_kK.
 $$
 \item There is an isomorphism $(\Picloc_{X/k})_{\reduced} \times K \cong (\Picloc_{X_K/K})_{\reduced}$.
 \item If $p>0$, then $x\in X$ is F-injective if and only if 
 $X_{K}$ is F-injective.
 \end{enumerate}
\end{Lemma}

\begin{proof}
It suffices to treat the case where $K$ is an algebraic closure of $k$.
Let $G_k:={\rm Gal}(K/k)$ be the absolute Galois group.

Then, we have a commutative diagram 
\begin{equation} 
\begin{gathered}
\xymatrix{
&\pietloc(X_K) \ar[r]\ar[d] & \pietloc(X)\ar[d]\\
1 \ar[r] & \piet(U_K) \ar[r]\ar[d] & \piet(U) \ar[r]\ar[d] & G_k \ar[r]\ar[d]^{=} & 1\\
1 \ar[r] &\piet(X_K) \ar[r] & \piet(X) \ar[r] & G_k\ar[r] &1
}
\end{gathered}
\end{equation}
where the two bottom rows are short exact sequences. 
The claim about $\pietloc$ now follows from a simple diagram chase.

By \cite[Proposition 5]{Nori}, taking fundamental group schemes
commutes with algebraic separable field extensions, from
which the second assertion follows. 

For the claim on $\Picloc_{X/k}$, the assumption that $x \in X$ is a $k$-rational point 
implies that $X_K$ is still local and the complement of its closed point is $U_K$. 
Thus, from the definition of $\Picloc_{X/k}$ in \cite[Section II]{Boutot}, 
we have $\Picloc_{X/k} \times K \cong \Picloc_{X_K/K}$. 
Since $k$ is perfect, restricting the functors to normal schemes commutes 
with field extensions, hence we get the stated isomorphism.

Finally, if $x \in X$ is a $k$-rational point, then the morphism induced by $X_K \to X$ on coordinate rings is a local homomorphism, so the claim on F-injectivity follows from
\cite[Theorem 3.9 and Theorem A]{Datta}.
\end{proof}

\begin{Remark}
 The class group is \emph{not} geometric, as Salmon's singularity
 in Remark \ref{rem: salmon} and Proposition \ref{prop: Mumfordsproposition} below show.
\end{Remark}

\subsection{The theorems of Flenner and Lipman}
Over fields of characteristic $0$, we have the following
algebraic analogue of Mumford's theorem.

\begin{Theorem}[Flenner]\label{thm: Mumford}
 Let $x\in X$ be normal surface singularity over a field $k$
 of characteristic $0$. Then, $x\in X$ is smooth if and only if 
$\pietloc(X)$ is trivial.
\end{Theorem}

\begin{proof}
If $k=\overline{k}$, then this is \cite{Flenner}.
If not, then Lemma \ref{lem: piloc geometric} shows that $\pietloc(X_{\overline{k}})$
is trivial.
Thus, $x_{\overline{k}}\in X_{\overline{k}}$ is smooth and 
hence, $x\in X$ is smooth.
\end{proof}

\begin{Remark}
In fact, Flenner \cite[Korollar (1.8)]{Flenner} proves that $x \in X$ is smooth 
if and only if $X$ is ``pure'' without assuming that $x \in X$ 
is a $k$-rational point. 
If $x \in X$ is a $k$-rational point, then purity of $X$ 
is the same as triviality of $\pietloc(X)$.
\end{Remark}

\begin{Remark}
This theorem is \emph{not} true in characteristic $p>0$: 
for example, if $x\in X$ is the rational double point of type $A_{p-1}$, then $\pietloc(X)$ is trivial. 
Moreover, in characteristic zero, the abelianisation of the local fundamental group is equal to the 
class group, see Proposition \ref{prop: torsfinet} and Proposition \ref{prop: torsfinabcl}.
Again, the $A_{p-1}$-singularity, which has class group $\C_p$, 
shows that this not true in characteristic $p>0$.
\end{Remark}

In a similar vein, for algebraically closed $k$ of arbitrary characteristic, there is the following relation 
between ${\rm Cl}(X)$ and the rationality of $x\in X$ due to Lipman \cite{Lipman}.

\begin{Theorem}[Lipman]\label{thm: Lipman}
Let $x\in X$ be normal surface singularity over an algebraically closed
field $k$.
 Then, $x\in X$ is rational if and only if ${\rm Cl}(X)$ is finite. 
 Moreover, if ${\rm Cl}(X)$ is trivial, then $x\in X$ is either smooth or 
 an $E_8$-singularity.
\end{Theorem}

\begin{Remark}\label{rem: salmon}
Lipman's theorem does not hold over non-closed fields, even in characteristic $0$.
The following example is due to Salmon \cite{Salmon}: 
for an arbitrary field $k$, 
the surface singularity $x \in X$ over $K:=k(u)$ defined by 
$K[[x,y,z]]/(z^2 + x^3 + uy^6)$ has trivial class group, but it is
an elliptic singularity, that is, if $\pi: \widetilde{X} \to X$ is the minimal resolution of singularities, 
then $R^1 \pi_* \mathcal{O}_{\widetilde{X}}$ is a $1$-dimensional vector space.
In particular, the class group of $X_{\overline{K}}$ is infinite.
\end{Remark}

The following is a corollary of Lemma \ref{lem: piloc geometric}
and Theorem \ref{thm: Lipman}.

\begin{Corollary}\label{cor: picloctwisted}
If $x\in X$ is a normal rational surface singularity over a perfect
field $k$, then $({\rm Picloc}_{X/k})_{\reduced}$ is a twisted form of the constant 
\'etale group scheme $\underline{{\rm Cl}(X)}$.
\end{Corollary}

\begin{proof}
Since $x\in X$ is rational, so is $X_{\overline{k}}$, 
hence ${\rm Cl}(X_{\overline{k}})$ is finite by Theorem \ref{thm: Lipman}.
But ${\rm Cl}(X_{\overline{k}})$ is the group of $\overline{k}$-rational points of 
$({\rm Picloc}_{X/k})_{\reduced} \times \overline{k} \cong ({\rm Picloc}_{X_{\overline{k}}/\overline{k}})_{\reduced}$, 
hence $\underline{{\rm Cl}(X)} \cong ({\rm Picloc}_{X/k})_{\reduced} \times \overline{k}$ 
and the claim follows.
\end{proof}

\subsection{The Flenner--Mumford theorem in characteristic $\mathbf{p}$}
Now, assume that $k$ is perfect of characteristic $p > 0$.

\begin{Theorem}
 \label{thm: Mumfordstheorem}
 Let $x\in X$ be a normal surface singularity over a perfect field $k$ of characteristic $p > 0$.
     Then, the following are equivalent:
     \begin{enumerate}
  \item \label{thm: item: smooth} $x\in X$ is a smooth point.
  \item \label{thm: item: piNloc trivial} $\piNloc(X)$ is trivial.
  \item \label{thm: item: and F-injective}
  $(\Picloc_{X/k})_{\reduced}$ is trivial, $\pietloc(X)$ is trivial, and $X$ is F-injective,
    \end{enumerate}
    If $k$ is algebraically closed, then these are also equivalent to
    \begin{enumerate}
    \setcounter{enumi}{3}
        \item  \label{thm: item: algclosed} ${\rm Cl}(X)$ is trivial, $\pietloc(X)$ is trivial, and $X$ is F-injective.
    \end{enumerate}
\end{Theorem}

\begin{proof}
By Lemma \ref{lem: piloc geometric}, Assertions (\ref{thm: item: smooth}), (\ref{thm: item: piNloc trivial}), and (\ref{thm: item: and F-injective}) are invariant under passing to $\overline{k}$, so we can and will assume that $k$ is algebraically closed.

By Corollary \ref{cor: noritrivial} and Lemma \ref{lem: piloc geometric}, 
the implication (\ref{thm: item: smooth}) $\Rightarrow$ (\ref{thm: item: piNloc trivial}) follows immediately from purity for $X_{\overline{k}}$.

Next, we prove (\ref{thm: item: piNloc trivial}) $\Rightarrow$ (\ref{thm: item: and F-injective}). By Corollary \ref{cor: noritrivial}, we know that $\overline{\Hfl{1}(U,G)}$ is trivial for every finite and abelian or finite and \'etale group scheme $G$. 
In particular, $\pietloc(X)$ is trivial by Lemma \ref{lem: piloc geometric} and $X$ is F-injective by 
Lemma \ref{lem: Finjective}. 
By Proposition \ref{prop: torsfinab}, we also know that 
$(\Picloc_{X/k})_{\reduced}[F]$ is trivial, hence $(\Picloc_{X/k})_{{\reduced}}$ is discrete, that is, its tangent space is trivial. 
This tangent space is nothing but $H^1(\widetilde{X},\mathcal{O}_{\widetilde{X}})$, where $\widetilde{X}$ is the minimal resolution of $X$. 
Hence, $X$ is rational, so $(\Picloc_{X/k})_{\reduced}$ is finite by Theorem \ref{thm: Lipman} and Corollary \ref{cor: picloctwisted}, 
and thus we can apply Proposition \ref{prop: torsfinabcl} to the finite group scheme $(\Picloc_{X/k})_{\reduced}$ to deduce that it is in fact trivial. 
Hence, (\ref{thm: item: piNloc trivial}) $\Rightarrow$ (\ref{thm: item: and F-injective}).

The implication (\ref{thm: item: and F-injective}) $\Rightarrow $ (\ref{thm: item: algclosed}) is clear because ${\rm Cl}(X) = (\Picloc_{X/k})_{\reduced}(k)$.

Finally, we prove (\ref{thm: item: algclosed})  $\Rightarrow $ (\ref{thm: item: smooth}). Since ${\rm Cl}(X)$ is trivial, Theorem \ref{thm: Lipman} implies that $X$ is either smooth or an RDP of type $E_8$. But by Table \ref{table: RDPEquations} and Table \ref{table: torsors with Picloclocloc}, there is no $F$-injective RDP $X$ of type $E_8$ such that $\pietloc(X)$ is trivial. Hence, $X$ must be smooth, which finishes the proof.
\end{proof}

Over non-closed perfect fields, we can classify all normal surface singularities that satisfy Assertion (\ref{thm: item: algclosed}) of Theorem \ref{thm: Mumfordstheorem}.

\begin{Proposition} \label{prop: Mumfordsproposition}
 Let $x\in X$ be a normal surface singularity over a perfect field $k$ of characteristic $p > 0$. 
 Assume that $x\in X$ is not smooth. 
 Then, the following are equivalent:
 \begin{enumerate}
     \item \label{prop: item: and F-injective} ${\rm Cl}(X)$ is trivial, $\pietloc(X)$ is trivial, 
     and $X$ is F-injective.
     \item \label{prop: item: counterexample} There exist $n \geq 0$ and 
     $a,b \in k[[x,y,z]]^{\times}$ such that $-ab$ is not a square modulo $(x,y,z)$ 
     and such that $X \cong \Spec k[[x,y,z]]/(z^{p^n} + ax^2 + by^2)$. 
     (This is an RDP of type $B_{(p^n-1)/2}$ in Lipman's terminology \cite{Lipman}.)
 \end{enumerate}
In particular, if $p = 2$, then (\ref{thm: item: algclosed}) 
implies (\ref{thm: item: smooth}) in Theorem \ref{thm: Mumfordstheorem}.
\end{Proposition}

\begin{proof}
The implication (\ref{prop: item: counterexample}) $\Rightarrow $ (\ref{prop: item: and F-injective})
follows from \cite[Remark 25.3]{Lipman}. 
Indeed, Lipman proves that ${\rm Cl}(X)$ is trivial and since $X_{\overline{k}}$ is obviously 
an RDP of type $A_{p^n-1}$, it follows from Table \ref{table: RDPEquations} and 
Table \ref{table: torsors with Picloclocloc} that $X$ is F-injective with trivial $\pietloc(X)$.

For the converse, we first claim that $X$ is rational. 
Indeed, since $\pietloc(X_{\overline{k}})$ is trivial, the torsion subgroup 
of ${\rm Cl}(X_{\overline{k}})$ is $p$-torsion. 
Therefore, the local Picard variety $(\Picloc_{X_{\overline{k}}/\overline{k}})^{\circ}_{\reduced}$ 
is unipotent, hence so is 
$(\Picloc_{X/k})^{\circ}_{\reduced}$, again by 
Lemma \ref{lem: piloc geometric}. 
Since $k$ is perfect, $(\Picloc_{X/k})^{\circ}_{\reduced}$ is split, that is, 
it is either trivial or contains $\mathbb{G}_{a,k}$. 
The latter case cannot occur, since $(\Picloc_{X/k})^{\circ}_{\reduced}(k) \subseteq {\rm Cl}(X)$, 
and the latter is trivial by assumption.
Hence, $(\Picloc_{X/k})^{\circ}_{\reduced}$ is trivial and thus, $X$ is rational.

If $p \neq 2$, the rational factorial normal surface singularities have been classified in 
\cite[Remark 25.3]{Lipman}. 
From the classification, it follows that $X_{\overline{k}}$ is an RDP of type $D_4,E_6,E_8$ or $A_n$. 
Since we assume that $\pietloc(X)$ is trivial and $X$ is F-injective, we can use 
Table \ref{table: RDPEquations} and Table \ref{table: torsors with Picloclocloc} 
to deduce that $X_{\overline{k}}$ is of type $A_{p^n-1}$. 
Then, again by \cite[Remark 25.3]{Lipman}, $X$ is given by an equation as in 
Assertion (\ref{prop: item: counterexample}).

If $p = 2$, then we use that $(\Picloc_{X/k})_{\reduced}$ is a twisted form of 
the constant group scheme 
associated to the cyclic $2$-group ${\rm Cl}(X_{\overline{k}})$ by Corollary \ref{cor: picloctwisted}. 
Since ${\rm Cl}(X_{\overline{k}})$ contains a unique non-trivial element of order $2$, this element is fixed by the 
Galois action, hence descends to a non-zero element in ${\rm Cl}(X)$, contradicting our assumption 
that ${\rm Cl}(X)$ is trivial. 
This finishes the proof.
\end{proof}
 
\begin{Remark}
We note that a slightly weaker version of Theorem \ref{thm: Mumfordstheorem} and over
algebraically closed fields was already obtained by Esnault and Viehweg 
\cite[Corollary 4.3]{EsnaultViehweg}. 
Our detailed study of rational double points in small characteristics allows 
us to treat the cases that were left open by them.
\end{Remark}

If $p\geq7$, then we may drop the condition on F-injectivity 
in (\ref{thm: item: and F-injective}) of the theorem.

\begin{Corollary}
 \label{cor: mumford}
 Let $x\in X$ be a normal surface singularity over a perfect
 field of characteristic $p\geq7$. 
 Then, the following are equivalent:
  \begin{enumerate}
  \item \label{cor: item: smooth}$x\in X$ is a smooth point.
   \item \label{cor: item: piNloc trivial} $\piNloc(X)$ is trivial.
  \item \label{cor: item: Cl and pietloc trivial}   $(\Picloc_{X/k})_{\reduced}$ is trivial and $\pietloc(X)$ is trivial.
  \end{enumerate}
      If $k$ is algebraically closed, then these are also equivalent to
    \begin{enumerate}
    \setcounter{enumi}{3}
        \item  \label{cor: item: algclosed} ${\rm Cl}(X)$ is trivial and $\pietloc(X)$ is trivial.
    \end{enumerate}
\end{Corollary}

\begin{proof}
As in the proof of Theorem \ref{thm: Mumfordstheorem}, we may assume that $k$ is algebraically closed.
The implications (\ref{cor: item: smooth}) $\Rightarrow$ (\ref{cor: item: piNloc trivial}) and 
(\ref{cor: item: piNloc trivial}) $\Rightarrow$ (\ref{cor: item: Cl and pietloc trivial}) 
have been shown in Theorem \ref{thm: Mumfordstheorem}. 
The implication (\ref{cor: item: Cl and pietloc trivial}) $\Rightarrow$ (\ref{cor: item: algclosed}) 
is clear.
Now, Table \ref{table: RDPEquations} shows that $p \geq 7$ implies that $X$ is 
automatically F-injective, so that (\ref{cor: item: algclosed}) 
$\Rightarrow$ (\ref{cor: item: smooth}) holds as well.
\end{proof}

\begin{Remark}
The equivalence (\ref{cor: item: smooth}) $\Leftrightarrow$ (\ref{cor: item: piNloc trivial}) 
over algebraically closed fields of characteristic $p\geq7$ is due to 
Esnault and Viehweg \cite{EsnaultViehweg}.
\end{Remark}

\begin{Remark}
Both conditions in (\ref{cor: item: Cl and pietloc trivial}) of 
Corollary \ref{cor: mumford} are needed for 
the implication (\ref{cor: item: Cl and pietloc trivial}) $\Rightarrow$ (\ref{cor: item: smooth}) to hold, even in arbitrarily large characteristics: 
indeed, the $A_{p-1}$-singularity has trivial local 
\'etale fundamental group and a non-trivial class group, 
whereas the $E_8$-singularity has trivial class group and a 
non-trivial local \'etale fundamental group.

Finally, if $p \leq 5$, then all three conditions in (\ref{thm: item: and F-injective}) of Theorem \ref{thm: Mumfordstheorem} 
are needed: this can be checked immediately using
Table \ref{table: RDPEquations}
and Table \ref{table: torsors with Picloclocloc}. 
\end{Remark}

\begin{Remark}
We do not know whether Theorem \ref{thm: Mumfordstheorem} holds over non-perfect fields (with smoothness replaced by regularity). 
The proof suggests that in order find a counterexample, one should look at 
normal surface singularities that are F-injective but not geometrically F-injective. 
By \cite[Proposition 4.2]{Enescu}, geometric F-injectivity does not necessarily 
imply F-injectivity over non-perfect fields. 
However, it seems to be unknown whether normal surface singularities that are 
geometrically F-injective but not F-injective exist.
\end{Remark}

\section{Very small actions and quotient singularities}
\label{sec: Nori}
From now on, $k$ is an algebraically closed field. 
Let $x \in X$ be a singularity as in Definition \ref{def: singularity}. 
In analogy with \cite[Definition 6.1 and Definition 6.2]{LRQ}, 
we define very small actions of finite group schemes and use this notion to 
define quotient singularities.

\begin{Definition}
 Let $G$ be a finite group scheme over $k$. 
 A \emph{very small action} of $G$ on $X$ is a $G$-action that is free on 
 $U = X \setminus \{x\}$ and such that $x \in X^{G}$, 
 where $X^G$ denotes the fixed locus.
\end{Definition}

\begin{Definition}
 \label{def: quotient singularity}
 A \emph{quotient singularity} $x\in X$ by a finite group scheme $G$ over $k$ is a 
 singularity $x \in X$ such that
 $$
   X \,\cong\, \widehat{\Aff}^d/G \,=\, \Spec ~ k[[u_1,...,u_d]]^G,
 $$
 where $G$ acts on $\widehat{\Aff}^d := \Spec k[[u_1,...,u_d]]$ via a very small action.
\end{Definition}

We refer to \cite[Remark 6.3]{LRQ} for a small discussion of this definition.

How can one decide whether a given singularity $x \in X$ is a quotient singularity 
by a finite group scheme $G$? 
As mentioned in Remark \ref{rem: quotientisnormal}, the singularity $x \in X$ has to be normal. Moreover, a representation of $X$ as a quotient singularity of 
$Y := \widehat{\Aff}^d = \Spec k[[u_1,...,u_d]]$ 
by $G$ gives rise to a local $G/N$-torsor over $X$ for every normal subgroup 
scheme $N \subseteq G$.
Since these local torsors are again quotient singularities, these local torsors must 
have normal total spaces.
In particular, Example \ref{ex: non-normal} shows that \emph{not} 
every non-trivial local torsor over $X$ arises in this way. 
In the following lemma, we collect some properties that local torsors of the form
$Y/N \to X$ have to satisfy.

\begin{Lemma} \label{lem: quotient induces local torsor}
Let $G$ be a finite group scheme over $k$ with a very small action on 
$Y = \widehat{\Aff}^d = \Spec k[[u_1,...,u_d]]$. Then, the following holds:
\begin{enumerate}
    \item For every subgroup scheme $H \subseteq G$, the quotient $Y/H$ is normal.
    \item For every subgroup scheme $1 \neq H \subseteq G$, the morphism $Y \to Y/H$ is a local $H$-torsor over $Y/H$ that does not extend to a global $H$-torsor.
    \item \label{item: quotient induces local torsor} For every normal subgroup scheme $N \subsetneq G$, the morphism $Y/N \to X$ 
    is a local $G/N$-torsor over $X$ that does not extend to a global $G/N$-torsor.
\end{enumerate}
\end{Lemma}

\begin{proof}
Set $S := k[[u_1,...,u_d]]$. Then we have $Y/H = \Spec S^H$, and since $S^H = {\Frac}(S^H) \cap S$, it is normal.

As for Claim (2), use that the $H$-action on $Y$ is very small, hence non-free, 
and therefore $Y \to Y/H$ cannot be a torsor, unless $H$ is trivial.

Finally, consider Claim (3). 
Since $G$ fixes the closed point of $Y$, $G/N$ fixes the closed point of $Y/N$,
hence $Y/N \to X$ cannot be a torsor.
\end{proof}
 
We also recall that only very special finite and connected group schemes 
admit very small actions.
 
\begin{Proposition}[{\cite[Proposition 6.9]{LRQ}}] \label{prop: linearlyreductive or unipotent}
 Let $G$ be a finite group scheme admitting a very small action on $X = \widehat{\Aff}^d$.
 Then the scheme underlying $G^{\circ}$ is isomorphic to $\Spec k[t]/(t^{p^n})$ 
 for some $n \geq 0$.
 Moreover, if $n\geq1$, then
 \begin{enumerate}
     \item \label{item: linearly reductive height n} either $G^{\circ}$ is linearly reductive of height $n$, hence isomorphic to $\bmu_{p^n}$,
     \item \label{item: unipotent height n} or else $G^{\circ}$ is unipotent of height $n$, hence $G^{\circ}[F^{i+1}]/G^{\circ}[F^{i}] \cong \balpha_p$ for all $0 \leq i < n$.
 \end{enumerate}
\end{Proposition}

\subsection{The local Nori fundamental group scheme}
A natural guess is that $\piNloc(U,X,x) \cong G$ for a quotient singularity $X=Y/G$.
First, we note that in this setting $G$ is always a quotient of $\piNloc(U,X,x)$.

\begin{Lemma}\label{lem: quotientsofpinloc}
 Let $x\in X := Y/G$ be a quotient singularity by a finite group scheme $G$ over $k$
 with $Y=\Spec k[[u_1,...,u_d]]$ as in Definition \ref{def: quotient singularity}.
 Then,
 \begin{enumerate}
     \item \label{item: quotientNori-reduced} $f : V \to U$ is Nori-reduced.
     \item \label{item: quotientsofpinloc2}
     
     The composition  
     $$
   \piNloc(U,X,x) \,\to\, \piNloc(U,x) \,\to\, G
    $$
    is surjective.
    \item \label{item: quotientsofpinloc etale} The induced morphisms on maximal \'etale quotients
    $$
   \piNloc(U,X,x)^{\et} \,\to\, \piNloc(U,x)^{\et} \,\to\, G^{\et}
    $$
    are isomorphisms.
 \end{enumerate}
\end{Lemma}
\begin{proof}
Let us first prove (\ref{item: quotientNori-reduced}). 
Assume that $f$ admits a reduction of structure to a subgroup $H \subseteq G$, 
that is, there exists an $H$-torsor $g: W \to U$ together with an 
$H$-equivariant morphism $\psi: W \to V$. 
Since $V$ and $U$ are integral, the induced maps 
$\mathcal{O}_{U,\eta_U} \to \mathcal{O}_{V,\eta_V} \to \mathcal{O}_{W,\eta_W}$
on the local rings at the respective generic points are injective, 
hence $|H| \geq |G|$. 
Since $H$ is a subgroup of $G$, this implies that $H = G$, 
hence $f_V$ is Nori-reduced. 

As for (\ref{item: quotientsofpinloc2}), since $f$ is Nori-reduced, there is an associated
surjective homomorphism $\piNloc(U,x)\to G$.
Let $\psi:\piNloc(U,X,x)\to\piNloc(U,x)\to G$ be the composition
and let $H:=\psi(\piNloc(U,X,x))$ be the image, which is a normal
subgroup scheme of $G$.
Thus, we obtain a commutative diagram of profinite group schemes, 
all of whose rows are exact and all of whose vertical arrows 
are surjective
\begin{equation} 
\begin{gathered}
\xymatrix{
1 \ar[r] & \piNloc(U,X,x) \ar[r]\ar[d] & \piNloc(U,x) \ar[r]\ar[d] & \piN(X,x) \ar[r]\ar[d] & 1\\
1 \ar[r] &H \ar[r] &G \ar[r] & G/H\ar[r] &1
}
\end{gathered}
\end{equation}
Now, if $\psi$ was not surjective, then $G/H$ would be non-trivial, which means
that the $G/H$-torsor $V/H\to U$ would extend to a $G/H$-torsor over $X$, contradicting Lemma \ref{lem: quotient induces local torsor}.

Finally, Claim (\ref{item: quotientsofpinloc etale}) is a consequence of the triviality of
$\piet(Y)$ and $\pi(V)$, as well as Grothendieck's theory
of the \'etale fundamental group.
\end{proof}

In particular, there exists a natural surjective morphism $\piNloc(U,X,x)\to G$, 
which identifies the maximal \'etale quotients of both group schemes.

\begin{Question}\label{ques: pinloc}
If $x\in X=Y/G$ is a quotient singularity by a finite $k$-group scheme $G$,
is the surjective morphism $\piNloc(U,X,x) \to G$ an isomorphism?
\end{Question}

Quite surprisingly, the following examples show that this is \emph{not} true,
not even for rational double points.

\begin{Example} \label{ex: weirdpinlocs}
In the following examples taken from Table \ref{table: torsors with Picloclocloc}, 
$x\in X$ is a quotient singularity 
by a finite group scheme $G$ and there exists a simple finite group scheme $H$ and a local $H$-torsor over $x\in X$ that does not extend to $X$ and such that $H$ is not a quotient of $G$. 
\begin{enumerate}
    \item $p = 3$, $X = E_6^1$, $G = \C_3$, $H = \bmu_3$.
    \item $p = 3$, $X = E_6^0$, $G = \balpha_3$, $H = \bmu_3$.
    \item $p = 2$, $X = E_8^2$, $G = \C_2$, $H = \balpha_2$.
\end{enumerate}
In all cases, the local $H$-torsor is necessarily Nori-reduced, since $H$ is simple. 
Thus, by the same argument as in the proof of 
Lemma \ref{lem: quotientsofpinloc} \eqref{item: quotientsofpinloc2}, 
the induced homomorphism $\piNloc(U,X,x) \to G$ is surjective, but since $H$ is not a quotient 
of $G$, we obtain $\piNloc(U,X,x)\not\cong G$.
\end{Example}

All these examples share the  property that $G$ is not linearly reductive, which leads
the following modified question:

\begin{Question}\label{ques: pinloc modified}
Does Question \ref{ques: pinloc} have a positive answer if $G$ is linearly 
reductive?
\end{Question}

\begin{Remark}
By \cite[Theorem 8.1]{LRQ}, the quotient presentation $X=Y/G$
of an lrq singularity $x\in X$ is unique. 
This could be viewed as evidence for a positive answer to 
Question \ref{ques: pinloc modified}.
However, we warn the reader that this uniqueness does \emph{not} imply that 
$\piNloc(U,X,x) \cong G$, since there might be other strictly local and non-smooth torsors 
over $X$, nor does the F-regularity of $X$ imply that $\piNloc(U,X,x)$ 
admits no unipotent quotients, since the map $\Hom(\piNloc(U,x),G) \to \Hom(\piNloc(U,X,x),G)$ 
in Proposition \ref{prop: torspinloc} might not be surjective. 
For example, we will prove uniqueness of the presentation $Y=X/G$ in the examples
of Example \ref{ex: weirdpinlocs} by Theorem \ref{thm: quotient} below, yet
Question \ref{ques: pinloc} has a negative answer in these cases.
\end{Remark}

\subsection{Some illuminating computations}
The fact that Question \ref{ques: pinloc} has a negative answer in general is quite remarkable:
\begin{enumerate}
    \item the examples of Example \ref{ex: weirdpinlocs} admit more local torsors
    than can be ``explained'' by the group scheme $G$.
    \item Actually, it could also be possible that $\piNloc(U,X,x)$ has quotients that 
    do not correspond to local torsors, see Proposition \ref{prop: torspinloc}.
\end{enumerate}
In order to get at least some idea and a glimpse at what is going on, we now study the 
surjective homomorphism $\piNloc(U,X,x)\to G$ from Lemma \ref{lem: quotientsofpinloc}
in detail.
In particular, we will see that a tempting approach to establishing
Question \ref{ques: pinloc} does not work, but leads to interesting insights.

Let $x\in X = Y/G$ be a quotient singularity as in Definition \ref{def: quotient singularity}.
With $U=X\setminus\{x\}$ and $V=Y\setminus\{y\}$ as usual,
we have a commutative diagram with exact rows, 
where the vertical maps are induced by pullbacks of torsors
\begin{equation} 
\label{eq: Noridiagram}
\begin{gathered}
\xymatrix{
1 \ar[r] & \piNloc(V,Y,y) \ar[r]\ar[d] & \piNloc(V,y) \ar[r]\ar[d]^{\pi_1} & \piN(Y,y) \ar[r]\ar[d]^{\pi_2} & 1\\
1 \ar[r] & \piNloc(U,X,x) \ar[r] & \piNloc(U,x) \ar[r] & \piN(X,x) \ar[r] & 1\\
}
\end{gathered}
\end{equation}
Since $y\in Y$ is a smooth point, we have that $\piNloc(V,Y,y)$ is trivial. 
If we assume that the images of the $\pi_i$ are normal, then the non-commutative
snake lemma will yield an exact sequence 
\begin{equation}
\label{eq: snake piloc}
\begin{split}
\xymatrix{
 & 1 \ar[r] & \Ker(\pi_1) \ar[r] & \Ker(\pi_2)   \ar@{->} `r[d] `[l] `^d[lll] `[dll]  [dll]  & \\
 & \piNloc(U,X,x) \ar[r] & \Coker(\pi_1) \ar[r] & \Coker(\pi_2) \ar[r] & 1
}
\end{split}
\end{equation}

By Lemma \ref{lem: quotientsofpinloc}, we would get that 
$\Coker(\pi_1)$ admits a surjection onto $G$.
Thus, in order to approach Question \ref{ques: pinloc} and Question \ref{ques: pinloc modified},
one might be tempted to believe that $\Ker(\pi_2)$ is trivial and that
$\Coker(\pi_1)\cong G$: then, $\Ker(\pi_1)$ would be trivial, $\Coker(\pi_2)$ would be trivial, 
and $\piNloc(U,X,x)\cong G$.
We note that this is true for the maximal \'etale quotients.

In the following, we will show that the above approach already fails in the 
simplest possible non-\'etale case, namely if $G = \bmu_p$. Below, we will consider the following example.

\begin{Example}
\label{ex: Ap-1 counterexample}
 Let $G:=\bmu_p$ act diagonally on $Y := \Spec k[[u_1,u_2]]$ with weights $(1,-1)$
 and set $X:=Y/G$, that is, $x\in X$ is a rational double point of type $A_{p-1}$.
\end{Example}

First, we note that $\Coker(\pi_1)$ is larger than $G$ in Example \ref{ex: Ap-1 counterexample}.

\begin{Lemma} \label{lem: cokerneltoolarge}
In the situation of Example \ref{ex: Ap-1 counterexample}, we have $\Coker(\pi_1) \not \cong G$, that is, the surjection from Lemma \ref{lem: quotientsofpinloc}
$$
 \Coker(\pi_1) \,\to\, G
$$
is not an isomorphism.
\end{Lemma}

Here, $\Coker(\pi_1)$ denotes the categorical cokernel of $\pi_1$, which can be described as the quotient of the target by the normal closure of the image of
$\pi_1$.

\begin{proof}
The quotient map $Y \to X$ corresponds to the
inclusion of rings $k[[u_1^p,u_2^p,u_1u_2]] \subseteq k[[u_1,u_2]]$.
Consider the $\balpha_p$-torsor $Z \to X$ given by adjoining 
a $p$-th root of $u_1^{p}$. Then, $Z = \Spec k[[u_1,u_2^{p},u_1u_2]]$ and $Y \to X$ factors through $Z$. 
Thus, the $\balpha_p$-torsor $Z \times_X Y \to Y$ is trivial.
In terms of Nori fundamental group schemes, this means that the Nori-reduced 
$\balpha_p$-torsor $Z \to X$ corresponds to a quotient map $\piNloc(U,x) \to \balpha_p$
that becomes trivial when precomposed with the natural homomorphism 
$\piNloc(V,y) \to \piNloc(U,x)$.
By the universal property of $\Coker(\pi_1)$, this means that $\piNloc(U,x) \to \balpha_p$ factors through a quotient map $\Coker(\pi_1) \to \balpha_p$. 
In particular, $\Coker(\pi_1)$ is larger than $G = \bmu_p$.
\end{proof}

\begin{Remark}
Note that the $\balpha_p$-torsor that we used in the above proof has a
non-normal total space. 
It is this non-normality that allows us to ``sandwich'' 
it between $X$ and $Y$. 
\end{Remark}

\begin{Remark}
\label{rem: cokerneltoolarge}
There are at least two ways of ``reading'' this lemma.
Namely, each of the the following two statements is equivalent
to Lemma \ref{lem: cokerneltoolarge}:
\begin{enumerate}
    \item There exist local torsors over $X$ that are not of the form $Y/H$ for some $H \subseteq G$, but that
    are nevertheless trivialised by pullback along $Y \to X$.
    \item There exist Nori-reduced $H$-torsors $W \to U$ that dominate $V \to U$, but whose
     induced $\Ker(H \to G)$-torsor $W \to V$ is not Nori-reduced.
\end{enumerate}

Note that neither of these two phenomena can occur in the setting of \'etale fundamental groups,
see, for example, \cite[Proposition 5.3.8]{Szamuely}.
\end{Remark}

If $Y\to X$ is a morphism between pointed schemes $(Y,y)$ and $(X,x)$,
then the induced morphism $\piet(Y,y) \to \piet(X,x)$  on \'etale fundamental groups 
is injective if and only if every finite \'etale torsor over $Y$ 
is dominated by the pullback of a finite \'etale torsor over $X$, 
see \cite[Corollary 5.5.9]{Szamuely}. 
In particular, if $Y \to X$ is \'etale and $Z \to Y$ is a Galois cover, 
then the tower $Z \to Y \to X$ admits a Galois closure, 
hence the map $\piet(Y,y) \to \piet(X,x)$ is injective.

Thus, assuming for a moment that \cite[Corollary 5.5.9]{Szamuely} also holds in the 
setting of local Nori fundamental group schemes, 
one is tempted to try to apply the approach of the previous paragraph to prove the 
injectivity of $\piNloc(V,y) \to \piNloc(U,x)$. 
The subtle question whether Galois closures exist in the setting of torsors under 
finite group schemes has been studied by Antei et. al. \cite{ABETZ}, where they refine
unsuccessful attempts from \cite{ABE} (withdrawn) and \cite{Garuti}.
In particular, there are examples of towers of torsors that do \emph{not} admit
a Galois closure \cite[Example 5.4]{ABETZ}. 
It turns out that this also happens in the setting of Example \ref{ex: Ap-1 counterexample}, 
as we will show below. 
In particular, it is unclear whether the natural map $\piNloc(V,y) \to \piNloc(U,x)$ 
is injective.

\begin{Lemma}
\label{lemma example 2}
 In the situaton of Example \ref{ex: Ap-1 counterexample}, there is a $\bmu_p$-torsor $Z \to Y$ such that the composition $Z \to X$ cannot be dominated by a local $H$-torsor over $X$ for any finite group scheme $H$.
\end{Lemma}

\begin{proof}
Consider the composition of the two maps 
\begin{eqnarray*}
Z = \Spec~k[[u_1^{1/p},u_2]] &\to& Y = \Spec~k[[u_1,u_2]] \\
Y &\to& X = \Spec~k[[u_1^p,u_2^p,u_1u_2]]
\end{eqnarray*}
as a composition of a $\bmu_p$-torsor $Z \to Y$ obtained by taking a $p$-th root of $1 + u_1$ 
and a local $\bmu_p$-torsor $Y \to X$ as in the previous example.

Assume that there exists a Galois closure for the tower $Z \to Y \to X$, that is, a local $H$-torsor $Z' \to X$ that factors equivariantly through $Z$. 
In particular, there are surjective homomorphisms $\alpha: H \to \bmu_p$ and 
$\beta: \Ker(\alpha) \to \bmu_p$.
Note that $Z' \to X$ cannot be a global torsor, hence the stabiliser $\Stab(z')$ 
of the closed point in $Z'$ is non-trivial. 
Since $Y$ is smooth, $\Ker(\alpha)$ acts freely on $Z'$ by purity. So, $\Stab(z') \cap \Ker(\alpha) = \{{\rm id}\}$, hence $\alpha$ restricted to 
$\Stab(z')$ is injective. Since $\bmu_p$ is simple, this implies that $\alpha$ induces an isomorphism $\Stab(z') \to \bmu_p$. 
In particular, we have $H \cong \Ker(\alpha) \rtimes \bmu_p$. 

Since $H \cong \Ker(\alpha) \rtimes \bmu_p$ and $\Ker(\alpha)/\Ker(\beta) \cong \bmu_p$, 
we have an isomorphism of schemes (but not necessarily of group schemes, since $\Ker(\beta)$ might not be normal in $H$)  
$$
H/\Ker(\beta) \,\cong\, \Ker(\alpha)/\Ker(\beta) \times \Spec k[x]/(x^p) \,\cong\, 
\Spec k[x,y]/(x^p,y^p).
$$
In particular, letting $\overline{K} := \overline{k(X)}$
be the geometric generic point of 
$X$, we have
\begin{equation}\label{eq: counterexampleforgaloisclosures}
Z_{\overline{K}} \,\cong\, Z'_{\overline{K}}/\Ker(\beta)_{\overline{K}} 
\,\cong\, H_{\overline{K}}/\Ker(\beta)_{\overline{K}} 
\,\cong\, \Spec \overline{K}[x,y]/(x^p,y^p).
\end{equation}
For the second isomorphism, we have used that the $H$-torsor 
$Z'_{\overline{K}} \to \Spec \overline{K}$ is trivial, since $\overline{K}$ is algebraically closed.
For the third isomorphism we have used that
$H/\Ker(\beta)$ is a universal geometric quotient, since $\Ker(\beta)$ acts freely on $H$.

On the other hand, our example satisfies $k(Z) \cong k(X)[\sqrt[p^2]{u_1^p}]$. This implies
$Z_{\overline{K}} \cong \Spec \overline{K}[x]/(x^{p^2})$, which contradicts 
Equation \eqref{eq: counterexampleforgaloisclosures}. 
Hence, the local $H$-torsor $Z' \to X$ cannot exist.
\end{proof}

\begin{Remark}
We note that if $G$ is \'etale, then Galois closures exist by \cite[Theorem 3.9]{ABETZ}. Nevertheless, even if $G$ is \'etale, then
it is not always true that $\Ker(\pi_2)$ is trivial, that $\Image(\pi_1)$ and $\Image(\pi_2)$ are normal, and that
$\Coker(\pi_1) = G$. 
If it was, then $\piNloc(U,X,x)^{\et} = G^{\et}$ would imply $\piNloc(U,X,x) = G$ 
if $G$ is \'etale,  but this fails for the examples (1) and (3) of 
Example \ref{ex: weirdpinlocs}.
\end{Remark}

\section{Quotient and non-quotient RDPs} \label{sec: rdp}

Using the results of Section \ref{sec: torsors over rdp} on local torsors over RDPs, 
we will now study the question whether a given RDP can be realised as a quotient singularity 
and if so, whether the finite group scheme realising it 
as a quotient singularity is unique.
Table \ref{table: Quotient- and non-quotient RDPs} gives our results, 
where $\BD_{n}, \BT_{24}, \BO_{48}, \BI_{120}$ are the binary polyhedral group schemes, 
which are \'etale except for $\BD_{n}$ with $p \mid n$ (see \cite[p.8]{LiedtkeSatriano} for the precise definition). 
The group $\Dic_{12}$ is the dicyclic group of order $12$. In this section, we will prove the following two theorems.

\begin{Theorem} \label{thm: quotient}
Let $x\in X$ be an RDP in characteristic $p>0$
marked with $\checkmark$ in Table \ref{table: Quotient- and non-quotient RDPs}. 
Then, the following holds:
\begin{enumerate}
    \item $x\in X$ is a quotient singularity by the indicated group scheme $G$.
    \item  For every realisation $X \cong \widehat{\Aff}^d/H$ of $X$ as a quotient singularity by some finite group scheme $H$, we have $H \cong G$.
\end{enumerate}
\end{Theorem}

\begin{Theorem} \label{thm: not quotient}
Let $x\in X$ be an RDP in characteristic $p>0$ marked with $\times$ in Table \ref{table: Quotient- and non-quotient RDPs}.
Then, $x \in X$ is not a quotient singularity.
\end{Theorem}

This leaves open the case of the $D_n^r$-singularities with $2 < 4r < n$ in characteristic $2$. Some of them are quotient singularities (Proposition \ref{prop: alpha2e}) while others are not (Proposition \ref{prop: notquotient dnr}) and finite group schemes of arbitrarily large length occur in quotient realisations of these $D_n^r$-singularities (Proposition \ref{prop: alpha2e}). 
It would be interesting to find the pattern that underlies these phenomena. As a partial result in this direction, we will determine the finite abelian group schemes that can occur in quotient realisations of these $D_n^r$-singularities.

\begin{Theorem} \label{thm: D_n in p=2}
Let $x \in X$ be an RDP of type $D_n^r$ with $4r < n$ in characteristic $2$.
Let $X = \widehat{\Aff}^d/G$ be a realisation of $X$ as a quotient singularity. 
Then, $G^{\ab} = \balpha_{2^e}$ for some $e \geq 1$ or $G^{\ab} = \bL_{2,e}[V-F^{e-1}]$ 
for some $e \geq 2$.
Conversely, each of these group schemes appears as $G$ for some RDP 
of type $D_n^r$ with $4r < n$.
\end{Theorem}

 \begin{table}[!h]
 $$
 \begin{array}{|c|c|c|c|}
 \hline
 \text{char.} & \text{RDP}\, X & \text{Quot. sing.?} & G \\ \hline \hline 
\begin{array}{c}
 \mbox{all} \\
 > 2 \\
 > 3 \\
 > 3 \\
 > 5
 \end{array}&
 \begin{array}{c}
 A_n \\
 D_n \\
 E_6 \\
 E_7 \\
 E_8
 \end{array} &
 \begin{array}{c}
 \checkmark \\
 \checkmark \\
 \checkmark \\
 \checkmark \\
 \checkmark 
 \end{array} & 
  \begin{array}{c}
 \bmu_{n+1} \\
 \BD_{4(n-2)} \\
 \BT_{24} \\
 \BO_{48} \\
 \BI_{120} 
 \end{array}
  \\ 
 \hline
 
5 & 
 \begin{array}{c}
 E_8^1 \\
 E_8^0
 \end{array}
 &
  \begin{array}{c}
 \checkmark \\
 \checkmark
 \end{array}
 &
  \begin{array}{c}
 \C_5 \\
 \balpha_5 
 \end{array}
 \\ \hline
 
3 &
\begin{array}{c}
E_6^1 \\
E_6^0 \\
E_7^1 \\
E_7^0 \\
E_8^2 \\
E_8^1 \\
E_8^0
\end{array}
&
\begin{array}{c}
\checkmark \\
\checkmark \\
\checkmark \\
\checkmark \\
\checkmark \\
\checkmark \\
\times
\end{array}
&
\begin{array}{c}
\C_3 \\
\balpha_3 \\
\C_6 \\
\balpha_3 \times \C_2 \\
\BT_{24} \\
\bM_2 \\
\times
\end{array}
\\ \hline

  2
  &
  \begin{array}{c}
D_n^{r} (4r > n) \\
D_n^{r} (4r = n) \\
D_{n}^{r} (2 < 4r < n) \\
D_{4n}^0 \\
D_{4n+2}^0 \\
D_{2n+1}^{1/2} \\
E_6^1 \\
E_6^0 \\
E_7^3 \\
E_7^2 \\
E_7^1 \\
E_7^0 \\
E_8^4 \\
E_8^3 \\
E_8^2 \\
E_8^1 \\
E_8^0 \\
  \end{array}
 &
   \begin{array}{c}
\times \\
\checkmark \\
? \\
\checkmark \\
 \times \\
  \times \\
\checkmark \\
\checkmark \\
\checkmark \\
\times \\
\times \\
\times \\
\checkmark \\
\times \\
\checkmark \\
\times \\
\checkmark
  \end{array}
   &
   \begin{array}{c}
  \times \\
  \C_2 \\
  ? \\
  \balpha_2 \\
  \times \\
  \times \\
  \C_6 \\
  \balpha_2 \times \C_3 \\
  \C_4 \\
  \times \\
  \times \\
  \times \\
  \Dic_{12} \\
  \times \\
  \C_2 \\
  \times \\
  \balpha_2 
  \end{array}
  \\ \hline
  \end{array}
   $$
 \caption{Quotient- and non-quotient RDPs}
 \label{table: Quotient- and non-quotient RDPs}
 \end{table}

\begin{Remark}
Note that Theorem \ref{thm: quotient} does not imply that given two realisations of $X$ 
as a quotient singularity by $G$, the two actions of $G$ on $\widehat{\Aff}^d$ are conjugate. 
If $G$ is linearly reductive, this is true by \cite[Theorem 8.1]{LRQ}. 
The main problem in the non-linearly reductive case is that the $G$-action cannot be 
linearised, making it very hard to find ``standard'' representatives in a 
conjugacy class of $G$-actions.
\end{Remark}

\subsection{Technical preliminaries}
In order to prove Theorem \ref{thm: quotient} and Theorem \ref{thm: not quotient}, 
we need explicit descriptions of local $G$-torsors over RDPs, 
at least in the cases where $G$ is not too complicated. 
In this section, we collect results of this kind, as well as some criteria 
that we can use to prove that a given RDP is \emph{not} a quotient singularity.

\subsubsection{The universal \'etale local torsor}
First, we observe that $\pietloc(X)$ contributes to every quotient realisation of a quotient singularity $x \in X$:

\begin{Proposition}
\label{prop: etale quotients unique}
Let $x\in X$ be a quotient singularity over $k$ with a quotient presentation
$$
 \pi \,:\, \widehat{\Aff}^d \,\to\, X\,=\,\widehat{\Aff}^d/G
$$
with respect to a very small action of a finite group scheme $G$ over $k$ on $\widehat{\Aff}^d$. 
Then, the following hold:
\begin{enumerate}
    \item $\pi$ factors through the integral closure $\pi^{\et}: Y \to X$ of the universal \'etale cover of $U$.
    \item If $Y$ is regular, then $\pi = \pi^{\et}$ and $G= \pietloc(X)$.
\end{enumerate}
\end{Proposition}

\begin{proof}
Since $(\widehat{\Aff}^d - \{0\})/G^\circ$ is \'etale simply connected and \'etale over $U$, the morphism $\pi^{\et}: Y := \widehat{\Aff}^d/G^\circ \to X$ is the integral closure of the universal \'etale cover of $U$. This shows Claim (1).

Now, if $Y$ is regular, then $\widehat{\Aff}^d \to Y$ is a torsor by purity, hence $G^\circ$ is trivial by Lemma \ref{lem: quotient induces local torsor}. Thus, $G$ is \'etale and $\pi = \pi^{\et}$, hence $G = \pietloc(X)$.
\end{proof}

In particular, $G \cong G^{\circ} \rtimes \pietloc(X)$. 
If $G^{\circ}$ is linearly reductive, then we can determine the semidirect product by 
looking at the action of $\pietloc(X)$ on $\mathrm{Cl}(X)$. 
If $G^{\circ}$ is not linearly reductive, then a first step to determining the 
semidirect product structure is to look at the action of $\pietloc(X)$ on
$H^d_{\idealm}(X, \OO)[\idealm]$. 
To determine this action, we use the following lemma. 
Here, for a module $M$ over a local ring $(R,\idealm)$, 
we denote by $M[\idealm]$ the submodule annihilated by the maximal
ideal $\idealm$.

\begin{Lemma} \label{lem: comparison of automorphisms}
 Let $x\in X = \Spec k[[x_1, \dots, x_{d+1}]/(f)$ be an isolated hypersurface singularity in 
 characteristic $p\geq0$.
 Suppose that a finite group $G$ acts on $X$.
 Let 
 $$
 \begin{array}{cccclcl}
 \chi &\colon& G &\to& \Aut(\Gamma(U, \Omega^d) \otimes k) &=& k^{\times}\\
 \psi &\colon& G &\to& \Aut(H^d_{\idealm}(X, \OO)[\idealm]) &=& k^{\times}
 \end{array}
 $$
 be the induced $1$-dimensional representations, where $U \subseteq X$ denotes as usual 
 the complement of the closed point.
 Then, $\chi = \psi^{-1}$.
\end{Lemma}

\begin{proof}
Let $g \in G$ be an element. 
To establish the lemma, we may assume that the order of $g$ is either prime to $p$ 
or a power of $p$. 
In the latter case, its actions on $1$-dimensional vector spaces are trivial. 
Assume the former. 
Then, we can lift the action to $k[[x_1, \dots, x_{d+1}]]$ and diagonalise it, 
so that $g(x_i) = a_i x_i$ and $g(f) = b f$ for some $a_i, b \in k^{\times}$.

The space $\Gamma(U, \Omega^d)$ is generated by a $d$-form $\omega$ satisfying 
$f_{x_i} \omega = (-1)^i dx_1 \wedge \dots \wedge \widehat{dx_i} \wedge \dots \wedge dx_{d+1}$,
hence $\chi(g) = b^{-1} \cdot \prod_i a_i$.
For the second space, write $f = \sum x_i h_i$ for some $h_i \in k[[x_1, \dots, x_{d+1}]]$ satisfying $g(h_i) = a_i^{-1} b h_i$.
Then, $H^d_{\idealm}(X, \OO)[\idealm] \cong H^{d-1}(U, \OO)[\idealm]$ is generated by 
the $(d-1)$-\v{C}ech cocycle $((-1)^i h_i / \prod_{j \neq i} x_j)$ with respect to the open 
covering $(D(x_i))_{i = 1}^{d+1}$ of $U$. 
Therefore, we have $\psi(g) = b \cdot (\prod_i a_i)^{-1}$, hence $\chi = \psi^{-1}$.
\end{proof}

\subsubsection{Diagonalisable torsors}
Next, we recall that, in general, $\bmu_{p^e}$-actions correspond to $\ZZ/p^e\ZZ$-gradings. 
In the following, we explain when such a grading corresponds to a local torsor.

\begin{Lemma} \label{lem: mu_p-torsor} 
Suppose $\Spec S \to \Spec R$ is a local $\bmu_{p^e}$-torsor of primitive class.
Then, the $\bmu_{p^e}$-action on $S$ corresponds to a $\ZZ/p^e\ZZ$-grading, 
that is, a decomposition 
$$
 S \,=\, \bigoplus_{i \,\in\, \ZZ/p^e\ZZ} S_i
$$
of $k$-vector spaces that satisfies
\begin{enumerate}
    \item $1 \,\in\, S_0$,
    \item $S_i \cdot S_j \subseteq S_{i+j}$,
    \item $S_0 = R$, and
    \item the ideal generated by 
$\bigoplus_{i \in \ZZ/p^e\ZZ, i \neq 0} S_i$ is $\idealn$-primary.
\end{enumerate}
Conversely, any local $R$-algebra $S$ with a $\ZZ/p^e\ZZ$-grading of the above form 
is a $\bmu_{p^e}$-torsor of primitive class over $\Spec R$.
\end{Lemma}

\begin{proof}
This is proved in \cite[Section 2.2 and Proposition 2.8]{Matsumotomu}. 
\end{proof}

More generally, local $\bmu_{p^{e_1}} \times \dots \times \bmu_{p^{e_m}}$-torsors 
are described by using
$(\ZZ/p^{e_1}\ZZ) \times \dots \times (\ZZ/p^{e_m}\ZZ)$-gradings.

\subsubsection{Explicit $\balpha_p$-torsors and quotients}
Having discussed \'etale and linearly reductive torsors, we now turn to an explicit description of local $\balpha_p$-torsors.
The key result is Lemma \ref{lem: description of local alpha_p-torsor} below, 
for which we need the following easy observation.

\begin{Lemma} \label{lem: alpha_p-torsor}
Let $R$ be a $k$-algebra.
If $Y = \Spec S \to X = \Spec R$ is an $\balpha_p$-torsor
and $D$ is the corresponding derivation on $S$,
then there exists $t \in S$ such that $D(t) = 1$ and $S = R[t] / (t^p - f)$ for some $f \in R$. 
\end{Lemma}

\begin{proof}
Since $D$ defines an $\balpha_p$-torsor structure, the image of 
$D^{p-1} \colon \OO_Y \to \OO_X$ is equal to $\OO_X$. 
Since $X$ is affine, this means that $D^{p-1} \colon S \to R$ is surjective. 
In particular, there exists $t \in S$ such that $D(t) = 1$.
Since $D(t^p) = 0$, we have $f := t^p \in R$. 
We observe that $R[T] / (T^p - f) \to S \colon T \to t$ is injective, 
and by comparing the rank we observe that it is an isomorphism.
\end{proof}

If $x \in X = \Spec R$ is a normal surface singularity and $x_1,x_2 \in R$ is a regular sequence such that $U = X \setminus \{x\} = \Spec R[x_1^{-1}] \cup \Spec R[x_2^{-1}]$, then 
\begin{equation}
\label{eq: local cohomology group}
R[(x_1 x_2)^{-1}] / \left( R[x_1^{-1}] + R[x_2^{-1}] \right) 
\,\cong\, H^2_{\idealm}(X, \OO_X)
\end{equation}
as seen in the proof of Proposition \ref{prop: torsors loclocloc new}. 
An element $e \in R[(x_1 x_2)^{-1}]$ inducing a class in $H^2_{\idealm}(X, \OO_X)[F]$ yields a local $\balpha_p$-torsor and we can use $e$ to give an explicit 
description of this local torsor. 

\begin{Lemma} \label{lem: description of local alpha_p-torsor}
Let $x \in X = \Spec R$ be a normal surface singularity and let
$\pi \colon Y \to X$ a local $\balpha_p$-torsor of primitive class.
Let $x_1, x_2 \in R$ be a regular sequence such that 
$U = X \setminus \{x\} =  \Spec R[x_1^{-1}] \cup \Spec R[x_2^{-1}]$, 
let $e \in R[(x_1 x_2)^{-1}]$ be an element representing
$$
[Y \to X] \quad\in\quad \overline{\Hfl{1}(U, \balpha_p)}  
\,\cong\, H^2_{\idealm}(X, \OO_X)[F] 
$$ 
via the isomorphism \eqref{eq: local cohomology group}.
Let $e_i \in R[x_i^{-1}]$ be elements satisfying $e^p = e_1 - e_2$. 
Then, the following hold:
\begin{enumerate}
    \item \label{lem: description of local alpha_p-torsor:local}
There exist $t_i \in \OO_{Y_U}[x_i^{-1}]$ and $b \in R$ such that 
$$
\OO_{Y_U}[x_i^{-1}] \,\cong\, R[x_i^{-1}][t_i] / (t_i^p - (e_i + b))
$$
as $\balpha_p$-torsors over $R[x_i^{-1}]$, where the torsor structure on the right-hand side is given by the derivation $d/dt_i$.
\item \label{lem: description of local alpha_p-torsor:global}
If $S \subseteq \Frac \OO_{Y_U}$ is an $R$-subalgebra satisfying
\begin{enumerate}
    \item $\OO_{Y_U}[x_i^{-1}] = S[x_i^{-1}]$, and \label{item: equal over U}
    \item $S$ satisfies $S_2$, \label{item: S is S_2}
\end{enumerate}
then $Y = \Spec S$.
\item \label{lem: description of local alpha_p-torsor:u_i}
If $n_1, n_2$  are non-negative integers such that $e = {x_1^{-n_1} x_2^{-n_2}}{h}$ with $h \in R$,
    and $u_i = x_i^{n_i} t_i$, then $S' := R[u_1, u_2]$ satisfies (\ref{item: equal over U}).
\end{enumerate}
\end{Lemma}
\begin{proof}
Recall that $Y$ is affine and satisfies $S_2$, by Lemma \ref{lem: carvajal-rojas}.
Since $Y$ does not extend to a global $\balpha_p$-torsor, $Y$ is a domain.

(\ref{lem: description of local alpha_p-torsor:local})
Since $Y \times_X \Spec R[x_i^{-1}] \to \Spec R[x_i^{-1}]$ is an $\balpha_p$-torsor over an affine scheme, by Lemma \ref{lem: alpha_p-torsor} it is of the form  
$\Spec R[x_i^{-1}][t_i] / (t_i^p - f_i)$
for some $f_i \in R[x_i^{-1}]$.
The class $[e]$ of the local $\balpha_p$-torsor is defined so that $e = t_1 - t_2$ in $R[(x_1 x_2)^{-1}] / \left( R[x_1^{-1}] + R[x_2^{-1}] \right)$.
By replacing $t_i$ with $t_i + a_i$ for appropriate $a_i \in R[x_i^{-1}]$, we may assume $e = t_1 - t_2$ in $R[(x_1 x_2)^{-1}]$,
and then we have $e_1 - e_2 = e^p = f_1 - f_2$.
Hence $b := f_1 - e_1 = f_2 - e_2$ belongs to $R[x_1^{-1}] \cap R[x_2^{-1}] = R$.

(\ref{lem: description of local alpha_p-torsor:global})
We have $\OO_S = \OO_Y$ over $U$ by (\ref{item: equal over U}). 
By taking the sections over $U$ of these sheaves we obtain $S = \Gamma(Y, \OO_Y)$
since both $S$ and $Y$ satisfy $S_2$ by (\ref{item: S is S_2}) and Lemma \ref{lem: carvajal-rojas} respectively,

(\ref{lem: description of local alpha_p-torsor:u_i})
    For (\ref{item: equal over U}), $t_i = x_i^{-n_i} u_i \in S'[x_i^{-1}]$ and $u_i = x_i^{n_i} t_i \in R[x_i^{-1}][t_i] / (\dots)$ are clear, and we have $u_1 = x_1^{n_1} t_1 = x_1^{n_1} t_2 + {x_2^{-n_2}}{h} \in R[x_2^{-1}][t_2] / (\dots)$.
\end{proof}

Occasionally, we need to show that certain singularities are not $\balpha_2$-quotient singularities. 
For RDPs, we can use the classification of $\balpha_p$-quotient RDPs given in \cite[Lemma 3.6]{Matsumotoalpha}. 
However, we also have to deal with some non-RDP singularities, in which case we use the following lemma.
Recall that the Tyurina ideal $I$ of $R = k[[x_1,x_2,x_3]] / (h)$ is defined to be the ideal 
$I = (h, h_{x_1}, h_{x_2}, h_{x_3}) \subseteq k[[x_1, x_2, x_3]]$ generated by $h$ 
and its partial derivatives.

\begin{Lemma} \label{lem: alpha_2 Tyurina}
Suppose $R = k[[x_1,x_2,x_3]] / (h)$ is an $\balpha_2$-quotient singularity. 
Then the Tyurina ideal $I$ is the Frobenius power $\ideala^{(2)}$ of some ideal $\ideala$. 
Moreover, for any ideal $\idealb$, $I + \idealb^{(2)}$ is also the Frobenius power 
of some ideal and hence, it is stable under taking partial derivatives. 
\end{Lemma}
\begin{proof}
The former part follows from we use the description of $\balpha_2$-quotients given in \cite[Theorem 3.8]{Matsumotoalpha} (with $T = 0$).
The latter part follows from the former part.
\end{proof}

\subsubsection{Criteria for uniqueness and non-existence of quotient presentations} 

Our main tool to show that certain RDPs cannot be quotient singularities 
will be the following, which is an immediate consequence of 
Lemma \ref{lem: quotient induces local torsor}:

\begin{Proposition} \label{prop: primitive class}
Let $x \in X$ be a quotient singularity by a finite group scheme $G$ and let 
$G' = G/N$ be an abelian quotient of $G$.
Then, 
$$
[\widehat{\Aff}^d/N \to X] \quad\in\quad 
\overline{\Hfl{1}(U, G')}_{\prim}.
$$
\end{Proposition}

\begin{proof}
For every $H \subsetneq G'$, the morphism $(\widehat{\Aff}^d/N)/H \to X$ is a 
local $G/H$-torsor that does not extend to a global $G/H$-torsor by the same argument as in 
Lemma \ref{lem: quotient induces local torsor} (\ref{item: quotient induces local torsor}). 
Hence, the class in $\overline{\Hfl{1}(U, G')}$ of the local $G'$-torsor 
$\widehat{\Aff}^d/N \to X$ is non-zero in $\overline{\Hfl{1}(U, G'/H)}$, 
which means precisely that the class does not come from $\overline{\Hfl{1}(U, H)}$ (here, we use left-exactness of the functor $\overline{\Hfl{1}(U,-)}$).
\end{proof}

\begin{Corollary} \label{cor: no primitive class}
Let $x \in X$ be a singularity and let $G'$ be an abelian group scheme with $\overline{H^1(U, G')}_{\prim} = \emptyset$. Then, $x \in X$ is not a quotient singularity by any finite group scheme $G$ having $G'$ as a quotient.
\end{Corollary}

Here is an application of these results to narrow 
down the possibilities for a given singularity 
to be a quotient singularity.

\begin{Lemma} \label{lem: possible argument}
Assume that $x \in X$ is a normal singularity that is not F-regular and with trivial
$\pietloc(X)$.
If $\overline{H^1(U, \balpha_{p^2})}_{\prim} = \overline{H^1(U, \bM_2)}_{\prim} = \emptyset$ and if $X$ is a quotient singularity, then it is a quotient singularity by $\balpha_p$. 
In particular, if $x\in X$ is not a quotient singularity 
by $\balpha_p$, then $X$ is no quotient singularity at all.
\end{Lemma}

\begin{proof}
If $x\in X$ were a quotient singularity by a very 
small action of a finite group scheme $G$ 
of height $n \geq 2$, then $G$ cannot be linearly reductive, as $X$ is not F-regular. 
Hence, we have 
$G/G[F^{n-2}] \in \{ \balpha_{p^2}, \bM_2\}$ 
by Proposition \ref{prop:  simple infinitesimal}, contradicting 
Corollary \ref{cor: no primitive class}.
\end{proof}

As an application of Lemma \ref{lem: comparison of automorphisms}, we can extend the uniqueness result in Lemma \ref{lem: possible argument} to extensions of \'etale group schemes by $\balpha_p$.   

\begin{Lemma} \label{lem: possible argument uniqueness}
 Let $x \in X$ be an RDP that is a quotient singularity 
 $X\cong\widehat{\Aff}^2/G$.
 Assume that the universal \'etale local torsor 
 $y \in Y$ over $X$ is an RDP that is a quotient singularity by $\balpha_p$ (and only by $\balpha_p$) and satisfies $\dim H^2_{\idealn}(\OO_Y)[F] = 1$,
 where $\idealn$ is the maximal ideal at $y$.
 Then, $G \cong \balpha_p \times \pietloc(X)$.
\end{Lemma}

\begin{proof}
Let $Z := \widehat{\Aff}^2 = \Spec k[[x_1,x_2]]$. 
By Proposition \ref{prop: etale quotients unique}, 
we have $G^{\et} = \pietloc(X)$, so $G$ is a semidirect
product $\balpha_p \rtimes \pietloc(X)$ corresponding to a
homomorphism 
$\phi \colon \pietloc(X) \to \Aut(\balpha_p) = k^{\times}$. 
It remains to show that $\phi$ is trivial. 
Since $X$ and $Y$ are RDPs, the action of 
$\pietloc(X)$ on $H^0(Y \setminus \idealn, \Omega_Y^2) \otimes k$ is trivial.
Thus, by Lemma \ref{lem: comparison of automorphisms}, 
also the action of $\pietloc(X)$ on $H^2_{\idealn}(\OO_Y)[\idealn] = H^2_{\idealn}(\OO_Y)[F]$
is trivial.

Let $x_1,x_2$ be a $\pietloc(X)$-regular sequence on $Y$ and let $D$ be the derivation corresponding to the $\balpha_p$-action on $Z$ (see \ref{ex: integratevectorfield}).
Write $U_i \subseteq Y$ for the open subscheme where $x_i$ is invertible and write $\pi$ for the morphism $Z \to Y$. 
Since $\pi|_{U_i}$ is an $\balpha_p$-torsor and $U_i$ is affine, by Lemma \ref{lem: alpha_p-torsor} we find $t_i \in \OO_Z[\frac{1}{x_i}]$ such that $D(t_i) = 1$ and such that $\OO_Z[\frac{1}{x_i}] = \OO_Y[\frac{1}{x_i},t_i]$.


Then, we have $D(t_1 - t_2) = 0$, hence $e := [t_1 - t_2] \in H^1(U,\OO_U)$. In fact, since $Z|_{U_i} \to U_i$ is obtained by restricting the $\alpha_p$-torsor $Z|_U \to U$, the class $e^p = t_1^p - t_2^p$ is trivial, hence $e \in H^2_{\idealn}(\OO_Y)[F]$. Under the isomorphism of Proposition \ref{prop: torslocloc}, $e$ corresponds to the class of the local torsor $\pi$, hence $e$ is non-trivial. For all $g \in \pietloc(X)$, we have the relation $g^{-1} D g = \phi(g^{-1}) \cdot D$. Plugging in $t_i$ and using that $g$ acts trivially on $k$, we get $D(g(t_i)) ) = \phi(g^{-1})$,
hence $D(g(t_i) - \phi(g^{-1}) \cdot t_i) = 0$, and so $g(t_i) \equiv \phi(g^{-1}) \cdot  t_i \pmod{\OO_Y[\frac{1}{x_i}]}$.
Hence, we have $g(e) = \phi(g^{-1}) \cdot e$. But $g$ acts trivially on $H^2_{\idealn}(\OO_Y)[F]$, hence $e = \phi(g^{-1}) e$, and so $\phi$ is trivial, which is what we wanted to show.
\end{proof}

\subsection{Strategy of proof} \label{subsec: quotient preliminary}
To prove Theorem \ref{thm: quotient} and Theorem \ref{thm: not quotient} for a given RDP 
$x \in X$, we take the following naive approach:
\begin{enumerate}
    \item Try to give an explicit realisation of $X$ as a quotient singularity by a finite group scheme $G$.
    \item If (1) works, show uniqueness of $G$.
    \item If (1) does not work, show that $X$ is no quotient singularity.
\end{enumerate}

Every RDP $x \in X$ is contained in at least one of the following four classes, 
depending on the properties of the universal \'etale local torsor $Y$ over $X$:
\begin{enumerate}
    \item[(A)] $Y$ is regular,
    \item[(B)] $Y$ is F-regular and $p \nmid |\pietloc(X)|$,
    \item[(C)] $Y$ is F-regular and $p \mid |\pietloc(X)|$,
    \item[(D)] $Y$ is not F-regular.
\end{enumerate}
Clearly, the classes (B), (C), and (D) are disjoint and every element of (A) is contained
in either (B) or (C).
The RDPs of class (A) have been classified by Artin \cite{ArtinRDP} and
RDPs in class (B) have been studied in \cite{LRQ}. 
We note that the RDPs in class (B) are precisely the F-regular RDPs.
The following statement follows from \cite{ArtinRDP} and 
\cite[Theorem 8.1, Theorem 11.1, Theorem 11.2]{LRQ}.

\begin{Theorem} \label{thm: quotient class A and B}
Let $x \in X$ be an RDP and let $Y$ be the universal \'etale 
local torsor over $X$. 
If $X$ is F-regular or $Y$ is regular, 
then Theorem \ref{thm: quotient} and Theorem \ref{thm: not quotient} hold for $X$. 
In particular, they hold in characteristic $\geq 7$.
\end{Theorem}

Thus, it remains to study RDPs of class (C) and (D). 
It is natural to study the RDPs with trivial local \'etale fundamental group first.
By Proposition \ref{prop: linearlyreductive or unipotent}, if such an RDP is a quotient singularity, then either by $\bmu_{p^{n}}$, hence it is of type $A_{p^n-1}$, or by an 
iterated extension of $\balpha_p$ by itself with underlying scheme $\Spec ~ k[t]/(t^{p^n})$.
Thus, to understand an RDP $x\in X$ of class (D), we have to understand local torsors over it
under extensions of $\balpha_p$.
This is where our results on $\Picloclocloc_{X/k}$ will be of fundamental importance.

To prove uniqueness of the group scheme $G$ by which a quotient RDP is a quotient singularity, 
we will have to classify the possible total spaces of local $\balpha_p$-torsors over some of RDPs of class (D) with trivial $\pietloc(X)$. For this we use Lemma \ref{lem: description of local alpha_p-torsor}.

In Sections \ref{subsec: rdp 5}--\ref{subsec: rdp 2}, we will prove Theorem \ref{thm: quotient} and 
Theorem \ref{thm: not quotient} in characteristics $5$, $3$, and $2$, respectively.

In Section \ref{subsec: dnr} we prove Theorem \ref{thm: D_n in p=2}.


\subsection{Characteristic 5} \label{subsec: rdp 5}
Assume that $k$ is an algebraically closed field of characteristic $5$. 
By Table \ref{table: RDPEquations} and Theorem \ref{thm: quotient class A and B}, 
it suffices to study the $E_8$-singularities. 
The $E_8^1$-singularity has a regular universal \'etale local torsor by \cite{ArtinRDP}
and thus, Proposition \ref{prop: etale quotients unique} applies.
Hence, again by Theorem \ref{thm: quotient class A and B}, we only have to study 
the $E_8^0$-singularity.

\begin{Lemma} \label{lem: main char 5}
Theorem \ref{thm: quotient} and Theorem \ref{thm: not quotient} hold for the 
$E_8^0$-singularity in characteristic $5$. 
In particular, $E_8^0$ is a quotient singularity.
\end{Lemma}

\begin{proof}
Let $x \in X$ be a singularity of type $E_8^0$. 
By \cite[Lemma 3.6]{Matsumotoalpha}, it is a quotient singularity by $\balpha_5$. 

Since we know that $\overline{H^1(U, G')}_{\prim} = \emptyset$ for 
$G' \in \{ \balpha_{25}, \bM_2 \}$ by 
Proposition \ref{prop: torsors loclocloc new} (Table \ref{table: torsors with Picloclocloc}), Lemma \ref{lem: possible argument} shows that $\balpha_5$ 
is also the only possible finite group scheme that can realise $x\in X$ 
as a quotient singularity. 
\end{proof}

\subsection{Characteristic 3} \label{subsec: rdp 3}
Assume that $k$ is an algebraically closed field of characteristic $3$. 
By Table \ref{table: RDPEquations} and Theorem \ref{thm: quotient class A and B}, 
it suffices to study the $E_n$-singularities.  
By \cite{ArtinRDP}, the RDPs of type $E_6^1$, $E_7^1$, and $E_8^2$ have a regular 
universal \'etale local torsor, hence we are left with the four RDPs of type 
$E_6^0, E_7^0, E_8^1,$ and $E_8^0$.

We need the following lemma on $\balpha_3$-torsors over some of these RDPs.
What makes this result non-trivial is that, a priori, a class in $\overline{H^1(U,\balpha_3)}$
could be representable by many torsors with non-isomorphic total spaces.
We will see examples of this phenomenon in Proposition \ref{prop: non-uniqueness of torsors} below.

\begin{Lemma} \label{lem: alphap-torsor char 3}
Let $x \in X$ be an RDP of type $E_6^0$ (resp. $E_8^1$) and let $Y \to X$ be a 
local $\balpha_3$-torsor of primitive class over $X$. 
Then, $Y$ is of type $A_0$ (resp. $E_6^0$).
\end{Lemma}

\begin{proof}
We use Lemma \ref{lem: description of local alpha_p-torsor}.
Write $R = k[[x,y,z]] / (f)$ with $f = z^2 + x^3 + y^3 g$ 
and $g = y$ (resp. $g = x^2 + y^2$) for $E_6^0$ 
(resp. $E_8^1$).
We use the regular sequence $(x_1, x_2) = (y, z)$.
As $\dim H^2_{\idealm}(\OO_X)[F] = 1$, the local $\balpha_3$-torsor $Y_U \to U$ corresponds to the 
class $a [y^{-1} z^{-1} x^2]$ (note that $[y^{-1} z^{-1} x^2]$ is equal to the class $f_1$ that appears in Proposition \ref{prop: torsors loclocloc new}, but with respect to a different affine cover of $U$) for some $a \in k^\times$, and, as explained in Example \ref{ex: endomorphismsofalphap}, we may assume $a = 1$.
We have 
$$
 \begin{array}{lcll}
  \OO_{Y_U}[y^{-1}] &=& R[y^{-1}][t_1] \,/\, (t_1^3 - (y^{-3} z + b)) & \mbox{ \quad and }\\
  \OO_{Y_U}[z^{-1}] &=& R[z^{-1}][t_2] \,/\, (t_2^3 - (z^{-3}g(z^2 - y^3 g) + b))
 \end{array}
$$
for some $b \in R$ and $t_1 - t_2 = y^{-1} z^{-1} x^2$.
Let $u_1 = y t_1$, $u_2 = z t_2$, $\xi = -y^{-1}({x + u_1^2}) = z^{-2}(z u_1 u_2 + y u_2^2 + x y^2 g)$, and set 
$S = R[\xi, u_1, u_2]$.
Then, $u_1, u_2, \xi$ are integral over $R$ since 
$\xi^3 = g + b z - b^2 y^3 \in R$.
Moreover, we have $\OO_{Y_U}[y^{-1}] \supset S[y^{-1}]$ and $\OO_{Y_U}[z^{-1}] \supset S[z^{-1}]$.
Also, $z = u_1^3 - b y^3$, $x = - u_1^2 - \xi y$, and 
$u_2 = y^{-1} (z u_1 - x^2) = y^{-1}((u_1^3 - b y^3) u_1 - (y \xi + u_1^2)^2)
= \xi u_1^2 - y \xi^2 - b y^2 u_1$.
Therefore, $S = R[u_1, u_2, \xi]$ is equal to $k[[y, u_1, \xi]] / (h)$ with
$h = -\xi^3 + g + b u_1^3 + b^2 y^3$.
In particular $S$, being a hypersurface, satisfies $S_2$, and hence by Lemma \ref{lem: description of local alpha_p-torsor} we have $Y = \Spec S$.
\begin{enumerate}
\item If $g = y$, then $h = y + [b^2y^3 + bu_1^3 - \xi^3]$.
\item If $g = x^2 + y^2$, we have $h = -\xi^3 + y^2 + u_1^4 + [b u_1^3 + b^2 y^3 + y^2 \xi^2 + 2 u_1^2 y \xi]$,
\end{enumerate}
In each case, the terms in brackets 
do not affect the singularity, hence $\Spec ~ S$ 
is an RDP of type $A_0$ or $E_6^0$ respectively.
\end{proof}

\begin{Lemma} \label{lem: main char 3}
 Theorem \ref{thm: quotient} and Theorem \ref{thm: not quotient} hold for the $E_6^0, E_7^0, E_8^1$, 
 and $E_8^0$-singularity in characteristic $3$. 
 In particular, $E_6^0$, $E_7^0$, and $E_8^1$ are quotient singularities, 
 but $E_8^0$ is not a quotient singularity.
\end{Lemma}
\begin{proof}
First, assume that $x \in X$ is of type $E_6^0$. 
By \cite[Lemma 3.6]{Matsumotoalpha}, $X$ is a 
quotient singularity by $\balpha_3$.
By the same argument as in Lemma \ref{lem: main char 5},
Theorem \ref{thm: quotient} holds for $X$.

Next, assume that $x \in X$ is of type $E_7^0$. 
Then, $X$ is a quotient singularity by 
$\balpha_3 \times \C_2$. 
Indeed, the quotient by the action on $k[[x,y]]$ defined by $(D(x), D(y)) = (y, x^3)$ and 
$(g(x), g(y)) = (-x, -y)$ is $E_7^0$. 
Lemma \ref{lem: possible argument uniqueness} shows that $\balpha_3 \times \C_2$ 
is the only finite group scheme by which $X$ can be a quotient singularity.

Now, assume that $x \in X$ is of type $E_8^1$. 
By Table \ref{table: torsors with Picloclocloc}, there exists a local $\bM_2$-torsor 
$Z \overset{\balpha_3}{\to} Y \overset{\balpha_3}{\to} X$,
such that neither $Z \to Y$ nor $Y \to X$ are $\balpha_3$-torsors. 
Then, by Lemma \ref{lem: alphap-torsor char 3}, 
$Y$ is of type $E_6^0$ and $Z$ is smooth.
To show the uniqueness of the group scheme, suppose that
$X = Z/G$ is a presentation of $X$ as a quotient singularity. 
Since $\pietloc(X)$ is trivial, we may assume that $G$ is of finite height, 
say $n$, and since $X$ is not F-regular, 
we have $G / G[F^{n-1}] = \balpha_3$. 
The total space $Y$ of the local $\balpha_3$-torsor $Y = Z/G[F^{n-1}] \to X$ is of type $E_6^0$ 
by Lemma \ref{lem: alphap-torsor char 3} and hence, $G[F^{n-1}] = \balpha_3$ by the uniqueness of the group scheme for $E_6^0$.
Hence $G$ is either $\balpha_{9}$ or $\bM_2$.
Now, it suffices to note that by 
Table \ref{table: torsors with Picloclocloc}, 
$G$ cannot be $\balpha_{9}$.
 
Finally, assume that $x \in X$ is of type $E_8^0$. 
Assume that $X$ is a quotient singularity by a very small action 
of a finite group scheme $G$. 
Since $\pietloc(X)$ is trivial and $X$ is not F-regular,
$G$ is an iterated extension of $\balpha_3$'s, say of height $n$. 
By \cite[Lemma 3.6]{Matsumotoalpha}, we have 
$G \neq \balpha_3$. 
Thus, by Proposition \ref{prop: simple infinitesimal}, 
we have $G/G[F^{n-2}] \in \{ \balpha_9 , \bM_2\}$. 
But by Table \ref{table: torsors with Picloclocloc}, we have 
$\Picloclocloc_{X/k} \cong \balpha_9$, hence it follows from 
Proposition \ref{prop: torsfinab} that $\overline{H^1(U, \balpha_{p^2})}_{\prim} = \overline{H^1(U, \bM_2)}_{\prim} = \emptyset$. 
Therefore, $X$ is not a quotient singularity by Lemma \ref{lem: possible argument}.
\end{proof}

\subsection{Characteristic 2} \label{subsec: rdp 2}
Assume that $k$ is an algebraically closed field of characteristic $2$. 
By Table \ref{table: RDPEquations} and Theorem \ref{thm: quotient class A and B}, it remains to study the 
$D_n^r, E_6^1, E_6^0, E_7^3, E_7^2, E_7^1, E_7^0, E_8^4, E_8^3, E_8^2, E_8^1$, 
and $E_8^0$-singularities. 
Of these, the $D_n^r$ with $4r = n$, the $E_6^1$, the $E_7^3$, the $E_8^4$, 
and the $E_8^2$-singularity have a regular universal \'etale local torsor by \cite{ArtinRDP}, 
hence we have to study the $D_n^r$ with $4r \neq n$, $E_6^0$, $E_7^2$, $E_7^1$, $E_7^0$, $E_8^3$, $E_8^1$, 
and $E_8^0$-singularity.

First, we study those of the remaining RDPs, whose universal \'etale local torsor is F-regular. 
By \cite{ArtinRDP}, these are precisely the $D_n^r$ 
with $4r > n$. 
It turns out that these are no quotient singularities. 
The key observation is the following.

\begin{Lemma}\label{Lemma: quotientsinchar2}
Let $k$ be an algebraically closed field of characteristic $2$. 
Then, $\bmu_2 \times \C_2$ does not admit a very small action on $\widehat{\Aff}^2$.
\end{Lemma}
 
\begin{proof}
Let $G = \bmu_2 \times \C_2$ and assume that $G$ acts freely outside the closed point on $\Spec R$ 
where $R = k[[u,v]]$. 
We can linearise the $\bmu_2$-action so that it acts 
via the diagonal action $u \mapsto \epsilon u, v \mapsto \epsilon v$ with $\epsilon^2 = 1$. 
Since $\C_2$ acts freely outside the closed point, it acts as the identity on the tangent space at $0$ by \cite{ArtinWild}, that is, the group $\C_2$ acts as 
$u \mapsto u + a, v \mapsto v + b$, 
where $a,b \in k[[u,v]]$ have degree at least $2$. 
By assumption, the actions of $\bmu_2$ and $\C_2$ commute and thus,
all monomials appearing in $a$ and $b$ have odd degree.
Now, by \cite{ArtinWild}, we have
$$
R^{\C_2} \,=\, k[[u^2 + au, v^2 + bv, bu + av]] 
\,\subseteq\, k[[u,v]].
$$
In particular, if $f \in R^{\C_2}$, then all monomials appearing in $f$ have even degree and hence $\bmu_2$ acts trivially on $R^{\C_2}$. This contradicts our assumption that $G$ acts freely outside the closed point.
\end{proof} 
 
\begin{Corollary} \label{cor: non-unipotent D}
Theorem \ref{thm: quotient} and Theorem \ref{thm: not quotient} hold for the 
$D_n^r$-singularities with $4r > n$
in characteristic $2$.
In particular, they are not quotient singularities.
\end{Corollary}
 
\begin{proof}
Seeking a contradiction, assume that
$X = Z/G$,  where the finite group scheme $G$ acts on
$Z = \widehat{\Aff}^2$ via a very small action.
We have $G/G^\circ \cong \pietloc(X)$ by Proposition \ref{prop: etale quotients unique} and 
$\pi \colon Z \to X$ factors through the universal
\'etale local torsor over $X$. 
In particular, by \cite{ArtinRDP}, there are morphisms 
$f \colon Z \to Y$ and $g \colon Y \to X$, such that
$Y$ is of type $A_{8r-2n-1}$ and 
$g$ is a local $\C_2$-torsor. 
Note that $Y$ is F-regular and thus $f$ is a local $\bmu_{8r-2n}$-torsor 
by Table \ref{table: Quotient- and non-quotient RDPs}. 
Therefore, the group scheme $G$ sits in a short exact sequence
$$
0 \,\to\, \bmu_{8r-2n} \,\to\, G \,\to\, \C_2 \,\to\, 0
$$
In particular, $G$ admits a subgroup scheme isomorphic 
to $\bmu_2 \times \C_2$. 
But by Lemma \ref{Lemma: quotientsinchar2}, 
the action of $\bmu_2 \times \C_2$ on $Z$ cannot be  
free outside the closed point, 
contradicting our assumption. 
Hence, $X$ is not a quotient singularity.
\end{proof}

This leaves us with $D_n^r$ with $4r < n$, $E_6^0$, $E_7^2$, $E_7^1$, $E_7^0$, $E_8^3$, $E_8^1$, and $E_8^0$.
Similar to Lemma \ref{lem: alphap-torsor char 3}
in characteristic $3$, we need some 
information on $\balpha_2$-torsors over some of these 
in order to proceed.

\begin{Lemma} \label{lem: alphap-torsor char 2}
Let $x \in X$ be an RDP of type $E_7^2, E_7^1$, or $E_8^3$
and let $Y \to X$ be a local $\balpha_2$-torsor of primitive class over $X$. 
Then, $Y$ is of type $D_5^{1/2}$, $D_7^{1/2}$, or $E_7^2$, respectively. 
\end{Lemma}

\begin{proof}
%
We use Lemma \ref{lem: description of local alpha_p-torsor}.
Write $R = k[[x,y,z]] / (f)$ with $f = z^2 + x^3 + y^2 g$ and with 
$g = (x+z)y$, $(x+yz)y$, $(y^2 + z)y$ for $E_7^2$, $E_7^1$, and $E_8^3$, respectively.
(In the $E_7^1$ case, we apply the coordinate change $x \mapsto x + yz$ and neglect terms of high degree to obtain the equation in Table \ref{table: RDPEquations}.)
We use the regular sequence $(x_1, x_2) = (x, y)$.
Since $\dim H^2_{\idealm}(\OO_X)[F] = 1$ by Table \ref{table: torsors with Picloclocloc},
the $\balpha_2$-torsor $Y_U \to U$ corresponds to the class 
$a [x^{-1} y^{-1} z]$ for some $a \in k^\times$. Again we may assume $a = 1$.
We have 
$$
\begin{array}{lcll}
 \OO_{Y_U}[x^{-1}] &=& R[x^{-1}][t_1] \,/\, (t_1^2 - (x^{-2}g + b)) &\mbox{ \quad  and}\\
 \OO_{Y_U}[y^{-1}] &=& R[y^{-1}][t_2] \,/\, (t_2^2 - (y^{-2} x + b)) 
\end{array}
$$ 
for some $b \in R$ and $t_1 - t_2 = x^{-1} y^{-1} z$.
Set
$$
 S \,:=\, R[u_1, u_2] \mbox{ \quad  where\quad} u_1 = x t_1 \mbox{  and  } u_2 = y t_2.
$$
Since $x = u_2^2 - b y^2$ and $z = y u_1 - x u_2 = y u_1 + u_2^3 + y^2 u_2 b$, 
we have $S = k[[y,u_1,u_2]] / (h)$, where $h = h(y,u_1,u_2) = u_1^2 + g + x^2 b$.
By Lemma \ref{lem: description of local alpha_p-torsor} we have $Y = \Spec S$.
We have the following three cases depending on the singularity $x\in X$ we started with:
\begin{enumerate}
 \item If $g = (x+z)y$, then 
 
 $h = u_1^2 + u_2^2 y + u_1 y^2 + [u_2^3 y + y^3 b (1 + u_2) + u_2^4 b + y^4 b^3]$.
 \item If $g = (x+yz)y$, then 
 
 $h = u_1^2 + u_2^2 y + u_1 y^3 + [u_2^3 y^2 + y^3 b (1 + u_2 y) + u_2^4 b + y^4 b^3]$.
 \item If $g = (y^2 + z)y$,  then 
 
 $h = u_1^2 + y^3 + y^2 u_1 + y u_2^3 + [u_2 y^3 b + u_2^4 b + y^4 b^3]$.
\end{enumerate}
In each case, the terms in brackets do not affect the singularity, 
and $S$ is an RDP of of type $D_5^{1/2}$, $D_7^{1/2}$, or $E_7^2$, respectively.
\end{proof}

Many of the remaining RDPs turn out not to be quotient singularities. 
In the following, we first treat some that are.

\begin{Lemma}
Theorem \ref{thm: quotient} and Theorem \ref{thm: not quotient} hold for the 
$D_{4n}^0, E_6^0$ and $E_8^0$-singularities in characteristic $2$. 
In particular, all of them are quotient singularities.
\end{Lemma}

\begin{proof}
If $X$ is of type $D_{4n}^0$ or $E_8^0$, then it is a
quotient singularity by $\balpha_2$, 
see \cite[Lemma 3.6]{Matsumotoalpha}. 
Since we know that $\overline{H^1(U, G')}_{\prim} = \emptyset$ 
for $G' \in \{ \balpha_{4}, \bM_2 \}$ by 
Proposition \ref{prop: torsors loclocloc new} (Table \ref{table: torsors with Picloclocloc}), 
we can use Corollary \ref{cor: no primitive class} 
to show that $\balpha_2$ is also the only finite group scheme 
that can realise $X$ as a quotient singularity.

If $x\in X$ is of type $E_6^0$, then it  
is a quotient singularity by $\balpha_2 \times \C_3$. 
Indeed, the quotient by the action on $k[[x,y]]$ defined by $(D(x), D(y)) = (y^2, x^2)$
and $(g(x), g(y)) = (\omega x, \omega^2 y)$, where $\omega$ is a primitive $3$-rd root of $1$, 
is $E_6^0$. 
By Lemma \ref{lem: possible argument uniqueness}, 
$G$ is unique.
\end{proof}

This leaves us with RDPs of type $D_n^r$ with $0 < 4r < n$, $D_{4n+2}^0$, 
$E_7^2$, $E_7^1$, $E_7^0$, $E_8^3$, and $E_8^1$.

\begin{Lemma} \label{lem: nonquotient easycases}
Theorem \ref{thm: quotient} and Theorem \ref{thm: not quotient} hold for the $D_{2n+1}^{1/2}$, 
$D_{4n+2}^0$, $E_7^2$, $E_7^1$, $E_7^0$, and $E_8^3$,
in characteristic $2$. 
In particular, they are not quotient singularities.
\end{Lemma}

\begin{proof}
Let $x \in X$ be one of the RDPs in the statement. 
We know that it is not an $\balpha_2$-quotient singularity 
by \cite[Lemma 3.6]{Matsumotoalpha}.

Now, if $x\in X$ is of type $D_{2n+1}^{1/2}$, $D_{4n+2}^0$, $E_7^1$ or $E_7^0$, 
then it follows from the description of $\Picloclocloc_{X/k}$ 
in Table \ref{table: torsors with Picloclocloc} 
that $\overline{H^1(U, G')}_{\prim} = \emptyset$ for $G' \in  \{ \balpha_{p^2}, \bM_2 \}$.
Hence, by Lemma \ref{lem: possible argument}, these are not quotient singularities.

Next, assume that $x\in X$ is of type $E_7^2$ or $E_8^3$.  
Seeking a contradiction, suppose that $X =  \widehat{\Aff}^2 /G$ is quotient singularity,
and let $\dots \to X_1 \to X$ be a factorisation of $Y \to X$ into local 
$\balpha_2$-torsors.
By Lemma \ref{lem: alphap-torsor char 2}, $X_1$ is a rational double point of type 
$D_5^{1/2}$ or $E_7^2$ respectively. 
By the previous paragraph, $D_5^{1/2}$ is not a quotient singularity.
Hence, we arrive at a contradiction. 
Therefore, $X$ is not a quotient singularity. 
\end{proof}

In the following, we prove Theorem \ref{thm: quotient} and 
Theorem \ref{thm: not quotient} for 
the $E_8^1$-singularity.

\begin{Lemma} \label{lem: E81main}
Theorem \ref{thm: quotient} and Theorem \ref{thm: not quotient} hold for the 
$E_8^1$-singularity in characteristic $2$. 
In particular, it is not a quotient singularity.
\end{Lemma}

\begin{proof}
Seeking a contradiction, suppose that $X = \widehat{\Aff}^2 / G$ 
with respect to a very small action of a finite group scheme $G$.
Since $\piet(X)$ is trivial, $G$ is local, say of height $n$.
Since $X$ is not F-regular, $G$ is not $\bmu_{2^n}$, hence by 
Proposition \ref{prop: linearlyreductive or unipotent}, $G$ is an $n$-fold 
extension of $\balpha_2$'s.
We have $n > 1$ by \cite[Lemma 3.6]{Matsumotoalpha}. 
In order to prove the lemma, it suffices to establish the 
following claims:
\begin{enumerate}
    \item \label{item: not n=2} $X_1 = \widehat{\Aff}^2 / G[F^{n-1}]$ is not an $\balpha_2$-quotient singularity. 
    Hence, $n > 2$, and the quotient $G[F^{n-1}] / G[F^{n-3}]$ 
    is isomorphic to $\balpha_4$ or $\bM_2$.
    \item \label{item: not alpha_4} $\overline{H^1(U_1, \balpha_4)}_{\prim} = \emptyset$, where $U_1 \subseteq X_1$ is the complement of the closed point. 
    Hence $G[F^{n-1}] / G[F^{n-3}] \not\cong \balpha_4$ by Corollary \ref{cor: no primitive class}.
    \item \label{item: not M_2} 
    $\overline{H^1(U, \bM_3)}_{\prim} = \emptyset$. 
    Hence,  $G/G[F^{n-3}] \not\cong \bM_3$ by Corollary \ref{cor: no primitive class}. 
    By Proposition \ref{prop: groupschemesoforderp3}, this implies $G[F^{n-1}]/G[F^{n-3}] \not \cong \bM_2$.
\end{enumerate}

First, we shall find an equation for $X_1$.
We use Lemma \ref{lem: description of local alpha_p-torsor}.
Write $X = \Spec R$ with $R = k[[x,y,z]] / (f)$ and $f = z^2 + x^3 + y^2 (y^3 + zxy)$.
Note that 
\[
\bigl( \frac{z}{xy} \bigr)^2 
 \,=\, \frac{y^3 + zxy}{x^2} \,+\, \frac{x}{y^2}.
\]
Since $H^2_{\idealm}(\OO_X)[F] = \langle [x^{-1} y^{-1} z] \rangle$ (see Proposition \ref{prop: torsors loclocloc new}),
the $\balpha_2$-torsor $U_1 := (X_1)_U \to U$ corresponds to a class $a [x^{-1} y^{-1} z]$
with $a \in k^\times$. Again we may assume $a = 1$.
We have 
\begin{align*}
\OO_{U_1}[x^{-1}] &= R[x^{-1}][t_1] / (t_1^2 - (x^{-2} (y^3 + zxy) + b)), \\
\OO_{U_1}[y^{-1}] &= R[y^{-1}][t_2] / (t_2^2 - (y^{-2} x + b))
\end{align*}
for some $b \in R$ and where
$t_1 - t_2 = x^{-1} y^{-1} z$.
Here, the $\balpha_2$-action is given by a derivation $D$ with 
$D(t_i) = 1$ for $i = 1,2$.
After possibly adding a constant to the $t_i$, we may assume that $b \in \idealm_{R}$.
Let $u_1 = x t_1$ and $u_2 = y t_2$.
We have 
\begin{equation} \label{eq: E81main,eq1}
u_1^2 + y^3 + zxy + x^2 b \,=\, u_2^2 + x + y^2 b 
\,=\, z + y u_1 + x u_2 \,=\, 0.
\end{equation}
Let $S = R[u_1,u_2]$. 
Solving Equation \eqref{eq: E81main,eq1} for $x$ and $z$, we 
have $S = k[[y,u_1,u_2]] / (g)$, with 
$$ 
g \,=\, u_1^2 + y^3 + u_2^4 b + y^4 b^3 + y u_2^5 + y^5 u_2 b^2 + u_1 y^2 u_2^2 + u_1 y^4 b. 
$$
By Lemma \ref{lem: description of local alpha_p-torsor}, we have $Y = \Spec ~ S$.

To show Claim (\ref{item: not n=2}), we show that the Tyurina ideal of $S$ does \emph{not} satisfy the property stated in the latter part of Lemma \ref{lem: alpha_2 Tyurina}.
Since $x \equiv 0$ and $z \equiv yu_1 \pmod{\idealm_{S}^{(2)}}$, we have 
$b \equiv \beta_1 y + \beta_2 y u_1 \pmod{\idealm_{S}^{(2)}}$ for some 
$\beta_1, \beta_2 \in k$. 
This implies $b_{u_2} \in \idealm_{S}^{(2)}$. 
Now, consider the ideal
$$
I' \,:=\, I + (y^2, u_1^2, y u_2^2, u_1 u_2^2, u_2^3)^{(2)} \,=\, 
I + (y^4, u_1^4, y^2 u_2^4, u_1^2 u_2^4, u_2^6). 
$$
We have
$$
yu_2^4 \,=\, g_{u_2} + u_2^4 b_{u_2} + y^4 (b^2 b_{u_2} + y b^2 + u_1 b_{u_2}) \,\in\, I'.
$$
But $(\partial / \partial y) (y u_2^4) = u_2^4 \notin I'$. 
By Lemma \ref{lem: alpha_2 Tyurina}, this shows that $X_1$ is not an $\balpha_2$-quotient singularity.

Next, we prove Claim (\ref{item: not alpha_4}).
Recall from Equation \eqref{eq: E81main,eq1} that 
$$
S \,=\, k[[x,y,u_1,u_2]] / (u_1^2 + y^3 + (yu_1 + xu_2)xy + x^2 b, u_2^2 + x + y^2 b),
$$ 
with $b \in R = k[[x,y,z]] = k[[x, y, y u_1 + x u_2]]$. 
This presentation of $S$ shows that it is a free $k[[x,y]]$-module with basis 
$u_1^{k_1} u_2^{k_2}$ ($k_1, k_2 \in \{0,1\}$), so we have a basis 
$e_{i,j}^{k_1,k_2} = [x^{-i} y^{-j} u_1^{k_1} u_2^{k_2}]$ 
($i,j \geq 1$, $k_1, k_2 \in \{0,1\}$)
of the $k$-vector space $H^2_{\idealm}(S)$.
We claim that $\overline{H^1(U, \balpha_{p^2})}_{\prim} = \emptyset$ and for this, 
it suffices to show that $H^2_{\idealm}(S)[F^2] = H^2_{\idealm}(S)[F] = \langle e_{1,1}^{1,0}, e_{1,1}^{1,1} \rangle =: W$.
The inclusions $H^2_{\idealm}(S)[F^2] \supset H^2_{\idealm}(S)[F] \supset W$ are clear 
and it remains to show that $F^{-1}(W) \subseteq W$.
It suffices to show that 
$\langle e_{1,1}^{0,1}, e_{2,1}^{1,1}, e_{1,2}^{1,1} \rangle \cap F^{-1}(W) = 0$. 
We have
$$
\begin{array}{cclcl}
    F(e_{1,1}^{0,1}) &=& \biggl[ \biggl( \dfrac{    u_2}{x y}   \biggr)^2 \biggr] &=& 
    \biggl[ \dfrac{1}{x y^2} \biggr], \\
    F(e_{2,1}^{1,1}) &=& \biggl[ \biggl( \dfrac{u_1 u_2}{x^2 y} \biggr)^2 \biggr] &=& 
    \biggl[ \dfrac{u_2}{x y} + \dfrac{b}{x y^2} \biggr], \\
    F(e_{1,2}^{1,1}) &=& \biggl[ \biggl( \dfrac{u_1 u_2}{x y^2} \biggr)^2 \biggr] &=& 
    \biggl[ \dfrac{1}{x y}  \biggr], 
\end{array}$$
and the first terms of the right-hand side are $k$-linearly independent.

Finally, for Claim (\ref{item: not M_2}), we use that
$\overline{H^1(U, \bM_3)}_{\prim} = \emptyset$ by Proposition \ref{prop: torsors loclocloc new} (Table \ref{table: torsors with Picloclocloc}). 
Hence,  $G/G[F^{n-3}] \not\cong \bM_3$ follows from Corollary \ref{cor: no primitive class}. 
\end{proof}

\begin{Remark} \label{rem: E81remark}
The proof of Lemma \ref{lem: E81main} shows that the total space of every local 
$\balpha_2$-torsor of primitive class over the $E_8^1$-singularity is a normal but \emph{non-rational} surface singularity. 
Thus, in order to understand local torsors of length $p^3$ --- which may very well 
have non-abelian structure group by Proposition \ref{prop: groupschemesoforderp3} --- 
over the $E_8^1$-singularity, we would have to calculate the (loc,loc)-part of the local 
Picard sheaf for all the (possibly non-isomorphic) non-rational singularities that occur.

This means that even to understand local torsors over the simplest class of surface singularities,
the RDPs, we have to study local torsors over some classes of non-rational singularities. 
This might give a taste of what awaits us when studying non-F-regular singularities 
of the MMP in higher dimensions.
\end{Remark}

This finishes the proof of Theorem \ref{thm: quotient} and Theorem \ref{thm: not quotient}. It remains to study the remaining $D_n^r$-singularities in characteristic $2$.

\subsection{The remaining $D_n^r$-singularities in characteristic 2} \label{subsec: dnr}
In this section, we study the $D_n^r$-singularities with $2 < 4r < n$ in more detail. 
We start with a description of $\balpha_2$-torsors over some of these singularities.

\begin{Lemma} \label{lem: alphap-torsor char 2 Dn}
Let $X$ be an RDP of type $D_n^r$ with $4r < n$ and $n-2r \in 4\ZZ$
in characteristic $p=2$.
Let $Y \to X$ be a local $\balpha_2$-torsor corresponding to a class $e$,
whose annihilator is equal to the annihilator of $H^2_{\idealm}(X,\OO_X)[F]$.
Then, the following hold:
\begin{enumerate}
    \item \label{item: type} $Y$ is an RDP of type $D_{4r}^{\max\{3r-n/2, 0\}}$ if $r \geq 1$, 
    of type $A_1$ if $r = 1/2$, and of type $A_0$ (smooth) if $r = 0$.
    \item \label{item: class} If $e' \in H^2_{\idealm}(X,W_l \OO_X)[F^m]$
    satisfies $F(e') = e$, then the pullback of $e'$ to $Y$ belongs to $H^2_{\idealn}(Y,\OO_Y)[F]$ and satisfies the same assumption as $e$.
\end{enumerate}
\end{Lemma}
\begin{proof}
Let $4j = n-2r$.
Suppose $n$ is even (resp.\ odd).
We may assume $X=\Spec R$ with $R = k[[x,y,z]]/(f)$,
$f = z^2 + x^2 y + y^{2j} g$ and $g = x (y^r + z)$ 
(resp.\ $g = z (y^{r-1/2} + x)$).
After possibly multiplying $e$ by a unit,
we may assume $e = [x^{-1} y^{-j} z]$.
As in Lemma \ref{lem: alphap-torsor char 2}, we have 
$$
\begin{array}{lcl}
  \OO_{Y_U}[x^{-1}] &=& R[x^{-1}][t_1] / (t_1^2 - (x^{-2} g + b)) \quad \text{and} \\
  \OO_{Y_U}[y^{-1}] &=& R[y^{-1}][t_2] / (t_2^2 - (y^{-(2j-1)}  + b))   
\end{array}
$$
for some $b \in R$ and $t_1 - t_2 = x^{-1} y^{-j} z$. 
After possibly adding a constant to the $t_i$, we may assume that $b \in \idealm$.
Set $S = R[u_1, u_2]$, where $u_1 = x t_1$ and $u_2 = y^j t_2$.
Then $S = k[[x, u_1, u_2]] / (h)$ with
\begin{align*}
h &= u_1^2 + x^2 u_2 + x y^r + x u_1 y^j + [x^2 b] , \\
(\text{resp.\ } h &= u_1^2 + x^2 u_2 + x u_2 y^{r-1/2} + x u_1 y^j + [u_1 y^{j+r-1/2} + x^2 b]),     
\end{align*}
where $y = u_2^2 + y^{2j} b$.
We conclude (again using Lemma \ref{lem: description of local alpha_p-torsor})
that $Y = \Spec S$, that the terms in brackets do not affect the singularity, 
and that $S$ is an RDP of the stated type ($D_{4r}^{\max\{2r-2j,0\}}$ if $r \geq 1$, of type
$A_1$ if $r = 1/2$, and of type $A_0$ if $r = 0$). 
This proves Claim (\ref{item: type}).

For Claim (\ref{item: class}), we first note that such an $e'$ exists if and only if $3j-1 < n/2$
(equivalently $2r-2j \geq 0$) and then, $e' = c f_{j'}^{(l')} + e'_2$, where $c$ is the
Teichm\"uller lift of a non-zero scalar, $(l',j')$ is the pair satisfying $C(l', j') = 3j-1$, 
and $e'_2$ is a linear combination of terms $f_{j''}^{(l'')}$ with $C(l'', j'') < C(l', j')$. 
It is straightforward to check that the pullback of $e'$ differs from 
$[x^{-1} y^{-2^{l'-1} j'} u_1 u_2^{2j + 2^{l'-1}-1}]$ by a unit, and that the annihilator of 
this element is $(u_1, x, u_2^j)$ (use that $y$ differs from $u_2^2$ by a unit). 
In other words, the pullback of $e'$ satisfies the assumptions of the lemma, which is what we 
had to show.
\end{proof}

Using the above lemma, we can realise some of these $D_n^r$-singularities as quotient 
singularities. 
Note however, that in these cases, we do not know the uniqueness of the group scheme 
realising them as quotient singularities.

\begin{Proposition} \label{prop: alpha2e} 
In characteristic $p=2$,
\begin{enumerate}
    \item for each $e \geq 1$ and $k \geq 1$, the RDP of type $D_{(2^{e+2}-4)k}^{(2^{e}-2)k}$  
  is a quotient singularity by $\balpha_{2^e}$, and
  \item for each $e \geq 2$, the RDP of type $D_{2^{e+1}-2}^{2^{e-1}-1}$  
  is a quotient singularity by $\bL_{2,e}[V-F^{e-1}]$.
  \end{enumerate}
\end{Proposition}

\begin{proof}
First, suppose that $X$ is of type $D_{(2^{e+2}-4)k}^{(2^{e}-2)k}$.
The element $f = f_{2^{e-1} k}^{(1)}$ in $\Dieulocloc(\Picloc_{X/k})$ 
satisfies $F^{e}(f) = 0$, and $F^{e-1}(f)$ satisfies the assumption of Lemma \ref{lem: alphap-torsor char 2 Dn}.
Let 
\begin{equation}
\label{eq: alpha2E quotient}
X_e \,\to\, \dots \,\to\, X_i \,\to\, \dots \,\to\, X_0 \,=\, X
\end{equation}
be an $\balpha_{2^e}$-torsor corresponding to this class, 
factored into local $\balpha_2$-torsors.
Using Lemma \ref{lem: alphap-torsor char 2 Dn} inductively, we deduce that $X_i$ is of type $D_{(2^{e+2}-2^{i+2})k}^{(2^{e} - 2^{i+1})k}$ for each $0 < i < e$, 
that the class of the $\balpha_2$-torsor $X_{i+1} \to X_i$ also satisfies the assumption of 
the lemma, and that $X_e$ is smooth. 
In particular, $X$ is a quotient singularity by $\balpha_{2^e}$.

Second, suppose that $X$ is of type $D_{2^{e+1}-2}^{2^{e-1}-1}$.
The element $f = f_{2^{e-2}}^{(2)}$ in $\Dieulocloc(\Picloc_{X/k})$ satisfies $F^e(f) = V(f)$, and $F^{e-1}(f)$ satisfies the assumption of Lemma \ref{lem: alphap-torsor char 2 Dn}. Let 
\begin{equation}
\label{eq: L2E quotient}
X_e \,\to\, \dots \,\to\, X_i \,\to\, \dots \,\to\, X_0 \,=\, X
\end{equation}
be an $\bL_{2,e}[V-F^{e-1}]$-torsor corresponding to this class, 
factored into local $\balpha_2$-torsors. Arguing as in the previous paragraph, we see that $X$ is a quotient singularity by $\bL_{2,e}[V-F^{e-1}]$. 
\end{proof}

However, not all of the $D_n^r$-singularities with $2 < 4r < n$ are quotient singularities, 
as the following result shows:

\begin{Proposition} \label{prop: notquotient dnr}
Let $x\in X$ be an RDP of type $ D_{4k+2}^{1}$ or $D_{4k+3}^{3/2}$ with $k \geq 2$
in characteristic $p=2$. 
Then, $X$ is not a quotient singularity.  
\end{Proposition}

\begin{proof}
We know that $X$ is not an $\balpha_2$-quotient singularity by 
\cite[Lemma 3.6]{Matsumotoalpha}.
We claim that $\overline{H^1(U, G')}_{\prim} = \emptyset$ for 
$G' \in  \{ \balpha_{p^2}, \bM_2 \}$. 
Then, it will follow from Lemma \ref{lem: possible argument} 
that $X$ is not a quotient singularity.

To prove $\overline{H^1(U, G')}_{\prim} = \emptyset$ for $G' \in  \{ \balpha_{p^2}, \bM_2 \}$,
we use the description of the Dieudonn\'e module $\varinjlim_{m,l} H^2_{\idealm}(X,W_l\OO_X)[F^m]$ given in Proposition \ref{prop: torsors loclocloc new}. 
We have to show that this Dieudonn\'e module satisfies 
$\Ker(F,V) = \Ker(F^2,V) = \Ker(F^2,V^2,F+V)$.

The Dieudonn\'e module is generated by $f_j^{(l)}$ for the pairs $(l,j)$ satisfying 
$C(l,j) \leq 2k$, and we have $F(f_j^{(l)}) = f_{C(l,j) - (2k-1)}^{(1)}$.
Since $C(l,j)$ is odd for $l = 1$, we have $\Ker(F^2, V) = \Ker(F, V)$.
Next, we have $F(f_j^{(l)}) = f_1$ if $C(l,j) = 2k$, and all other generators are 
annihilated by $F$. 
Letting $(l,j)$ as above, we have $V(f_j^{(l)}) = f_j^{(l-1)}$, 
which belongs to the image of $F$ only if $(l,j) = (2,1)$, 
but this implies $k = 1$, which contradicts the assumption $k \geq 2$.
\end{proof}

Finally, we prove Theorem \ref{thm: D_n in p=2} on the structure of a finite group 
scheme $G$ realising a $D_n^r$-singularity with $4r < n$ as quotient singularity.

\begin{proof}[Proof of Theorem \ref{thm: D_n in p=2}]
Suppose that $x \in X$ is of type $D_n^r$ with $4r < n$ and that $X$ is a 
quotient singularity by a very small action of a finite group scheme $G$.
Since $\pietloc(X)$ is trivial and $X$ is not F-regular, $G$ satisfies 
$G[F^{i+1}]/G[F^{i}] \cong \balpha_2$ for all $i$ smaller than the height of $G$ 
by Proposition \ref{prop: linearlyreductive or unipotent} (\ref{item: unipotent height n}). 
Let $G^{\ab} = G/N$ be its abelianisation.
Then, $\widehat{\Aff}^2/N \to X$ is a local $G^{\ab}$-torsor and induces an injective 
homomorphism $(G^{\ab})^D \to \Picloc_{X/k}$.
Since $G^{\ab}$ is of (loc,loc)-type, this induces an injective homomorphism 
$(G^{\ab})^D \to \Picloclocloc_{X/k}$
and thus, an injective homomorphism between their Dieudonn\'e modules. 
By Proposition \ref{prop: torsors loclocloc new} 
(see also Proposition \ref{prop: picloclocloc of DNr}), 
$\Dieulocloc(\Picloclocloc_{X/k})$ is killed by $2$, hence so is $\Dieulocloc(G^{\ab})$.
Since $G^{\ab}[F^{i+1}]/G^{\ab}[F^{i}] \cong \balpha_2$, this shows that $\Dieulocloc(G^{\ab})$ 
is generated by one element $f$ annihilated by $V^e$, but not by $V^{e-1}$, 
where $2^e$ is the length of $G^{\ab}$.
Since $FV = 2 = 0$ on this Dieudonn\'e module, we have $F(f) = c V^{e-1}(f)$ with $c \in k$. 
If $c = 0$, then $G^{\ab}$ is $\balpha_{2^e}$. 
If $c \neq 0$, then we may assume $c = 1$, hence $G^{\ab} = \bL_{2,e}[V-F^{e-1}]$, 
and $e \geq 2$ follows because $G^{\ab}$ is of (loc,loc)-type.

We have already seen in Proposition \ref{prop: alpha2e} that these group schemes indeed 
occur as $G$.
\end{proof}

The results of this subsection can be viewed as partial answers to the following two questions, 
which are still open in general:

\begin{Questions}
Let $x \in X$ be an RDP of type $D_n^r$ with $2 < 4r < n$ in characteristic $p=2$.
\begin{enumerate}
    \item For which pairs $(n,r)$ is $X$ a quotient singularity?
    \item If $X = \widehat{\Aff}^d/G$ is a realisation of $X$ as a quotient singularity, 
    is $G \in \{ \balpha_{2^e}, \bL_{2,e}[V - F^{e-1}] \}$?
\end{enumerate}
\end{Questions}

\begin{Remark} \label{rem: quotient dnr}
In view of Proposition \ref{prop: alpha2e} and Proposition \ref{prop: notquotient dnr}, an optimistic guess would be that an RDP
of type $D_n^r$ with $2 < 4r < n$ is a quotient by $\balpha_{2^e}$ if and only if $(n,r) = ((2^{e+2}-4)k, (2^{e}-2)k)$ for some $k \geq 1$ and a quotient by $\bL_{2,e}[V - F^{e-1}]$ if and only if $(n, r) = (2^{e+1}-2, 2^{e-1}-1)$,
and not a quotient in all other cases.
\end{Remark}

\subsection{Generalised quotient singularities} \label{subsec: generalised quotient singularities}
In \cite{LRQ}, we have seen that
lrq singularities, that is, quotient singularities by finite and
linearly reductive group schemes, behave extremely nice:
the singularity determines the group scheme and the quotient presentation
uniquely \cite[Section 8]{LRQ} and
many invariants of the singularity can be read off from this
presentation, see \cite[Section 7]{LRQ}.
By \cite[Theorem 11.2]{LRQ}, this applies
to the F-regular RDPs.

If an RDP singularity is not F-regular, then 
Table \ref{table: Quotient- and non-quotient RDPs} shows that in many
cases, the singularity is a quotient singularity in the sense
of Definition \ref{def: quotient singularity}.
Moreover, we even have uniqueness statements in these cases.
Of course, the group scheme will not be linearly reductive
in these cases and by \cite[Proposition 6.4]{LRQ},
the action of this group scheme on $\widehat{\Aff}^2_k$
will not be linearisable.
However, at least some of the invariants, such as the local
fundamental group of the singularity, can be read off from
the presentation, see Proposition \ref{prop: etale quotients unique}.

This begs for the question whether the non-quotient RDPs
of Theorem \ref{thm: not quotient} are quotient singularities
in a more general sense than Definition \ref{def: quotient singularity}.
Here are a couple of notions and suggestions of what such 
generalisations could be:

\begin{Definition}
 \label{def: generalised quotient singularity}
 A \emph{generalised quotient singularity} $x\in X$ over $k$ is a 
 singularity $x \in X$, such that there exists
 a finite morphism $f:\widehat{\Aff}^d\to X$, such that
 at least one of the following holds
 \begin{enumerate}
     \item \label{item: iterated} There exists a factorisation of $f$
    $$
       f\,:\,\widehat{\Aff}^d\,=\,X_e\,\to\,\ldots\,X_1\,\to\, X_0 = X,
    $$
    such that for each step $i$ there exists a finite $k$-group scheme
    $G_i$ and an action on $X_i$ that is free outside the closed point and fixes the closed point and such that $X_i\to X_{i-1}$ is the quotient
    by this $G_i$-action.
    \item  \label{item: generalised}
    There exists a finite flat group scheme $G$ over $X$ that acts on $\widehat{\Aff}^d := \Spec k[[u_1,...,u_d]]$, such that $f$ is the quotient by $G$, $f_U$ is a $G_U$-torsor and $G_x$ fixes the closed point of $\widehat{\Aff}^d$.
     \item \label{item: iteratedandgeneralised}
     There exists a factorisation of $f$
    $$
       f\,:\,\widehat{\Aff}^d\,=\,X_e\,\to\,\ldots\,X_1\,\to\, X_0 = X,
    $$
    such that for each step $i$ there exists a finite flat $X_{i-1}$-group scheme
    $G_i$ and an action of $G_i$ on $X_i$ such that $X_i\to X_{i-1}$ is the quotient
    by this $G_i$-action, $U_i \to U_{i-1}$ is a $(G_i)_{U_{i-1}}$-torsor and $(G_i)_{x_{i-1}}$ fixes the closed point of $X_i$. 
    \end{enumerate}
\end{Definition}

\begin{Remark}
In characteristic $0$, Properties (1), (2), and (3) of Definition \ref{def: generalised quotient singularity} 
are equivalent to each other and also equivalent to $x\in X$ being a quotient singularity 
in the sense of Definition \ref{def: quotient singularity}:

Since Property (3) is the most general, it suffices to show that (3) implies that $x\in X$ 
is a quotient singularity in the sense of Definition \ref{def: quotient singularity}. 
So, suppose that $f$ as in Part (3) of Definition \ref{def: generalised quotient singularity} 
exists. 
Then, $f|_U$ is \'etale and $\widehat{\Aff}^d \setminus \{0\}$ is simply connected, 
hence $\widehat{\Aff}^d \setminus \{0\}$ is the universal \'etale cover of $U$.
Thus, $x\in X$ is a quotient singularity by $\pietloc(X)$. 
\end{Remark}

We will see in Theorem \ref{thm: generalisedquotient} that in characteristic $p > 0$, 
Definition \ref{def: generalised quotient singularity} is indeed more general than 
Definition \ref{def: quotient singularity}. 
First, we recall how to construct a group scheme over $X$ from a realisation of $x\in X$ 
as a quotient of a singularity $y\in Y$ by a derivation.

\begin{Lemma} \label{lem: derivationyieldsgroupscheme}  
Let $x \in X$ and $y \in Y$ be singularities. 
Let $\pi \colon Y \to X$ be a finite morphism.
Suppose $\pi \colon Y \to X$ is a quotient by a $p$-closed derivation $D$.
Let $\Fix(D)$ be the zero locus of $D$. 
Then, there exists a finite flat group scheme $G$ over $X$ and an action of $G$ on $Y$,
such that $\pi$ is the quotient by this $G$-action, 
$\pi_{X \setminus \Fix(D)}$ is a $G_{X \setminus \Fix(D)}$-torsor, 
and the $G$-action fixes $\Fix(D)$.
\end{Lemma}

\begin{proof}
Assume $X=\Spec R$.
Let $D^p = f D$. 
Since $Y$ is reduced, we have $f \in R$ by \cite[Lemma 2.3]{Matsumotoinseparable}.
We consider the restricted $R$-Lie algebra $\mathfrak{g}$ with underlying $R$-module $R$, 
trivial Lie bracket, and $p$-operation mapping $1$ to $f$. 
Let $G = \Spec ~ U_p(\mathfrak{g})^\vee$ be the finite flat $X$-group scheme of height $1$ 
that is the spectrum of the linear dual of the restricted universal enveloping algebra of $\mathfrak{g}$. 
By \cite[Th\'eor\`eme 7.2. (ii)]{SGA31}, the derivation $D$ determines an action of $G$ on $Y$ 
over $X$. 
The fixed locus of the $G$-action is precisely $\Fix(D)$ and the $G$-action is free 
outside $\Fix(D)$. 
Indeed, this can be checked on fibres of $\pi$, where it is well-known. 
Hence, $\pi$ is a $G$-torsor over $X \setminus \Fix(D)$ and $G$ fixes $\Fix(D)$.
\end{proof}

\begin{Theorem} \label{thm: generalisedquotient}
Let $x\in X$ be an RDP in characteristic $p>0$.
\begin{enumerate}
    \item If $p\geq7$, then $X$ is an lrq singularity.
    \item If $p\geq5$, then $X$ is a quotient singularity in the sense of Definition \ref{def: quotient singularity}. 
    If $p=5$, then there exist RDPs that are not lrq singularities.
    \item If $p\geq3$, then $X$ is a generalised quotient singularity in the sense of Definition \ref{def: generalised quotient singularity} (\ref{item: generalised}).
    If $p=3$, then there exist RDPs that are not quotient singularities in the sense of Definition \ref{def: quotient singularity}. 
    \item \label{item: everythingisgeneralisedquotient} If $p\geq2$, then $X$ is a generalised quotient singularity in the sense of Definition \ref{def: generalised quotient singularity} (\ref{item: iteratedandgeneralised}).
\end{enumerate}
\end{Theorem}

\begin{proof}
Assertion (1) follows from \cite[Theorem 11.2]{LRQ} and 
Assertion (2) follows from Theorem \ref{thm: quotient}.

For Assertion (3), note that by Theorem \ref{thm: not quotient} it suffices to prove that 
the RDP of type $E_8^0$ in characteristic $3$ is a generalised quotient singularity in the 
sense of Definition \ref{def: generalised quotient singularity} (\ref{item: generalised}). 
Using Lemma \ref{lem: derivationyieldsgroupscheme}, this follows from the realisation of 
$E_8^0$ as a quotient of $k[[x,y]]$ by a $p$-closed derivation given in 
\cite[Theorem 1.1]{Matsumotoinseparable}.

Assertion (4) follows by applying Lemma \ref{lem: derivationyieldsgroupscheme} to 
\cite[Theorem 4.1]{Matsumotoinseparable}.
\end{proof}

\begin{Remark}
In particular, if $p>0$, then the singularities in
Definition \ref{def: generalised quotient singularity} are more general than those
in Definition \ref{def: quotient singularity}.
\end{Remark}

\begin{Remark}
In Theorem \ref{thm: generalisedquotient} (\ref{item: everythingisgeneralisedquotient}), 
it is possible to find a presentation as a generalised quotient singularity, 
such that every intermediate $X_i$ is an RDP as well.
For all RDPs different from $E_8^1$ in characteristic $2$, 
this is proved in \cite[Theorem 4.1]{Matsumotoinseparable}. 
For the remaining case, it is proved in \cite[Remark 4.4]{Matsumotoinseparable} that there is a local $\balpha_4$-torsor 
over $E_8^1$ with total space of type $D_8^0$, and $D_8^0$ is a quotient singularity by $\balpha_2$.
\end{Remark}

\begin{Remark}
\label{rem: KollarConjecture}
In \cite[Conjecture 2.24.1]{KollarLectures}, Koll\'ar conjectured that 
a normal complex variety $Y$ has only quotient singularities if and only if
there exists a finite and dominant morphism $f:Y\to X$ from a smooth variety.
This conjecture is true for surfaces, but open already for threefolds.
If $x\in X$ is an RDP singularity in positive characteristic, 
then Artin showed that there
always exists such a finite and dominant morphism,
see the discussion after \cite[Question (1.3)]{ArtinRDP}.
Theorem \ref{thm: generalisedquotient} shows that $x\in X$
need not be a quotient singularity in the sense of
Definition \ref{def: quotient singularity}.
In particular, Koll\'ar's conjecture is false in positive characteristic.
In fact, it already fails for RDPs, which are two-dimensional.
However, as Theorem \ref{thm: generalisedquotient} shows, 
Koll\'ar's conjecture might still hold in positive characteristics 
if one allows more generalised quotient singularities, for example,
generalised quotient singularities in the sense of Definition \ref{def: generalised quotient singularity}.
\end{Remark}

\section{Pathologies and counter-examples}\label{sec: pathologies}
In this final section, we give a couple of examples of local torsors over the rational double points.
These examples show bad behavior of such torsors over non-F-regular singularities, which
have no analogue in characteristic zero.

\subsection{Non-uniqueness of representatives}
If $G$ is a finite abelian $k$-group scheme and $x\in X=\Spec R$ 
is a singularity over $k$, then we have an exact sequence
 $$
  \Hfl{1}(X,G) \,\to\, \Hfl{1}(U,G) \,\to\, \overline{\Hfl{1}(U,G)}  \,\to\,0,
 $$
where $x\in X$ is the closed point and $U=X-\{x\}$.
Now, if $G$ is \'etale, that is, a finite group, then
$\Hfl{1}(X,G)$ is trivial by Hensel's lemma.
In particular, every class in $\overline{\Hfl{1}(U,G)} $ can be lifted to a 
\emph{unique} local $G$-torsor.

If $G$ is not \'etale, then $\Hfl{1}(X,G)$ is non-trivial and then,
an equivalence class does not have a unique representative.
In some cases, these non-unique representatives lead to local $G$-torsors whose
total spaces are abstractly isomorphic, see, for example, 
Lemma \ref{lem: alphap-torsor char 3} or Lemma \ref{lem: alphap-torsor char 2}.
However, it turns out that there are examples of RDPs where a single equivalence class of 
local torsors admits infinitely many representatives whose total spaces are pairwise 
non-isomorphic RDPs.

\begin{Proposition} \label{prop: non-uniqueness of torsors}
 Let $x\in X$ be an RDP of type $D_4^0$ (resp.\ $D_6^0$) in characteristic $p=2$.
 Let $G = \bmu_2$ (resp.\ $G = \balpha_2$).
 Then, there exists an infinite 
 sequence of local $G$-torsors
  $$ 
      \{X_i\,\to\,X\}_{i\in\NN},
  $$
  such that the images of the classes of $[X_i \to X]\in\Hfl{1}(U,G)$
  in $\overline{\Hfl{1}(U,G)} $ coincide and such that the $X_i$ are pairwise non-isomorphic 
  RDPs of type $D_{4+2i}^{1}$ (resp.\ $D_{4+2i}^0$).
\end{Proposition}

\begin{proof}
If $X$ is of type $D_4^0$, recall that an equation for $D_4^0$ is given by $z^2 + x^2y + xy^2 = 0$ and that the class group of $D_4^0$ is generated by the ideals $(z,x),(z,y),$ and $(z,x+y)$.
In particular, the action of the symmetric group $S_3$ on $D_4^0$ given by permuting $x$, $y$,
and $x+y$ 
induces a transitive action on ${\rm Cl}(X)$.
If $X$ is of type $D_6^0$, then $\dim_k \overline{\Hfl{1}(U,\balpha_{p})} = 1$ 
and the  $\Aut(\balpha_p) = R^\times$-action induces a transitive action on
$\overline{\Hfl{1}(U,\balpha_{p})} \setminus \{0\}$. 
This means that in both cases there is a unique primitive class in $\overline{\Hfl{1}(U,G)} $ up to automorphisms of $X$ and $G$. Thus, when proving our claim, we do not have to keep track of the class of $[X_i \to X]$ 
in $\overline{\Hfl{1}(U,G)} $. 

If $X$ is $D_4^0$, then the sequence $X_i$ is given in Row 7 of \cite[Table 3]{Matsumotoinseparable} 
if one sets $n = 2$ and $m = i + 1$. 
If $X$ is $D_6^0$, then the sequence $X_i$ is given in Row 5 of \cite[Table 3]{Matsumotoinseparable} 
if one sets $m = 2$ and $m' = i + 1$.
\end{proof}

\subsection{(Non-)normal representatives}

Despite the fact that local torsor classes need not have unique representatives,
it is remarkable that we can always find normal representatives of primitive classes
of $\bmu_n$-torsors and $\balpha_p$-torsors.
However this is not true in general for $\bmu_2 \times \bmu_2$-torsors.

\begin{Theorem} \label{thm: (no) normal representative}
 Let $x \in X$ be an RDP in characteristic $p>0$.
 \begin{enumerate}
   \item \label{item: no normal representative} If $x\in X$ is of type $D_{2m}^r$ and $p=2$
 with $r < m-1$ and $H = {\rm Cl}(X)$, then every primitive class in  $\overline{\Hfl{1}(U,H^D)}$ does not admit a normal representative.
 \item \label{item: normal representative mu} If $H\subseteq {\rm Cl}(X)$ is a subgroup and we are not in the situation
 of (\ref{item: no normal representative}), then every primitive class in  $\overline{\Hfl{1}(U,H^D)}$
 admits a normal representative.
 \item \label{item: normal representative alpha} Every primitive class in $\overline{\Hfl{1}(U,\balpha_p)}$ admits a normal representative.
 \end{enumerate}
\end{Theorem}

\begin{proof}
First, assume that $X$ is F-regular. 
If $G$ is finite and abelian and there exists a local $G$-torsor $Y \to X$ of primitive class, 
then $G$ is linearly reductive by \cite[Theorem D]{Carvajal-Rojas} and hence, $Y$ is also F-regular 
by \cite[Theorem C]{Carvajal-Rojas}. 
In particular, $Y$ is normal (in fact, it is an RDP). 
This proves (\ref{item: no normal representative}), (\ref{item: normal representative mu}), and 
(\ref{item: normal representative alpha}) for all F-regular RDPs. In the following, we give a proof for the non-F-regular RDPs, which only exist in characteristic $p$ equal to $5$,$3$, or $2$.

If $p = 5$, the RDPs that are not $F$-regular are the ones of type $E_8^r$. In these cases, ${\rm Cl}(X)$ is trivial, hence (\ref{item: no normal representative}) and (\ref{item: normal representative mu}) are trivially satisfied if $p = 5$. For Claim (\ref{item: normal representative alpha}), this space is trivial if $r=1$, hence we assume $r=0$. Note that, up to scalar multiplication (which is induced by the action of $\Aut(\balpha_p)$, as explained in Example \ref{ex: endomorphismsofalphap}), there is a unique primitive class in $\overline{\Hfl{1}(U,\balpha_p)}$, which is represented by the local $\balpha_p$-torsor realising $X$ as a quotient singularity by $\balpha_p$. In particular, the total space of this local $\balpha_p$-torsor is normal, which proves Claim (\ref{item: normal representative alpha}) if $p = 5$.

If $p = 3$, the RDPs that are not F-regular are $E_6^r,E_7^r$, and $E_8^r$. Assume that $X$ is one of these RDPs. For $H \subseteq {\rm Cl}(X)$, the primitive class in $\overline{\Hfl{1}(U,H^D)}$ is unique up to $\Aut(H)$ and it has a normal representative, either because $H^D$ is \'etale or by \cite[Table 3]{Matsumotoinseparable}. Hence, Claims (\ref{item: no normal representative}) and (\ref{item: normal representative mu}) hold if $p = 3$. Similarly, the primitive class in $\overline{\Hfl{1}(U,\balpha_p)}$, if it exists, is unique up to scaling and a normal representative of this class is given in \cite[Table 3]{Matsumotoinseparable}. This proves Claim (\ref{item: normal representative alpha}) if $p = 3$.

If $p = 2$, we assume that $X$ is one of the RDPs that are not F-regular, that is, either $D_n^r$ or $E_n^r$.

First, consider Claim (\ref{item: no normal representative}), that is, assume that $X$ is of type $D_{2m}^r$ with $r < m -1$ and $H = {\rm Cl}(X) = \C_2 \times \C_2$. Seeking a contradiction, assume that $\Spec S \to \Spec R$ is a local $H^D$-torsor of primitive class with $S$ normal.
Since $S$ is normal and $(\Frac S)^{(2)} \subseteq \Frac R$, we have $R = S^{(2)}$, where $S^{(2)}$ denotes the pullback of $S$ along the Frobenius on $k$. 
In particular $S$ is an RDP of the same type as $R$.
The maximal ideal $\idealn$ of $S$ is generated by three elements $x_1,x_2,x_3 \in \idealn$ that are homogeneous with respect to the $H$-graded structure (recall Lemma \ref{lem: mu_p-torsor} for the graded structure induced by the $H^D$-torsor structure).
If $x_i$ is of degree $0 \in H$, then $x_i \in R = S^{(2)}$, a contradiction.
If $x_i$ and $x_j$ ($i \neq j$) are of the same degree, then $x_i x_j \in R = S^{(2)}$, hence $x_i x_j$ is equal to the square of some element in $\idealn$, which is impossible since $S$ is of type $D_{2m}^r$.
Hence $\{ \deg(x_i) \} = H \setminus \{ 0 \}$.

With respect to each subgroup scheme $\bmu_2 \subseteq H^D$, exactly two of the $x_i$ (say $x_1$ and $x_2$) have non-trivial degree, and as in \cite[Lemma 4.5(2)]{Matsumotomu} there exists a homogeneous element $F \in k[[x_1, x_2, x_3]]$ with $S = k[[x_1, x_2, x_3]] / (F)$. Then $\deg(F) = 0$, since the coefficient of $x_i^2$ is non-zero for some $i$ (otherwise $S$ cannot be an RDP of type $D_{2m}$). 
Since this holds for every subgroup scheme $\bmu_2 \subseteq H^D$, we have $F \in k[[x_1^2, x_2^2, x_3^2, x_1 x_2 x_3]]$.
Suppose the coefficient of $x_1 x_2 x_3$ is non-zero. Then, by Fedder's criterion \cite[Proposition 1.7]{Feddercriterion}, $R$ is $F$-split, contradicting the assumption that $r < m-1$.
But if the coefficient is zero, then $F$ is of the form $w^2 + G$ with $G \in \idealm^4$, which cannot define an RDP.
This contradiction shows that $S$ cannot be normal, hence Claim (\ref{item: no normal representative}) is proved.

To prove Claim (\ref{item: normal representative mu}), assume first that $H$ is not cyclic. Then, by Table \ref{table: torsors with Picloclocloc} and since we are not in the situation of Claim (\ref{item: no normal representative}), $X$ is of type $D_{2m}^{r}$ with $r = m-1$, $H = {\rm Cl}(X) = \C_2 \times \C_2$, and the primitive class in $\overline{\Hfl{1}(U,H^D)}$ is unique up to the action of $\Aut(H)$. Let $S = k[[x_1,x_2,x_3]] / (x_1^2 + x_2^2 + x_3^{2(m-1)} + x_1 x_2 x_3)$ and equip it with a $\ZZ/2\ZZ \times \ZZ/2\ZZ$-graded structure such that the $x_i$ are homogeneous of pairwise distinct non-zero degree. Then $\Spec S \to \Spec S^{(2)}$ is a $\bmu_2 \times \bmu_2$-torsor over $X$. Since $S$ is normal, this proves Claim (\ref{item: normal representative mu}) for non-cyclic $H$. 

Hence, we may assume that $H \subseteq {\rm Cl}(X)$ is cyclic, say of order $n > 1$. If $p \nmid n$, then $H$ is \'etale and all representatives are normal. If $p \mid n$, then $n = 2$ and $X$ is of type $E_7^r$, $n = 2$ and $X$ is of type $D_{2m}^r$, or $n \in \{2,4\}$ and $X$ is of type $D_{2m+1}^{r + 1/2}$.

If $X$ is of type $E_7^r$, then $n = 2$ and the primitive class in $\overline{\Hfl{1}(U,H^D)}$ is unique, so it suffices to give one normal representative of a local $\bmu_2$-torsor of primitive class over $X$. Set $\varepsilon := z,Y,zY,Y^2$ for $r = 3,2,1,0$, respectively. 
Then 
$$
  \Spec k[[x,y,z]] / (x^2 + z^3 + y^4 + xy\varepsilon) \,\to\, 
  \Spec k[[w,z,Y]] / (w^2 + Y(z^3 + Y^2 + w\varepsilon))
$$
with $Y = y^2$ and $w = xy$ is a normal $\bmu_2$-torsor over $E_7^r$,
where the $\bmu_2$-action is given by the $\ZZ/2\ZZ$-graded ring structure with homogeneous 
elements $x, y, z$ of weight $1,1,0$ respectively.
For $r = 3,2$, this is $E_7^3$, $E_8^3$ respectively.
For $r = 1,0$, this is not a rational double point, but a normal elliptic singularity (see Remark \ref{rem: salmon}).

If $X$ is of type $D_{2m}^r$ and $n = 2$, then there are either one or two primitive classes in $\overline{\Hfl{1}(U,H^D)}$ up to the action of $\Aut(X)$. The morphisms
\begin{gather*}
\Spec k[[x,y,z]] / (x^2 + y^2 z + z^m + z^{m-r} xy)
\\ \to \Spec k[[w,z,Y]] / (w^2 + Y(Y z + z^m + z^{m-r} w)))
\end{gather*}
and 
\begin{gather*}
\Spec k[[x,y,z]] / (x^2 + z^2 + z^3 + y^{2(m-1)} z + y^{2(m-r-1)} z xy) 
\\ \to \Spec k[[w,z,Y]] / (w^2 + Y(z^2 + z^3 + Y^{m-1} z + Y^{m-r-1} z w))
\end{gather*}
are normal local $\bmu_2$-torsors of primitive class,
and up to $\Aut(X)$ these classes covers all non-trivial classes of $\overline{\Hfl{1}(U,G)}$ (if $2m > 4$, the first one covers two classes and the second covers the remaining one; if $2m = 4$, then $\Aut(X)$ acts on $\mathrm{Cl}(X) \setminus \{0\}$ transitively). More precisely, the total space of the first $\bmu_2$-torsor is a rational double point of type $D_{3m-2r}^{m/2}$ and the second one is not a rational double point if $2m > 4$, but still normal.

If $X$ is of type $D_{2m+1}^{r+1/2}$, then the primitive class in $\overline{\Hfl{1}(U,H^D)}$ is unique up to $\Aut(H)$. The morphism
\begin{gather*}
\Spec k[[x_1, y_2, z_3]] / ((1 + x_1 z_3) z_3^2 + x_1^2 + y_2^{2m-2r-1} (y_2^{2r} (1+x_1 z_3) + x_1 z_3))
\\ \to \Spec k[[y_2, A, Z]] / (A^2 + (1+A) Z^2 + Z y_2^{2m-2r-1} (y_2^{2r} (1+A) + A))
\\ \to \Spec k[[Y, A, B]] / ((1+A) B^2 + Y A^2 + B Y^{m-r} (Y^r (1+A) + A)),
\end{gather*}
where $A = x_1 z_3$, $Z = z_3^2$, $Y = y_2^2$, $B = y_2 Z$, is a local $\bmu_4$-torsor of primitive class over $X$, 
under the $\ZZ/4\ZZ$-graded ring structure with homogeneous elements $x_1, y_2, z_3$ of indicated weights.
The intermediate torsor is a local $\bmu_2$-torsor of primitive class. The total spaces of both the local $\bmu_4$-torsor and the local $\bmu_2$-torsor are normal. 
This finishes the proof of Claim (\ref{item: normal representative mu}).

Finally, consider Claim (\ref{item: normal representative alpha}). If $X$ is not of type $D_n^r$ or $E_8^0$, then the primitive class in $\overline{\Hfl{1}(U,\balpha_p)}$, if it exists, is unique up to scaling, hence it suffices to give a single local $\balpha_p$-torsor of primitive class over $X$ with normal total space. Examples of those are given in \cite[Tables 2 and 3, and proof of Theorem 4.1]{Matsumotoinseparable}. We will use that, since $p = 2$, the annihilator of the class of a local $\balpha_p$-torsor $\Spec S \to \Spec R$ given by a derivation $D$ is equal to the ideal $(\Image(D)) \subseteq R$ by \cite[Lemma 6.3]{Matsumotoalpha}. 

If $X$ is of type $E_8^0$, then $\overline{\Hfl{1}(U,\balpha_p)}$ is $2$-dimensional, generated as an $R$-module by the local $\balpha_p$-torsor with smooth total space given in \cite[Lemma 3.6]{Matsumotoalpha}. The annihilator of this generator is $(x,y^2,z)$ with respect to the equation of $X$ given in Table \ref{table: RDPEquations}. Next consider the local $\balpha_p$-torsor
\begin{gather*}
\Spec k[[x,y,z]] / (x^2 + z^3 + y^4 (y^3 + z x)) 
\\ \to \Spec k[[w,z,Y]] / (w^2 + Y^3 + z^2 (z^3 + Y^2 w)),
\end{gather*}
where $Y = y^2$ and $w = y^3 + z x$,
with derivation $D$ defined by $(D(x),D(y),D(z)) = (y^2, z, 0)$.
Then $(\Image(D)) = (D(x),D(y),D(xy)) = (Y, z, w)$. This local $\balpha_p$-torsor has normal total space and its class in $\overline{\Hfl{1}(U,\balpha_p)}$ is $k$-linearly independent from the class of the local torsor with smooth total space, since its annihilator has different colength.

If $X$ is of type $D_n^r$, then, by Table \ref{table: torsors with Picloclocloc} (Proposition \ref{prop: torsors loclocloc new}), $\overline{\Hfl{1}(U,\balpha_p)} $ is $l := \floor{\frac{1}{2}(\frac{n}{2}-r)}$-dimensional, generated by an element whose annihilator is $(x,y^l,z)$ with respect to the equation given in Table \ref{table: RDPEquations}.
Thus it suffices to give, for each $1 \leq i \leq l$, a local $\balpha_p$-torsor with normal total space whose class $e_i$ satisfies $\length(R/\Ann(e_i)) = i$.
Take an odd integer $j \gg 1$.
Let $s = 2r \in \ZZ$ (not necessarily $\in 2\ZZ$) and $h = (n - 2r - 4i)/2 \in \ZZ_{\geq 0}$. 
Consider the following local $\balpha_p$-torsor:
\begin{gather*}
\Spec k[[x,y,z]] / (x^2 + z^j + y^{2h} z (z y + y^s + y^{2i} x)) 
\\ \to \Spec k[[w,z,Y]] / (w^2 + z^2 Y + Y^{s} + Y^{2i} (z^j + Y^h z w)),
\end{gather*}
where $Y = y^2$ and $w = z y + y^s + y^{2i} x$,
with derivation $D$ defined by $(D(x),D(y),D(z)) = (z + s y^{s-1}, y^{2i}, 0)$.
We calculate that $(\Image(D)) = (D(x),D(y),D(xy)) = (z, Y^i, w)$, so that the class $e_i$ of this local $\balpha_p$-torsor satisfies $\length(R/\Ann(e_i)) = i$, as desired. Since this local $\balpha_p$-torsor has normal total space, this finishes the proof of Claim (\ref{item: normal representative alpha}).
\end{proof}

\subsection{Non-rational representatives}
Although it is always possible to find normal representatives
of local $\bmu_n$-classes, it may \emph{not} be possible to find representatives
whose total spaces are rational double point singularities:

\begin{Proposition} \label{prop: non-rational rep}
 For both RDPs of type $E_7^1$ and $E_7^0$ in characteristic $p=2$,
 the unique non-trivial class
 in $\overline{\Hfl{1}(U,\bmu_2)} $ can be represented by elliptic
 Gorenstein singularities, but not by RDPs.
\end{Proposition}

\begin{proof}
We have constructed elliptic Gorenstein singularities as total spaces
of local $\bmu_2$-torsors over the $E_7^1$ and $E_7^0$-singularities
in characteristic $2$ in the proof of Theorem \ref{thm: (no) normal representative}.
On the other hand, by \cite[Table 3]{Matsumotoinseparable}, there exist no $\bmu_2$-actions
on a rational double point singularity in characteristic $2$ that are free outside the closed point and with quotient of type $E_7^1$ and $E_7^0$.
\end{proof}

\begin{Remark}
We have already seen in the proof of Lemma \ref{lem: E81main}
(see also Remark \ref{rem: E81remark}) that if $x\in X$ is the RDP 
of type $E_8^1$ in characteristic $2$, then the primitive classes in $\overline{\Hfl{1}(U,\balpha_2)} $ can be represented by normal and non-rational
Gorenstein singularities, but not by RDPs.
What makes the previous proposition remarkable is that this phenomenon
also shows up for local $\bmu_2$-torsors, which should be better behaved
than local $\balpha_2$-torsors as $\bmu_2$ is linearly reductive.
\end{Remark}

\subsection{Local torsors as normalisations of global torsors}
Another weird phenomenon, which can only occur because the total space of a torsor under an infinitesimal group scheme over a smooth scheme is not necessarily smooth, is the following:
the example of an RDP $x\in X$ and a global $\balpha_p$-torsor over $X$ with non-normal total space and whose normalisation is a local $\balpha_p$-torsor that no longer extends to $X$.

\begin{Example} \label{ex: non-normal}
  Let $x\in X$ be the RDP of type $E_6^0$ in characteristic $p=3$.
  Then, there exists an $\balpha_3$-torsor $Y\to X$,
  such that $Y$ is non-normal, and such that the composition $Y^{\nu} \to Y \to X$ 
  has the structure of a local $\balpha_3$-torsor of primitive class, where $Y^\nu$ 
  denotes the normalisation of $Y$.
\end{Example}

\begin{proof}
Let $R:= k [[x,y,z]]/(z^2+x^3-y^4)$ be the RDP of type $E_6^0$ in $p=3$.
Set $S:= k [[u,v]]$ and consider the embedding $R\to S$ given by
$x\mapsto u^4-v^2$, $y\mapsto u^3$, $z\mapsto v^3$.
This map exhibits the $E_6^0$-singularity $x\in X$ as a quotient of $S$ by the
additive vector field $D=v\frac{\partial}{\partial u}-u^3\frac{\partial}{\partial v}$.
In particular, this gives an explicit description of $x\in X$ as $\balpha_3$-quotient
singularity.

Next, consider the embedding $R\to R[\sqrt[3]{y}]=:T$. 
In particular, $\Spec T\to\Spec R$ carries the structure of an $\balpha_3$-torsor.
We define an embedding of $R$-algebras $T\to S$ via $\sqrt[3]{y}\mapsto u$.
This way, $T$ becomes a subring of $S$ and both have the same field of fractions.
Now, $T$ is not normal and since $S$ is normal and 
integral over $T$, it follows that $S$ is the normalisation of $T$.
\end{proof}

\begin{Remark}
We remark that the morphism $Y^{\nu} \to Y$ in Example \ref{ex: non-normal} is 
\emph{not} $\balpha_3$-equivariant.

In fact, let $x\in X$ be a singularity, let 
$Y \to X$ be a local $G$-torsor for some finite group scheme $G$, 
and let $Z \to Y$ be any morphism such that the composition
$Z \to Y \to X$ admits the structure of a local $G$-torsor. 
If $Z \to Y$ is equivariant, then it restricts to an equivariant morphism of $G$-torsors 
over $U$, hence to an isomorphism over $U$ and then, by the universal property of 
integral closures, this shows that $Z \to Y$ is an isomorphism. 
In particular, $Y$ is normal if and only if $Z$ is and thus, in the setting of 
Example \ref{ex: non-normal}, it is impossible to choose the two $\balpha_3$-actions 
compatibly.
\end{Remark}

\subsection{Infinite towers}
The local \'etale fundamental group of a rational double point is finite
and the local \'etale covers lead to a ``simpler'' rational double point that covers the given one.
In characteristic zero, the universal local \'etale cover is actually smooth.
When studying local $G$-torsors under finite and infinitesimal group schemes, by Proposition \ref{prop: non-rational rep}, the total space of the local $G$-torsor can have a singularity that is more complex than the one we started with. 
In the following, we will note that the singularity might not change at all, that is, there are RDPs that are local torsors over themselves.
The following proposition, which was proven by the third named author in \cite[Table 3]{Matsumotoinseparable}, gives a classification of such RDPs in the case where $G \in \{\balpha_p,\bmu_p\}$.

\begin{Proposition}\label{prop: infinitetowers}
Let $x \in X$ be a rational double point and $G \in \{\balpha_p,\bmu_p\}$. 
Then, $X$ admits the structure of a local $G$-torsor over itself if and only if one of the following holds:
\begin{enumerate}
    \item $p = 3$, $G = \bmu_3$, and $X$ is of type $E_6^1$.
    \item $p = 3$, $G = \balpha_3$, and $X$ is of type $E_8^0$.
    \item $p = 2$, $G = \bmu_2$, and $X$ is of type $D_{4m}^{m}$.
    \item $p = 2$, $G = \bmu_2$, and $X$ is of type $E_7^3$.
    \item $p = 2$, $G = \balpha_2$, and $X$ is of type $D_{4m+2}^0$.
    \item $p = 2$, $G = \balpha_2$, and $X$ is of type $D_{8m}^{2m}$.
\end{enumerate}
\end{Proposition}

In particular, these examples give rise to infinite towers of local $G$-torsors
$$
   ...\,\to\,X_{i+1} \,\to\, X_i \,\to\, ...
$$
such that each $X_i$ is a rational double point of the same type. 
We note that this phenomenon can even happen for rational double points that are quotient singularities 
by Table \ref{table: Quotient- and non-quotient RDPs}.

\subsection{Deformations of (non-)quotient singularities}
Deformations of quotient singularities have been studied by
Schlessinger \cite{Schlessinger} and Riemenschneider \cite{Riemenschneider}.
Over the complex numbers, quotient singularities are rigid in dimension $\geq3$ and in dimension $2$, deformations
of quotient singularities are again quotient singularities by
a theorem of Esnault and Viehweg 
\cite{EsnaultViehweg2} (and former conjecture of Riemenschneider).
Similar results hold in characteristic $p>0$ for quotient singularities
by linearly reductive group schemes, that is, lrq singularities,
see \cite{LRQ}.

For quotient singularities by finite group schemes in the sense
of Definition \ref{def: quotient singularity} that are not
linearly reductive, these results no longer hold true.
For example, Schlessinger's rigidity in dimension $d\geq3$ 
result may fail, as we already 
observed in \cite[Section 9.2]{LRQ}.

\begin{Example}
Let $x\in X$ be an RDP of type $E_8^r$ in characteristic $p=2$.
For $r\leq s\leq 4$, there exists an equi-characteristic
deformation of $x\in X$ to an RDP of type $E_8^s$. 
Moreover, there are deformations of $x \in X$ to an RDP of type $A_1$, as well as
smoothing deformations. Inspecting Table \ref{table: Quotient- and non-quotient RDPs},
we see that this gives examples of deformations of RDPs, such that
\begin{enumerate}
    \item a quotient singularity deforms to a non-quotient
    singularity and vice versa.
    This is in contrast to characteristic zero \cite{EsnaultViehweg2}
    or to lrq singularities \cite{LRQ, SatoTakagi},
    where a quotient singularity always deforms to quotient singularity
    (but maybe with respect to a different group (scheme)).
    \item Even if a quotient singularity deforms to a quotient
    singularity, then upper semicontinuity of the lengths
    of the group schemes may not hold.
    We refer to \cite[Section 12]{LRQ} for a
    discussion of this semicontinuity and
    especially \cite[Proposition 12.11]{LRQ},
    where we establish it for RDPs.
    \item Even if a quotient singularity by a group scheme $G$ deforms to a quotient singularity 
    by a group scheme $G'$ of the same length as $G$, then $G'$ need not be a deformation of $G$.
\end{enumerate}
\end{Example}

Using Table \ref{table: Quotient- and non-quotient RDPs}, it is not difficult to 
see that some of these phenomena also show up for RDPs 
in characteristic $p=3$ and $p=5$, but we leave this to the reader.

\begin{Remark}
The results of this section show that local torsors over RDPs can display rather
pathological behaviour, which cannot occur for F-regular RDPs or in characteristic $0$.
Of course, non-F-regular RDPs exist in characteristic $2\leq p\leq 5$ only.
However, it is very likely that these phenomena will also appear in arbitrary positive
characteristic, when considering canonical singularities 
of arbitrary dimension.
\end{Remark}

\bibliographystyle{alpha}
\bibliography{Bibliography}

\end{document}